\documentclass[12pt]{amsart}

\usepackage{a4wide, amssymb, array, multirow, tikz, threeparttable}
\usetikzlibrary{arrows,positioning,calc}

\newcolumntype{V}{m{.40\linewidth}}
\newcolumntype{W}{m{.60\linewidth}}
\newcommand{\upperlinespace}{1.5ex}
\newcommand{\nodedistance}{0.6}
\newcommand{\circleinnersep}{1.6}
\renewcommand{\arraystretch}{1.2}
\newcommand{\bulletspacest}{.35\linewidth}
\newcommand{\bulletspacend}{.4\linewidth}
\newcommand{\bulletspacerd}{.2\linewidth}

\newtheorem{theorem}{Theorem}[section]
\newtheorem{lemma}[theorem]{Lemma}
\newtheorem{proposition}[theorem]{Proposition}
\newtheorem{corollary}[theorem]{Corollary}
\newtheorem{maintheorem}{Main Theorem}
\newtheorem*{theorem*}{Theorem}
\theoremstyle{definition}
\newtheorem{remark}[theorem]{Remark}

\newcommand{\C}{\ensuremath{\mathbb{C}}}
\newcommand{\R}{\ensuremath{\mathbb{R}}}
\renewcommand{\H}{\ensuremath{\mathbb{H}}}
\renewcommand{\mod}{\ensuremath{\mathrm{mod\,}}}
\newcommand{\Gr}{\ensuremath{\mathrm{Gr}}}
\newcommand{\g}[1]{\ensuremath{\mathfrak{#1}}}
\newcommand{\cal}[1]{\ensuremath{\mathcal{#1}}}
\newcommand{\wt}[1]{\ensuremath{\widetilde{#1}}}
\DeclareMathOperator{\tr}{tr}

\DeclareMathOperator{\codim}{codim}
\DeclareMathOperator{\Ad}{Ad}

\DeclareMathOperator{\ad}{ad}
\DeclareMathOperator{\Id}{Id}
\DeclareMathOperator{\rank}{rank}
\renewcommand{\O}{\ensuremath{\mathrm{O}}}
\newcommand{\SO}{\ensuremath{\mathrm{SO}}}
\newcommand{\SU}{\ensuremath{\mathrm{SU}}}
\newcommand{\Sp}{\ensuremath{\mathrm{Sp}}}
\newcommand{\Spin}{\ensuremath{\mathrm{Spin}}}
\newcommand{\Pin}{\ensuremath{\mathrm{Pin}}}
\newcommand{\Cl}{\ensuremath{\mathrm{Cl}}}
\newcommand{\U}{\ensuremath{\mathrm{U}}}

\newcommand{\Aut}{\ensuremath{\mathrm{Aut}}}
\newcommand{\Out}{\ensuremath{\mathrm{Out}}}
\newcommand{\End}{\ensuremath{\mathrm{End}}}

\newcommand{\h}{\ensuremath{h}}
\newcommand{\w}{\ensuremath{\omega}}
\newcommand{\e}{\ensuremath{e}}

\begin{document}
\title{Isoparametric foliations on complex projective spaces}

\author[M. Dom\'{\i}nguez-V\'{a}zquez]{Miguel Dom\'{\i}nguez-V\'{a}zquez}
\address{Instituto de Matem\'atica Pura e Aplicada (IMPA), Brazil.}
\email{mvazquez@impa.br}

\thanks{The author has been supported by the FPU programme of the Spanish Government, by a Marie-Curie European Reintegration Grant (PERG04-GA-2008-239162), and projects MTM2009-07756 and INCITE09207151PR (Spain).}



\begin{abstract}
Irreducible isoparametric foliations of arbitrary codimension $q$ on complex projective spaces $\C P^n$ are classified, for $(q,n)\neq (1,15)$. 
Remarkably, there are noncongruent examples that pull back under the Hopf map to congruent foliations on the sphere. Moreover, there exist many inhomogeneous isoparametric foliations, even of higher codimension. In fact, every irreducible isoparametric foliation on $\C P^n$ is homogeneous if and only if $n+1$ is prime.

The main tool developed in this work is a method to study singular Riemannian foliations with closed leaves on complex projective spaces. This method is based on certain graph that generalizes extended Vogan diagrams of inner symmetric spaces.
\end{abstract}


\maketitle


\section{Introduction}

Since its beginnings with the works of Somigliana, Levi-Civita, Segre, and Cartan, the theory of isoparametric foliations has been a fruitful area of research in Differential Geometry; see \cite{Th10} (and references therein) for a survey. Initially, only isoparametric hypersurfaces in real space forms were studied. Recall that a hypersurface is called isoparametric if its nearby equidistant hypersurfaces have constant mean curvature. Isoparametric hypersurfaces in real space forms have many remarkable properties. For example, they are precisely the hypersurfaces with constant principal curvatures, and every locally defined isoparametric hypersurface can be extended to a complete foliation of the ambient space. Here and henceforth, a foliation will be a singular Riemannian foliation as defined in \S\ref{sec:foliations} (cf.~\cite{Th10}).

Isoparametric hypersurfaces in Euclidean and real hyperbolic spaces were classified by Segre and Cartan, respectively. The corresponding foliations turn out to be homogeneous, that is, orbit foliations of isometric actions on the ambient space. However, this is not true in the case of spheres. Examples of inhomogeneous isoparametric foliations of codimension one were constructed by Ozeki and Takeuchi and then generalized by Ferus, Karcher, and M\"unzner using Clifford modules \cite{FKM81}; we call these examples the FKM-foliations. Hence, the problem in spheres is very involved and interesting, see~\cite{Ya93}. Over the last few years, there have been several major advances towards a final classification, which seems not to be very far. Among other investigations by several people, the results of Stolz \cite{St99}, Cecil, Chi, Jensen \cite{CCJ07}, Immervoll \cite{Im08}, Chi \cite{Ch11a}, \cite{Ch11b}, and Miyaoka \cite{Mi09}, imply that every isoparametric foliation of codimension one on a sphere is the orbit foliation of the isotropy representation of a semisimple symmetric space of rank two, or it is an FKM-foliation, or its isoparametric hypersurfaces satisfy $(g,m_1,m_2)=(4,7,8)$. Here $g$ is the number of principal curvatures of such a hypersurface, and $m_1=m_3$, $m_2=m_4$ their multiplicities.

The general theory of isoparametric foliations of arbitrary codimension in real space forms was developed by Terng \cite{Te85}. She defined a submanifold to be isoparametric if its normal bundle is flat and if it has constant principal curvatures in the direction of any parallel normal field. In real space forms, every isoparametric submanifold extends to a global isoparametric foliation. Indeed, these foliations can be seen as the level sets of the so-called isoparametric maps (see \cite{Te85}). Moreover, a foliation on a space form is isoparametric if and only if it is polar, i.e.\ if there is a complete totally geodesic submanifold $\Sigma_p$ through every regular point $p$ intersecting all leaves orthogonally ($\Sigma_p$ is then called a section). 

The classification problem of isoparametric submanifolds in Euclidean and real hyperbolic spaces has been reduced to the problem in spheres \cite{Te85}, \cite{Wu92}. Although the classification in spheres is still open for hypersurfaces, it has been completed for higher codimension, in which case all examples are homogeneous. More precisely, every irreducible isoparametric foliation of higher codimension on a sphere is the orbit foliation of an $s$-representation (that is, of the isotropy representation of a semisimple symmetric space). This remarkable result is due to Thorbergsson \cite{Th91} and will be important in our work.

The attempts to generalize isoparametric foliations to ambient spaces of nonconstant curvature have led to several different but related concepts, such as the notion of polar foliation introduced by Alexandrino \cite{Al04}, or the notion, introduced by Terng and Thorbergsson \cite{TT95}, of equifocal submanifolds of compact symmetric spaces (i.e.\ closed submanifolds with globally flat and abelian normal bundle, and whose focal directions and distances are invariant under parallel translation in the normal bundle). 
The study of these concepts has motivated many interesting ideas by several authors. For example, the combination of results by Christ~\cite{Ch02} and Lytchak~\cite{Ly12} imply the homogeneity of every irreducible polar foliation of codimension at least three on simply connected, irreducible, compact symmetric spaces of rank greater than one. Recently, Ge and Tang~\cite{GT13} have investigated isoparametric foliations of codimension one on more general manifolds, such as exotic spheres.

In our work, the definition of isoparametric submanifold that we will consider is the one due to Heintze, Liu, and Olmos \cite{HLO06}, which extends the notions given above of isoparametric hypersurface in any Riemannian manifold and of isoparametric submanifold of a real space form. Hence, we will say that a submanifold $M$ of a Riemannian manifold is an \emph{isoparametric submanifold} if the following properties are satisfied:
\begin{itemize}
\item[(a)] The normal bundle $\nu M$ is flat.
\item[(b)] Every parallel submanifold $M'$ of $M$ has constant mean curvature with respect to every parallel normal vector field of $M'$.
\item[(c)] $M$ admits sections, i.e.\ for each $p\in M$ there exists a totally geodesic submanifold $\Sigma_p$ that meets $M$ at $p$ orthogonally and whose dimension is the codimension of $M$.
\end{itemize}
The locally defined parallel submanifolds of an isoparametric submanifold are isoparametric as well, and thus define locally an \emph{isoparametric foliation}. 

The purpose of this work is to address the classification of isoparametric submanifolds of arbitrary codimension on complex projective spaces $\C P^n$.
It is important to mention that, as for real space forms, every isoparametric submanifold in a complex projective space extends to a globally defined isoparametric foliation (see Remark~\ref{rem:global_foliation}).

Before stating the main results of our work, it is convenient to settle some terminology. Let $\mathcal{F}$ be a foliation of the unit sphere $S^{2n+1}$ of $\R^{2n+2}$, and let $J$ be a complex structure on $\R^{2n+2}$ (in this work, by complex structure we mean an orthogonal, skew-symmetric transformation). We will say that $J$ preserves $\mathcal{F}$ if $\mathcal{F}$ is the pullback of a foliation on $\C P^n$ under the Hopf map $S^{2n+1}\to\C P^n$ determined by $J$ or, equivalently, if the leaves of $\mathcal{F}$ are foliated by the Hopf circles determined by $J$. We will say that a foliation $\mathcal{G}$ of a complex projective space $\C P^n$ is irreducible if there is no totally geodesic $\C P^k$, $k\in\{0,\ldots,n-1\}$, which is foliated by leaves of $\mathcal{G}$. Let $(G,K)$ be a semisimple symmetric pair. We will denote by $\mathcal{F}_{G/K}$ the orbit foliation of the isotropy representation of $G/K$ restricted to the unit sphere of the tangent space $T_{eK}(G/K)$. (Here $e$ denotes the identity element of $G$.) 

Our work was initially motivated by the following observation of Xiao \cite{Xi00t}. He noticed that, if $G/K$ is the real Grassmann manifold $\Gr_2(\R^{n+3})=\SO(n+3)/\mathrm{S}(\mathrm{O}(2)\times \mathrm{O}(n+1))$ with odd $n$, one can find two complex structures $J_1$ and $J_2$ on $T_{eK}(G/K)$ that preserve $\mathcal{F}_{G/K}$, and such that the projections of any fixed regular leaf of $\mathcal{F}_{G/K}$ via the corresponding Hopf maps $\pi_1,\pi_2\colon S^{2n+1}\subset T_{eK}(G/K)\to\C P^n$ yield two noncongruent isoparametric hypersurfaces of $\C P^n$, one of which is homogeneous while the other one is not.

Therefore, it seems natural to address the following problem: given an isoparametric foliation $\cal{F}$ on the sphere $S^{2n+1}$, find the set $\mathcal{J}_\cal{F}$ of complex structures on $\R^{2n+2}$ that preserve $\mathcal{F}$, determine the quotient set $\mathcal{J}_\cal{F}/\!\sim$, where $\sim$ stands for the equivalence relation ``give rise to congruent foliations on $\C P^n$", and finally decide which elements of $\mathcal{J}_\cal{F}/\!\sim$ provide homogeneous foliations on the corresponding $\C P^n$. Note that determining $\mathcal{J}_\cal{F}/\!\sim$ is then equivalent to classifying (up to congruence in $\C P^n$) those foliations on $\C P^n$ that pull back under the Hopf map to a foliation congruent to $\mathcal{F}$. We will write $N(\cal{F})$ for the cardinality of $\mathcal{J}_\cal{F}/\!\sim$. 

On the one hand, we will carry out this investigation for all isoparametric foliations $\cal{F}_{G/K}$ arisen from $s$-representations, thus obtaining the following result. Recall that $G/K$ is called inner if $\rank G=\rank K$.

\begin{maintheorem}
Let $G/K$ be an irreducible inner compact symmetric space of rank greater than one and set $n=\frac{1}{2}\dim G/K-1$. Then, up to congruence in $\C P^n$, there are exactly $N(\cal{F}_{G/K})\geq 1$ isoparametric foliations on $\C P^n$ whose pullback under the Hopf map gives a foliation congruent to $\mathcal{F}_{G/K}$, where:
\begin{itemize}
\item $N(\cal{F}_{G/K})=1+\left[\frac{\nu}{2}\right]+\left[\frac{p -\nu+1}{2}\right]$,\; if $G/K=\Gr_\nu(\C^{p+1})$ with $2\nu\neq p+1$,
\item $N(\cal{F}_{G/K})=1+\left[\frac{\nu}{2}\right]$,\; if $G/K=\Gr_\nu(\C^{p+1})$ with $2\nu= p+1$,
\item $N(\cal{F}_{G/K})=2$,\; if $G/K=\Gr_\nu(\H^{p})$ with $2\nu\neq p$, or if $G/K=\Gr_{2\nu}(\R^{2p})$ with $2\nu\neq p$, or if $G/K\in\{\mathbf{D\;III}, \mathbf{E \;II}, \mathbf{E \;III}, \mathbf{E \;VI}\}$,
\item $N(\cal{F}_{G/K})=1$,\; otherwise.
\end{itemize}

Moreover, if $G/K$ is Hermitian, exactly one of those $N(\cal{F}_{G/K})$ foliations is homogeneous. If $G/K$ is not Hermitian, all $N(\cal{F}_{G/K})$ foliations are inhomogeneous.

Conversely, if $\mathcal{G}$ is an irreducible isoparametric foliation of codimension greater than one on $\C P^n$, then there exist an irreducible inner compact symmetric space $G/K$ of dimension $2n+2$ and a complex structure $J$ on $T_{eK}(G/K)$ preserving $\mathcal{F}_{G/K}$ such that $\mathcal{G}$ is the projection of $\mathcal{F}_{G/K}$ by the Hopf map associated to $J$.
\end{maintheorem}

On the other hand, we investigate FKM-foliations satisfying $m_1\leq m_2$. To understand this condition and Main Theorem 2 below, let us briefly recall some known facts; see \S\ref{subsec:auto_FKM} for more details. Given a symmetric Clifford system $(P_0,\dots,P_m)$ on $\R^{2n+2}$, the corresponding FKM-foliation $\cal{F}_\cal{P}$ depends only on the $(m+1)$-dimensional vector space of symmetric matrices $\cal{P}=\mathrm{span}\{P_0,\dots,P_m\}$. The hypersurfaces of $\cal{F}_\cal{P}$ have $g=4$ principal curvatures with multiplicities $(m_1,m_2)=(m,n-m)$. Let $\Cl^*_{m+1}$ be the Clifford algebra of $\R^{m+1}$ with positive definite quadratic form. There is one equivalence class $\g{d}$ of irreducible $\Cl^*_{m+1}$-modules if $m\not\equiv 0\,(\mod 4)$, and two equivalence classes $\g{d}_+$, $\g{d}_-$ if $m\equiv 0\,(\mod 4)$.
Then $\cal{P}$ determines a representation of $\Cl^*_{m+1}$ on $\R^{2n+2}$, which is then equivalent to $\oplus_{i=1}^k\g{d}$ for some $k$ if $m\not\equiv 0\,(\mod 4)$, or to $(\oplus_{i=1}^{k_+}\g{d}_+)\oplus(\oplus_{i=1}^{k_-}\g{d}_-)$ for some $k_+,k_-$ if $m\equiv 0\,(\mod 4)$.

The condition $m_1\leq m_2$ always holds, except for $8$ FKM-examples. However, some of these exceptions are homogeneous or congruent to other FKM-foliations, so that only two examples remain unsettled, namely: both FKM-foliations with $(m_1,m_2)=(8,7)$. Intriguingly, such examples belong to the only open case in the classification of isoparametric hypersurfaces in spheres. Now, combining these results with Thorbergsson's theorem, our work classifies all irreducible isoparametric foliations of arbitrary codimension $q$ on $\C P^n$, except if $n=15$ and $q=1$. More explicitly, we have:

\begin{maintheorem}
Let $\cal{F}_\cal{P}$ be an FKM-foliation on $S^{2n+1}$ with $\dim \cal{P}=m+1$. Assume that $m_1\leq m_2$. Then, up to congruence in $\C P^n$, there are exactly $N(\cal{F}_\cal{P})\geq 1$ isoparametric foliations on $\C P^n$ that pull back under the Hopf map to a foliation congruent to $\cal{F}_\cal{P}$, where: 
\begin{itemize}
\item $N(\cal{F}_{\cal{P}})=2$,\; if $m\equiv 0\,(\mod 8)$ with $k_+$ and $k_-$ even, or if $m\equiv 1,7\,(\mod 8)$ with $k$ even, or if $m\equiv 3,4,5\,(\mod 8)$,
\item $N(\cal{F}_{\cal{P}})=2+\left[\frac{k}{2}\right]$,\; if $m\equiv 2,6\,(\mod 8)$,
\item $N(\cal{F}_{\cal{P}})=1$,\; otherwise.
\end{itemize}

Conversely, if $\mathcal{G}$ is an isoparametric foliation of codimension one on $\C P^n$, then there is a foliation $\cal{F}$ on $S^{2n+1}$ and a complex structure $J$ on $\R^{2n+2}$ preserving $\cal{F}$ such that $\cal{G}$ is the projection of $\cal{F}$ by the Hopf map associated to $J$, where
\begin{itemize}
\item $\cal{F}=\cal{F}_\cal{P}$ is an FKM-foliation satisfying $m_1\leq m_2$, or
\item $\cal{F}=\cal{F}_{G/K}$ for some inner compact symmetric space $G/K$ of rank $2$, or
\item $\cal{F}$ is an inhomogeneous isoparametric foliation of codimension one on $S^{31}$ whose hypersurfaces have $g=4$ principal curvatures with multiplicities $(7,8)$.
\end{itemize}
\end{maintheorem}

An important consequence of Main Theorem 1 is the existence of \emph{irreducible inhomogeneous isoparametric foliations of higher codimension} on complex projective spaces. This constrasts with the situation in higher rank symmetric spaces and also shows the impossibility of extending Thorbergsson's homogeneity theorem to complex projective spaces. Note that a foliation on $\C P^n$ is polar if and only if it is isoparametric (cf.~Proposition~\ref{prop:HopfMap} and \cite[Prop.~9.1]{Ly12}).  
Our work also generalizes known results on the existence of inhomogeneous examples of codimension one on $\C P^n$ (see \cite{Wa82}, \cite{Xi00t}, \cite{GTY11}). As shown by Ge, Tang, and Yan \cite{GTY11}, these examples exist if and only if $n\geq 3$ is odd (cf.~Theorem~\ref{th:codimension}(i)).

Several ingredients are fundamental in the proof of the Main Theorems. The classification results of isoparametric foliations on spheres and the nice behaviour of isoparametric submanifolds with respect to the Hopf map constitute the starting point of our arguments. However, the main tool we develop is certain general theory for the study of foliations with closed leaves on complex projective spaces. This is based, on the one hand, on the consideration of the automorphism group of foliations $\cal{F}\subset S^{2n+1}$, i.e.\ the group of orthogonal transformations of $\R^{2n+2}$ that map leaves of $\cal{F}$ to leaves of $\cal{F}$. This motivates the calculation of this group for homogeneous polar foliations on Euclidean spaces and for FKM-foliations satisfying $m_1\leq m_2$. On the other hand, our method requires the study of the symmetries of certain graph (the \emph{lowest weight diagram}) that we associate with $\cal{F}$. If $G/K$ is inner and $\cal{F}=\cal{F}_{G/K}$, such a diagram amounts to the \emph{extended Vogan diagram} of $G/K$. Finally, a subtle improvement of a result of Podest\`a and Thorbergsson \cite{PT99} gives us a criterion to decide when an isoparametric foliation on $\C P^n$ is homogeneous, from where we obtain some nice consequences, for example:
\begin{theorem*}
Every irreducible isoparametric foliation on $\C P^n$ is homogeneous if and only if $n+1$ is a prime number.
\end{theorem*}

During the writing of this paper, I first encountered another article of Xiao \cite{Xi00a}, where he claims to obtain the classification of isoparametric submanifolds in $\C P^n$. However, the arguments and classification in \cite{Xi00a} seem to have several crucial gaps. Firstly, the author uses the maximality property for $s$-representations (see \S\ref{subsec:auto_homogeneous}) without actually referring to it. Secondly, the study of the inner symmetric spaces \textbf{E V} and \textbf{E VIII} is missing there. But more importantly, although he mentions that there are pairs of noncongruent isoparametric submanifolds in $\C P^n$ with congruent inverse images, surprisingly this is not reflected in his classification, since for each inner symmetric space $G/K$ considered, only one complex structure is specified. Therefore, the congruence problem (which is the main difficulty in our work) is completely disregarded, as well as the study of the homogeneity.

Our paper is organized as follows. In Section~\ref{sec:Hopf} we study the behaviour of isoparametric submanifolds with respect to the Hopf map. In Section \ref{sec:automorphisms} we find the group of automorphisms of homogeneous polar foliations on Euclidean spaces (\S\ref{subsec:auto_homogeneous}), and of FKM-foliations satisfying $m_1\leq m_2$ (\S\ref{subsec:auto_FKM}). In Section \ref{sec:foliations} we study general foliations with closed leaves on $\C P^n$, first characterizing the complex structures that preserve a given foliation on a sphere (\S\ref{subsec:preserving}), and then studying the congruence of the projected foliations (\S\ref{subsec:congruence}). We particularize this theory and obtain the corresponding classifications for
homogeneous polar foliations in Section~\ref{sec:homogeneous}, and for FKM-foliations in Section~\ref{sec:FKM}. Finally, in Section~\ref{sec:homogeneity} we study the homogeneity of the resulting isoparametric foliations on $\C P^n$.

\textbf{Acknowledgments.} This work was started during a stay at the University of Cologne in Spring 2011. I am deeply indebted to my host Prof.\ Gudlaugur Thorbergsson for many enlightening discussions, helpful ideas, and for his interest in this work. I would also like to thank my advisor Prof.\ Jos\'e Carlos D\'iaz-Ramos for his constant support.

\section{Isoparametric submanifolds and the Hopf map}\label{sec:Hopf}

In this section we study the behaviour of isoparametric submanifolds with respect to the Hopf map and comment on some related questions.

Let us first recall the construction of the complex projective space $\C P^n$.
Consider the Euclidean space $\R^{2n+2}$ with the usual scalar product $\langle \cdot,\cdot \rangle$, and take a complex structure $J$ on $\R^{2n+2}$. Then $J$ induces a principal fiber bundle with total space the unit sphere $S^{2n+1}$, with base space the complex projective space $\C P^n$, and with structural group $S^1$; the corresponding projection $\pi\colon S^{2n+1}\to\C P^n$ is called the Hopf map. For every point $x\in S^{2n+1}$, the vector $Jx\in T_x S^{2n+1}$ is vertical, i.e.\ tangent to the $S^1$-fibers of $\pi$, and these fibers are geodesics in the sphere. Consider the distribution $\mathcal{H}$ on $S^{2n+1}$ defined by the orthogonal subspaces to the fibers. Then the differential of $\pi$ induces a linear isomorphism of $\mathcal{H}_x$ onto $T_{\pi(x)}\C P^n$, for each $x\in S^{2n+1}$. The Fubini-Study metric on $\C P^n$ of constant holomorphic sectional curvature $4$ is defined by $\langle X,Y\rangle=\langle \widetilde{X},\widetilde{Y}\rangle$, where $X,Y\in T_{\pi(x)}\C P^n$ and $\widetilde{X}, \widetilde{Y}$ are their horizontal lifts at $x$. Moreover, the complex structure $J$ on $\R^{2n+2}$ leaves $\mathcal{H}$ invariant and induces via $\pi$ the canonical K\"ahler structure $J$ on $\C P^n$.

By construction, the Hopf map is a Riemannian submersion. Denote by $\widetilde{\nabla}$ and $\nabla$ the Levi-Civita connections of $S^{2n+1}$ and $\C P^n$, respectively. Then, for all tangent vector fields $X,Y$ on $\C P^n$ we have that
$\widetilde{\nabla}_{\widetilde{X}}\widetilde{Y}=\widetilde{\nabla_X Y}+O'N(\widetilde{X},\widetilde{Y})$.
Here $O'N$ is one of the tensors of O'Neill, which satisfies $O'N(\widetilde{X},\widetilde{Y})=(\widetilde{\nabla}_{\widetilde{X}}\widetilde{Y})^\mathcal{V}=\frac{1}{2}[\widetilde{X},\widetilde{Y}]^\mathcal{V}$, where $(\cdot)^\mathcal{V}$ denotes orthogonal projection onto the vertical space.

Heintze, Liu, and Olmos showed in \cite[Th.~3.4]{HLO06} that, if $\pi\colon E \to B$ is a Riemannian submersion with minimal fibers and $M \subset B$ an embedded submanifold, then $\widetilde{M}= \pi^{-1}M$ is isoparametric with horizontal sections if and only if $M$ is isoparametric and $O'N = 0$ on all horizontal lifts of tangent vectors to sections of $M$; moreover, in this situation, $\pi$ maps sections of $\widetilde{M}$ to sections of $M$. Using this result, we can show the following.

\begin{proposition}\label{prop:HopfMap}
Let $M$ be an embedded submanifold of positive dimension in $\C P^n$, and $\widetilde{M}=\pi^{-1} M$ its lift to $S^{2n+1}$. Then $M$ is isoparametric if and only if $\widetilde{M}$ is isoparametric.

In this situation, $\pi$ maps sections of $\widetilde{M}$ (which are horizontal) to sections of $M$ (which are totally real).
\end{proposition}
\begin{proof}
First notice that the fibers of the Hopf map are minimal (in fact, totally geodesic) and that, if $\widetilde{M}$ is isoparametric, it necessarily has horizontal sections (since $\widetilde{M}$ is union of $S^1$-fibers). Therefore, by the result in \cite{HLO06}, if $\widetilde{M}$ is isoparametric, then $M$ is isoparametric. 

Assume now that $M$ is isoparametric. Let $X,Y$ be arbitrary tangent vector fields to the sections of $M$. Denote by $\xi$ the outer unit normal vector field to $S^{2n+1}$, so $J\xi$ is a vertical vector field on $S^{2n+1}$. Let $D$ be the Levi-Civita connection of $\R^{2n+2}$. We have:
\begin{align*}
\langle O'N(\widetilde{X},\widetilde{Y}),J\xi\rangle & = \langle \widetilde{\nabla}_{\widetilde{X}} \widetilde{Y},J\xi\rangle = \langle D_{\widetilde{X}} \widetilde{Y}, J\xi \rangle = -\langle D_{\widetilde{X}} J\widetilde{Y},\xi\rangle=\langle J\widetilde{Y},\widetilde{X}\rangle=\langle JY,X\rangle,
\end{align*}
since $S^{2n+1}$ is a totally umbilical hypersurface in $\R^{2n+2}$, and $J\mathcal{H}=\mathcal{H}$. Hence, the proposition will follow from the result in \cite{HLO06} once we show that sections of $M$ are totally real.

It is known that any totally geodesic submanifold of $\C P^n$ must be either a totally real or a complex submanifold. By continuity, if one section of $M$ is complex, then all sections of $M$ are complex. 
However, if the sections of $M$ were complex, then $M$ would be a complex submanifold of $\C P^n$, and hence K\"ahler, but this is impossible because there are no K\"ahler submanifolds of positive dimension in $\C P^n$ with flat normal bundle (see, for example, \cite[Th.\ 19]{AD04}). Hence, all sections of $M$ are totally real and the result follows.
\end{proof}

Proposition \ref{prop:HopfMap} guarantees that every isoparametric submanifold in a complex projective space can be obtained by projecting some isoparametric submanifold in a sphere under the Hopf map. 

\begin{remark}\label{rem:global_foliation}
As well as for space forms, every isoparametric submanifold in $\C P^n$ can be extended to a global isoparametric foliation on $\C P^n$. Let us show this. Every isoparametric submanifold extends locally to an isoparametric foliation. By Proposition \ref{prop:HopfMap} the lift of this local foliation to an open set $U$ of $S^{2n+1}$ is again isoparametric. By \cite[Th.~3.4]{Te87} and \cite[Th.~D]{Te85}, a local isoparametric foliation of $S^{2n+1}$ can be extended to an isoparametric foliation $\mathcal{F}$ of the whole sphere in a unique way; moreover, this foliation is defined by the level sets of the restriction $F\rvert_{S^{2n+1}}$ of a polynomial function $F=(F_1,\ldots,F_k)\colon\R^{2n+2}\to\R^k$, where $k$ is the lowest codimension of the leaves, and the gradients $\nabla F_1,\ldots,\nabla F_k$ define $k$ global normal vector fields on every leaf; on each regular leaf these fields conform a basis of the normal space. Consider the analytic function $f\colon S^{2n+1}\to \R^k$, defined by $x\mapsto (\langle Jx,(\nabla F_1)_x\rangle,\ldots, \langle Jx,(\nabla F_k)_x\rangle)$.
Since $f$ is constantly equal to zero in $U$ (the leaves of this local foliation are foliated by Hopf fibers), by analiticity we get that $f=0$ identically on $S^{2n+1}$, and therefore, $\mathcal{F}$ can be projected to a global isoparametric foliation on $\C P^n$.
\end{remark}

Every isoparametric foliation on a sphere determines an isoparametric foliation on the whole Euclidean space via homotheties. Conversely, if the leaves of an isoparametric foliation on a Euclidean space are compact, then they are contained in concentric spheres. Moreover, an isoparametric foliation of codimension $k-1$ on a sphere is said to be irreducible if its associated Coxeter system of rank $k$ (in the sense of Terng \cite{Te85}) is irreducible, or equivalently, if there is no proper totally geodesic submanifold of the sphere being a union of leaves of the foliation. 
Similarly, we will say that an isoparametric foliation on a complex projective space $\C P^n$ is \emph{irreducible} if there is no proper totally geodesic complex projective subspace $\C P^k$, $k<n$, that is a union of leaves of the foliation. Hence, an isoparametric foliation on a complex projective space is irreducible if and only if its lift to the sphere $S^{2n+1}$ is an irreducible isoparametric foliation. This follows from the fact that the only totally geodesic submanifolds of $S^{2n+1}$ which are foliated by Hopf circles are intersections of $S^{2n+1}$ with complex subspaces of $\C^{n+1}$. 

According to Proposition~\ref{prop:HopfMap}, the problem of classifying irreducible isoparametric foliations on $\C P^n$ amounts to determining which irreducible isoparametric foliations on $S^{2n+1}$ are such that their leaves contain the $S^1$-fibers of the Hopf map. 
Our approach lies, therefore, on the classification of isoparametric foliations on spheres. For irreducible isoparametric foliations of codimension greater than one, Thorbergsson's theorem \cite{Th91} guarantees that they are exactly orbit foliations $\cal{F}_{G/K}$ of isotropy representations of irreducible semisimple symmetric spaces $G/K$. In codimension one, M\"unzner's result \cite[I]{Mu80} ensures that the number of principal curvatures of an isoparametric hypersurface in a sphere is $g\in\{1, 2, 3, 4, 6\}$, and the corresponding multiplicities $m_1,\ldots,m_g$ satisfy $m_i=m_{i+2}$ (indices modulo $g$). As already commented, the classification of isoparametric hypersurfaces in spheres has been completed, except if $(g,m_1,m_2)=(4,7,8)$. For more information on this problem, see \cite{CCJ07}, \cite{Im08}, \cite{Ch11b}, \cite{Mi09}, and references therein. 

These results imply that every irreducible isoparametric foliation on a sphere is an FKM-foliation or a homogeneous polar foliation, excluding the exceptional case of codimension one. We finish this section with a result that will be needed later.

\begin{proposition}\label{prop:622}
Let $M$ be an isoparametric hypersurface in a sphere with $(g,m_1,m_2)\in\{(4,2,2),(6,2,2)\}$. Then $M$ is not the pullback of a hypersurface in $\C P^n$ ($n=4,6$) under some Hopf map.
\end{proposition}
\begin{proof}
M\"unzner \cite[II]{Mu80} determined the cohomology rings $H^*(M,\mathbb{Z}_2)$ of isoparametric hypersurfaces $M$ in spheres. It follows from this result that $H^q(M,\mathbb{Z}_2)=0$ for all odd integers $q\in\{1,\ldots,\dim M\}$, and that $2g=\dim_{\mathbb{Z}_2} H^*(M,\mathbb{Z}_2)$. Therefore the Euler characteristic of $M$ is $\chi(M)=2g\neq 0$. This implies that $M$ is not foliated by Hopf circles: otherwise, the complex structure $J$ would determine a globally defined non-vanishing tangent vector field on $M$, which would imply $\chi(M)=0$ because of the Hopf index theorem.
\end{proof}

\section{The group of automorphisms of an isoparametric foliation}\label{sec:automorphisms}
Our aim in this section is to determine the whole (not necessarily connected) group of orthogonal transformations leaving invariant a given isoparametric foliation on a sphere. We will call these transformations the automorphisms of the foliation. We carry out this study for the case of orbit foliations of $s$-representations (or, equivalently, for homogeneous polar foliations) in \S\ref{subsec:auto_homogeneous}, and for the case of FKM-foliations satisfying $m_1\leq m_2$ in~\S\ref{subsec:auto_FKM}.

\subsection{The group of automorphisms of a homogeneous polar foliation}\label{subsec:auto_homogeneous}

Dadok \cite{Da85} classified homogeneous polar foliations (or equivalently, polar actions up to orbit equivalence) on Euclidean spaces. He proved that these foliations are orbit foliations of isotropy representations of (Riemannian) symmetric spaces.

Since any homogeneous polar foliation on a Euclidean space is the product of a homogeneous polar foliation with compact leaves times an affine subspace, we will just consider homogeneous polar foliations with compact leaves. This means that the symmetric space $G/K$ whose isotropy representation defines the foliation is semisimple. Moreover, since the duality between symmetric spaces of compact and noncompact type preserves their isotropy representations, we will assume that $G/K$ is of compact type.

Given a compact symmetric pair $(G,K)$, we will write the Cartan decomposition of the Lie algebra $\g{g}$ of $G$ as $\g{g}=\g{k}\oplus\g{p}$, where $\g{k}$ is the Lie algebra of the isotropy group $K$ and $\g{p}$ is the orthogonal complement of $\g{k}$ in $\g{g}$ with respect to the Killing form $B_\g{g}$ of $\g{g}$, which is negative definite. 
Moreover, $\g{p}$ is endowed with the metric $\langle\cdot,\cdot\rangle=-B_\g{g}\rvert_{\g{p}\times\g{p}}$. The $s$-representation of $(G, K)$ can be seen as the adjoint representation $K\to \O(\g{p})$, $k\mapsto \Ad(k)\rvert_\g{p}$. 

We will say that a symmetric pair $(G,K)$ satisfies the \emph{maximality property} if it is effective, $K$ is connected, and $\Ad(K)\rvert_\g{p}$ is the maximal connected subgroup of $\O(\g{p})$ acting on $\g{p}$ with the same orbits as the $s$-representation of $(G,K)$.

Let $(G,K)$ be an effective compact symmetric pair, with $K$ connected. Eschenburg and Heintze proved that, if $G/K$ is irreducible and of rank greater than two, then $(G,K)$ satisfies the maximality property \cite{EH99an}; the same holds if $G/K$ is irreducible and of rank two (this follows from \cite{Da85}; cf. \cite[p.~392, Remark~1]{EH99mz}). If $G/K$ is reducible, then $G/K$ satisfies the maximality property whenever all irreducible factors of rank equal to one and dimension $n$ are assumed to be spheres represented by the symmetric pair $(\SO(n+1),\SO(n))$ (see~\cite[p.~391]{EH99mz}). Therefore, for the study of geometric properties of the foliations induced by $s$-representations, it is not a restriction of generality to assume that the corresponding symmetric pairs satisfy the maximality property.

This property allows us to determine the whole group $\Aut(\cal{F}_{G/K})$ of automorphisms of the orbit foliation $\cal{F}_{G/K}$ of the isotropy representation of $G/K$. If $(G,K)$ satisfies the maximality property, then the identity connected component of $\Aut(\cal{F}_{G/K})$ is $\Ad(K)\rvert_\g{p}$, but $\Aut(\cal{F}_{G/K})$ might have several connected components.

\begin{theorem}\label{th:automPolar}
Let $(G,K)$ be a compact symmetric pair that satisfies the maximality property, and with Cartan decomposition $\g{g}=\g{k}\oplus\g{p}$. 

Then there is a Lie group isomorphism between the group $\Aut(\g{g},\g{k})$ of automorphisms of $\g{g}$ that restrict to automorphisms of $\g{k}$ and the group $\Aut(\cal{F}_{G/K})$ of orthogonal transformations of $\g{p}$ that map leaves of $\cal{F}_{G/K}$ to leaves of $\cal{F}_{G/K}$.
\end{theorem}
\begin{proof}
We will show that the restriction map $\Psi\colon\Aut(\g{g},\g{k})\to\Aut(\cal{F}_{G/K})$, $\varphi\mapsto\varphi\rvert_\g{p}$ yields the desired isomorphism.
Let $\varphi\in \Aut(\g{g},\g{k})$. Clearly $\varphi\rvert_\g{p}\in \O(\g{p})$. Since $\varphi$ preserves $\g{k}$, if we fix $P\in\g{p}$, we easily get that $\varphi(\Ad(K)P)=\Ad(K)\varphi(P)$.
Hence $\varphi\rvert_\g{p}$ sends the orbit through $P$ to the orbit through $\varphi(P)$, from where it follows that $\Psi$ is a well-defined Lie group homomorphism.
In order to show that $\Psi$ is one-to-one, let $\varphi\in \Aut(\g{g},\g{k})$ with $\varphi\rvert_\g{p}=\Id_\g{p}$ and take arbitrary elements $X\in\g{k}$ and $P\in\g{p}$. Then $[X,P]=\varphi[X,P]=[\varphi X, P]$, so $\ad(\varphi X- X)\rvert_\g{p}=0$. By the effectiveness of $(G,K)$, we have $\varphi X=X$, and hence $\varphi= \Id$.

It remains to prove that $\Psi$ is onto. Let $A\in \Aut(\cal{F}_{G/K})$. The maximality property implies that $A \Ad(K)\rvert_\g{p} A^{-1}=\Ad(K)\rvert_\g{p}$. Then the effectiveness of $(G,K)$ entails the existence of an automorphism $\phi_A$ of $\g{k}$ defined by $\ad(\phi_A(X))\rvert_\g{p}=A \ad(X)\rvert_\g{p}A^{-1}$, for all $X\in \g{k}$.

Now for each $A\in\Aut(\cal{F}_{G/K})$ we construct an automorphism $\varphi_A\in \Aut(\g{g},\g{k})$ whose restriction to $\g{p}$ is $A$. Define $\varphi_A$ as the linear endomorphism of $\g{g}=\g{k}\oplus\g{p}$ given by $\varphi_A\rvert_\g{k}=\phi_A$ and $\varphi_A\rvert_\g{p}=A$. Clearly, it is a linear isomorphism preserving the Cartan decomposition and with the desired restriction to $\g{p}$. We just have to see that it respects the Lie bracket.

Let $P_1, P_2\in \g{p}$ and $X\in \g{k}$ be arbitrary elements. Denote by $B_\g{k}$ and $B_\g{g}$ the Killing forms of $\g{k}$ and $\g{g}$, respectively. Then, using that $\varphi_A\rvert_\g{k}=\phi_A\in \Aut(\g{k})$, the definition of $\varphi_A$, and the invariance of the trace operator under conjugation, we get
\begin{align*}
B_\g{g}(\varphi_A X,\varphi_A[P_1,P_2]) &= B_\g{k}(\varphi_A X,\varphi_A[P_1,P_2]) + \tr_\g{p} (\ad(\varphi_A X)\ad(\varphi_A[P_1,P_2]))
\\
&=B_\g{k}(X,[P_1,P_2])+\tr_\g{p}(\ad (X)\ad([P_1,P_2]))=B_\g{g}(X,[P_1,P_2]).
\end{align*}
Using this property, the definition of $\varphi_A$, and the fact that $A\in \O(\g{p})$, now we have:
\begin{align*}
B_\g{g}(X,\varphi_A[P_1,P_2])\! &= B_\g{g}(\varphi_A \varphi_A^{-1} X,\varphi_A [P_1,P_2]) = B_\g{g}(\varphi_A^{-1}X,[P_1,P_2]) 
= B_\g{g}(\ad(\varphi_A^{-1}X)P_1,P_2)\\ 
\!&= B_\g{g}(A^{-1} \ad(X) A P_1,P_2) = B_\g{g}(\ad(X) A P_1, A P_2)
= B_\g{g}(X,[\varphi_A P_1,\varphi_A P_2]).
\end{align*}
Since $X\in\g{k}$ is arbitrary, $\varphi_A[P_1,P_2]\in\g{k}$, and $B_\g{g}$ is nondegenerate, we get that $\varphi_A[P_1,P_2]=[\varphi_A P_1,\varphi_A P_2]$, for every $P_1,P_2\in\g{p}$. Furthermore, using the previous properties, we obtain:
\begin{align*}
B_\g{g}(\varphi_A[P_1,X],P_2) &= B_\g{g}([P_1,X],\varphi_A^{-1} P_2)= B_\g{g}(X,[\varphi_A^{-1} P_2,P_1])
= B_\g{g}(\varphi_A X,\varphi_A[\varphi_A^{-1} P_2,P_1])
\\
&=B_\g{g}(\varphi_A X,[P_2,\varphi_A P_1])=B_\g{g}([\varphi_A P_1, \varphi_A X],P_2),
\end{align*}
from where $\varphi_A[P,X]=[\varphi_A P,\varphi_A X]$ for all $P\in\g{p}$ and all $X\in\g{k}$. We conclude that $\varphi_A\in\Aut(\g{g},\g{k})$. 
Therefore, $\Psi$ is onto and the proof is finished.
\end{proof}

By \cite[Ch.~VII, Prop.~4.1]{Lo69} the group $\Aut(\g{g},\g{k})$ is isomorphic to the isotropy group of the base point of $G/K$ in the whole isometry group of $G/K$, and also to the group $\Aut (\g{p})$ of automorphisms of $\g{p}$, that is, the group of linear isomorphisms $A$ of $\g{p}$ such that $A[P_1,[P_2,P_3]]=[AP_1,[AP_2,AP_3]]$, for all $P_1,P_2,P_3\in\g{p}$.

Let us conclude this subsection with the following observation.

\begin{remark} \label{rem:foliation_determines_space}
A homogeneous polar foliation defined by the $s$-representation of a compact symmetric pair satisfying the maximality property determines the corresponding orthogonal symmetric pair (up to a permutation of its irreducible factors). Let us give a quick argument for this claim. It is enough to check it for irreducible symmetric pairs. Let $(G,K)$ and $(G',K')$ be compact irreducible symmetric pairs satisfying the maximality property, and let $\g{g}=\g{k}\oplus\g{p}$ and $\g{g}'=\g{k}'\oplus\g{p}'$ be their Cartan decompositions. If the foliations $\cal{F}_{G/K}$ and $\cal{F}_{G'/K'}$ are congruent, there is an orthogonal map $A$ between $\g{p}$ and $\g{p}'$ that maps leaves of $\mathcal{F}_{G/K}$ to leaves of $\mathcal{F}_{G'/K'}$. Hence $A\Ad(K)\rvert_\g{p}A^{-1}$ acts on $\g{p}'$ with the same orbits as $\Ad(K')\rvert_{\g{p}'}$. The maximality property implies $A\Ad(K)\rvert_\g{p}A^{-1}=\Ad(K')\rvert_{\g{p}'}$, and hence $\g{k}$ and $\g{k'}$ must be isomorphic. Moreover, the ranks of $G/K$ and $G'/K'$ must be equal (so that the codimensions of both foliations agree), as well as the dimensions of both symmetric spaces (so that $\dim\g{p}=\dim\g{p}'$). But one can check (by direct inspection, see~\cite[p.~516--519]{He78}) that these invariants determine the compact irreducible orthogonal symmetric pair $(\g{g},\g{k})$.
\end{remark}

\subsection{The group of automorphisms of an FKM-foliation}\label{subsec:auto_FKM}
Our goal now is to determine the group of automorphisms of the isoparametric foliations in spheres constructed by Ferus, Karcher, and M\"unzner in \cite{FKM81}. We will do this for almost all such examples, with only some exceptions mentioned below. However, these exceptions can be reduced to only two, namely the inhomogeneous FKM-examples whose multiplicities are $(m_1,m_2)=(8,7)$.

Let us begin by reminding the reader about the construction of the FKM-foliations. For details missing here we refer to the original paper \cite{FKM81}. For more information on Clifford algebras and their representations, see \cite[Ch.~I]{LM89}.

Let $\Cl(\cal{E})=\Cl_{m+1}^*$ be the Clifford algebra associated to $\cal{E}=\R^{m+1}$ endowed with the standard positive definite quadratic form. Thus $\Cl(\cal{E})$ can be regarded as the algebra generated by an orthonormal basis $\{E_0,\dots,E_m\}$ of $\cal{E}$ (and the unit $1$) subject to the relations $E_iE_j+E_jE_i=2\delta_{ij} 1$ for every $i,j\in\{0,\dots,m\}$, where $\delta_{ij}$ is the Kronecker delta. Set $V=\R^{2n+2}$ and let $\chi\colon\Cl(\cal{E})\to \End(V)$ be a representation of the Clifford algebra $\Cl(\cal{E})$. Endow $V$ with a positive definite $\Pin(\cal{E})$-invariant inner product $\langle\cdot,\cdot\rangle$. Let us put $P_i=\chi(E_i)$ for $i=0,\dots, m$. Then $(P_0,\dots,P_m)$ is what in \cite{FKM81} is called a \emph{(symmetric) Clifford system}, i.e.\ an $(m+1)$-tuple of symmetric matrices on $V$ which satisfy $P_iP_j+P_jP_i=2\delta_{ij} \Id$ for all $i,j\in\{0,\dots, m\}$. We also define $\cal{P}=\mathrm{span}\{P_0,\dots,P_m\}$ and endow this vector space with the inner product induced by $\chi$, which turns out to be given by $\langle P,P'\rangle=(1/\dim V)\tr(PP')$, for $P,P'\in\cal{P}$.

Assume that $m_2=n-m>0$. Then the FKM-foliation $\cal{F}_\cal{P}$ associated to the Clifford system $(P_0,\dots,P_m)$ is defined by the level sets of $F\rvert_{S(V)}$, where $S(V)$ is the unit sphere of $V$ and $F\colon V\to\R$ is the Cartan-M\"unzner polynomial:
\[
F(x)=\langle x, x\rangle^2-2\sum_{i=0}^m\langle P_i x, x\rangle^2.
\]
The corresponding isoparametric hypersurfaces have $g=4$ principal curvatures with multiplicities $(m_1,m_2)=(m,n-m)$. 
This construction does not depend on the particular matrices $P_0,\dots, P_m$, but only on the unit sphere $S(\cal{P})$ of $\cal{P}$. $S(\cal{P})$ is called the Clifford sphere of the foliation. Moreover, two FKM-foliations are congruent if and only if their Clifford spheres are conjugate under an orthogonal transformation of $V$. 

The two focal submanifolds of an FKM-foliation are never congruent (see~\cite[I, p.~59]{Mu80} for the pair $(1,1)$). Since a focal submanifold determines the whole isoparametric foliation, it follows that every automorphism of $\cal{F}_\cal{P}$ maps each leaf onto itself. 
Hence $A\in \Aut(\cal{F}_\cal{P})$ if and only if $F(Ax)=F(x)$ for all $x\in V$, where $F$ is the Cartan-M\"unzner polynomial of $\cal{F}_\cal{P}$. This means that the Clifford systems $(P_0,\dots, P_m)$ and $(A^{-1}P_0 A,\dots, A^{-1}P_m A)$ define the same foliation. Therefore $A\in\Aut(\cal{F}_\cal{P})$ if and only if $\cal{F}_\cal{P}=\cal{F}_{A^{-1}\cal{P}A}$.

Let $\SO(\cal{P})\cup\O^-(\cal{P})$ and $\Spin(\cal{P})\cup\Pin^-(\cal{P})$ be the decompositions of the orthogonal group $\O(\cal{P})$ and of the pin group $\Pin(\cal{P})$ in connected components, respectively.
We define the following subsets of the orthogonal group of $V$:
\begin{align*}
\U^+(\cal{P})&=\{U\in\O(V):PU=UP\text{ for all }P\in\cal{P}\}, 
\\
\U^-(\cal{P})&=\{U\in\O(V):PU=-UP\text{ for all }P\in\cal{P}\},
\end{align*}
and $\U^\pm(\cal{P})=\U^+(\cal{P})\cup\U^-(\cal{P})$. The set $\U^-(\cal{P})$ might be empty. Elements in $\U^+(\cal{P})$ commute with those in $\Pin(\cal{P})$, while elements in $\U^-(\cal{P})$ commute with the ones in $\Spin(\cal{P})$ and anticommute those in $\Pin^-(\cal{P})$. 
Moreover, $\Pin(\cal{P})$ and $\U^\pm(\cal{P})$ are subgroups of $\Aut(\cal{F}_\cal{P})$.

An important remark for our work is that if $m_1\leq m_2$, then the FKM-foliation determines the Clifford sphere $S(\cal{P})$, or equivalently, the space $\cal{P}$ (see \cite[\S4.6]{FKM81}). This observation allows us to show the following structure result for $\Aut(\cal{F}_\cal{P})$.

\begin{theorem}\label{th:auto_FKM_1}
Let $\cal{F}_{\cal{P}}$ be an FKM-foliation satisfying $m_1\leq m_2$. We have:
\begin{itemize}
\item[(i)] If $m$ is odd, then $\Aut(\cal{F}_\cal{P})\cong \Pin(\cal{P})\cdot \U^+(\cal{P})$.
\item[(ii)] If $m$ is even, then $\Aut(\cal{F}_\cal{P})\cong \Spin(\cal{P})\cdot \U^\pm(\cal{P})$.
\end{itemize}
In both cases $\Aut(\cal{F}_\cal{P})$ is isomorphic to a direct product modulo the center $Z(\Spin(\cal{P}))$ of $\Spin(\cal{P})$, which is $\{\pm\Id,\pm P_0\cdots P_m\}\cong \mathbb{Z}_4$ if $m\equiv 1\,(\mod 4)$; $\{\pm\Id,\pm P_0\cdots P_m\}\cong \mathbb{Z}_2\times\mathbb{Z}_2$ if $m\equiv 3\,(\mod 4)$; and $\{\pm\Id\}\cong\mathbb{Z}_2$ if $m$ is even.
\end{theorem}
\begin{proof}
We saw that $A\in \Aut(\cal{F}_\cal{P})$ if and only if $\cal{F}_\cal{P}=\cal{F}_{A^{-1}\cal{P}A}$. If $m_1\leq m_2$, the previous condition is equivalent to $\cal{P}=A\cal{P}A^{-1}$. 

Assume that $m$ is odd. Under this assumption, the adjoint representation $\Ad\colon \Pin(\cal{P})\to\O(\cal{P})$ is onto and $\Ad(\Pin^-(\cal{P}))=\O^-(\cal{P})$. Consider the group homomorphism $\Psi\colon \Pin(\cal{P})\times \U^+(\cal{P})\to\Aut(\cal{F}_\cal{P})$, $(Q,U)\mapsto QU$. Its kernel is isomorphic to $\Pin(\cal{P})\cap\U^+(\cal{P})$, which is exactly $Z(\Spin(\cal{P}))$ (note that $\Pin^-(\cal{P})\cap\U^+(\cal{P})=\emptyset$). We show that $\Psi$ is onto. Let $A\in\Aut(\cal{F}_\cal{P})$. Then $\varphi_A\colon \cal{P}\to\cal{P}$, $P\mapsto APA^{-1}$, is an orthogonal transformation of $\cal{P}$. Since the  adjoint representation is onto, we can find a $Q\in\Pin(\cal{P})$ such that $\Ad(Q)=\varphi_A$, that is, $QP_iQ^{-1}=AP_iA^{-1}$ for all $i=0,\dots, m$. Hence $U=Q^{-1}A\in\U^+(\cal{P})$ and then $\Psi(Q,U)=A$. Thus we have proved (i).

Let $m$ be even. Consider the group homomorphism $\Psi\colon \Spin(\cal{P})\times \U^\pm(\cal{P})\to\Aut(\cal{F}_\cal{P})$, $(Q,U)\mapsto QU$. Its kernel is isomorphic to $\Spin(\cal{P})\cap\U^\pm(\cal{P})$. Since $m+1$ is odd, then $-\Id\in \O^-(\cal{P})$, so $\Spin(\cal{P})\cap\U^-(\cal{P})=\emptyset$, and $\Spin(\cal{P})\cap\U^\pm(\cal{P})=Z(\Spin(\cal{P}))$. Now if $A\in\Aut(\cal{F}_\cal{P})$, then $\varphi_A\colon \cal{P}\to\cal{P}$, $P\mapsto APA^{-1}$, is an orthogonal transformation of $\cal{P}$. On the one hand, if $\varphi_A\in\SO(\cal{P})$, there exists a $Q\in\Spin(\cal{P})$ such that $\Ad(Q)=\varphi_A$, and $\Psi(Q,U)=A$, where $U=Q^{-1}A\in\U^+(\cal{P})$. On the other hand, if $\varphi_A\in\O^-(\cal{P})$, since $m+1$ is odd, then $-\varphi_A\in\SO(\cal{P})$, so there is a $Q\in\Spin(\cal{P})$ so that $\Ad(Q)=-\varphi_A$, and hence $\Psi(Q,U)=A$, where $U=Q^{-1}A\in\U^-(\cal{P})$. This proves (ii).

Finally, the assertions involving $Z(\Spin(\cal{P}))$ are well-known (see \cite[Th.~VII.7.5]{Si96}).
\end{proof}

Up to congruence, there are only $8$ FKM-foliations for which $m_1>m_2$. These are the ones with multiplicities $(m_1,m_2)$ equal to $(2,1)$, $(4,3)$ (two noncongruent examples), $(5,2)$, $(6,1)$, $(8,7)$ (two noncongruent examples), and $(9,6)$. However, on the one hand, the FKM-foliations with pairs $(2,1)$, $(6,1)$, and $(5,2)$ are congruent to those FKM-foliations with pairs $(1,2)$, $(1,6)$, and $(2,5)$, respectively; and one of the examples with pair $(4,3)$ is congruent to the FKM-foliation with pair $(3,4)$. On the other hand, the other example with multiplicities $(4,3)$ is homogeneous, as well as the FKM-foliation with pair $(9,6)$. Therefore, in our investigation of FKM-foliations we are putting aside only the two inhomogeneous FKM-foliations with pair $(m_1,m_2)=(8,7)$.

Our purpose now is to calculate $\U^+(\cal{P})$. To do this, we need first to recall some facts about representations of the Clifford algebras $\Cl^*_{m+1}$. See \cite[Ch.~I]{LM89} for details.

Each Clifford algebra $\Cl^*_{m+1}$ is a matrix algebra over some associative division algebra: $\R$, $\C$, or $\H$. We state the classification of low-dimensional Clifford algebras in Table~\ref{table:clifford}, where $\mathbb{K}(r)$ denotes the algebra of $(r\times r)$-matrices over $\mathbb{K}=\R,\C,\H$. The higher-dimensional Clifford algebras $\Cl^*_{m+1}$ can be obtained recursively by means of $\Cl^*_{m+8}=\Cl^*_{m}\otimes\R(16)$.
\begin{table}[h!]
\renewcommand{\arraystretch}{1.5}
\begin{tabular}{cccccccccc} \hline 
$m$        & $0$            & $1$          & $2$     & $3$     & $4$     & $5$                & $6$     & $7$     & $\dots$
\\
\hline
$\Cl^*_{m+1}$ & $\R\oplus\R$ & $\R(2)$ & $\C(2)$ & $\H(2)$ & $\H(2)\oplus\H(2)$ & $\H(4)$ & $\C(8)$ & $\R(16)$ & $\dots$
\\
\hline
\end{tabular}
\vspace{2ex}
\caption{Classification of the Clifford algebras $\Cl^*_{m+1}$}
\label{table:clifford}
\end{table}
The classification of $\Cl^*_{m+1}$-modules is obtained directly from the classification of the corresponding Clifford algebras. The algebra $\mathbb{K}(r)$ has only one equivalence class $\g{d}$ of irreducible representations (on $\mathbb{K}^r$), whereas the algebra $\mathbb{K}(r)\oplus\mathbb{K}(r)$ has exactly two equivalence classes of irreducible representations $\g{d}_+$, $\g{d}_-$ (both on $\mathbb{K}^r$), given by the projections onto each one of the factors. 
If $\Cl^*_{m+1}=\mathbb{K}(r)$, we will let $\chi\colon \Cl^*_{m+1}\to \End(\g{d})$ be the corresponding irreducible representation, and if $\Cl^*_{m+1}=\mathbb{K}(r)\oplus\mathbb{K}(r)$, the irreducible representations will be denoted by $\chi_+\colon \Cl^*_{m+1}\to\End(\g{d}_+)$ and $\chi_-\colon \Cl^*_{m+1}\to\End(\g{d}_-)$.
Therefore, if $m\not\equiv 0\, (\mod 4)$, each $\Cl^*_{m+1}$-module $V$ is isomorphic to $\oplus_{i=1}^k\g{d}$ for certain positive integer $k$, while if $m\equiv 0\, (\mod 4)$ each $\Cl^*_{m+1}$-module $V$ is isomorphic to $(\oplus_{i=1}^{k_+}\g{d}_+)\oplus(\oplus_{i=1}^{k_-}\g{d}_-)$ for nonnegative integers $k_-$, $k_+$; we will write $k=k_++k_->0$. The corresponding Clifford algebra representations will be denoted by $\chi_k$ and $\chi_{k_+,k_-}$, respectively. Furthermore, in the case $m\equiv 0\, (\mod 4)$ we can and will assume that $\chi_-=\chi_+\circ\alpha$, where $\alpha$ is the canonical involution of $\Cl^*_{m+1}=\Cl(\cal{E})$ that extends the map $-\Id$ on $\cal{E}$.

\begin{theorem}\label{th:U+}
Let $\cal{P}$ be a symmetric Clifford system. The group $\U^+(\cal{P})$ is isomorphic to 
\[
\begin{array}{rl@{\qquad\qquad}rl}
\O(k), \qquad & \text{if } m\equiv 1,7\,(\mod 8), 
\\
\U(k), \qquad & \text{if } m\equiv 2,6\,(\mod 8), & \O(k_+)\times \O(k_-), \qquad & \text{if } m\equiv 0\,(\mod 8),
\\
\Sp(k), \qquad & \text{if } m\equiv 3,5\,(\mod 8), & \Sp(k_+)\times \Sp(k_-), \qquad & \text{if } m\equiv 4\,(\mod 8).
\end{array}
\]
\end{theorem}
\begin{proof}
First, assume that $m\not\equiv 0\,(\mod 4)$. The real endomorphisms $U$ of $V=\oplus_{i=1}^k\g{d}$ can be identified with matrices $(U_{ij})$ with $U_{ij}\in\End(\g{d})$ for $i,j=1,\dots,k$. The endomorphisms $U=(U_{ij})$ that commute with the elements in $\Cl(\cal{P})=\chi_k(\Cl^*_{m+1})$ are exactly those whose $U_{ij}$ commute with the elements in $\chi(\Cl^*_{m+1})=\mathbb{K}(r)$. Equivalently, the $U_{ij}$ belong to the commuting subalgebra of $\mathbb{K}(r)$, which is isomorphic to $\mathbb{K}$. Hence the algebra of endomorphisms $U$ that commute with $\Cl(\cal{P})$ is isomorphic to $\mathbb{K}(k)$. Now $\U^+(\cal{P})$ is the set of those endomorphisms $U$ commuting with $\Cl(\cal{P})$ that are orthogonal transformations of $V$. Since $\R(k)\cap\O(k)=\O(k)$, $\C(k)\cap\O(2k)=\U(k)$, and $\H(k)\cap\O(4k)=\Sp(k)$, we get that $\U^+(\cal{P})$ is isomorphic to $\O(k)$ if $m\equiv 1,7\,(\mod 8)$, $\U(k)$ if $m\equiv 2,6\,(\mod 8)$, or $\Sp(k)$ if $m\equiv 3,5\,(\mod 8)$.

Let $m\equiv 0\,(\mod 4)$ and put $V=(\oplus_{i=1}^{k_+}\g{d}_+)\oplus(\oplus_{i=1}^{k_-}\g{d}_-)$. 
Arguing as above, one can show that the algebra of endomorphisms $U=(U_{ij})$ that commute with the elements of $\chi_{k_+,k_-}(\Cl^*_{m+1})$ is isomorphic to $\mathbb{K}(k_+)\oplus\mathbb{K}(k_-)$; note that if, for example, $i\in\{1,\dots,k_+\}$ and $j\in\{k_++1,\dots, k_++k_-\}$, then $U_{ij}\chi_-(f)=\chi_+(f)U_{ij}$ for all $f\in \Cl^*_{m+1}$ if and only if $U_{ij}=0$, since $\chi_+$ and $\chi_-$ are inequivalent representations. Restricting to orthogonal transformations of $V$, one readily finishes the proof.
\end{proof}

Let $\{e_1,\dots, e_k\}$ be the canonical $\mathbb{K}$-basis of $\mathbb{K}^k$, for $\mathbb{K}\in\{\R,\C,\H\}$. Let us regard $\g{d}$, $\g{d}^\pm$, and $\mathbb{K}^k$ as right vector spaces, in order to deal also with the quaternionic case. Assume, for example, that $m$ is odd, and let $\wt{\U}^+(\cal{P})$ be the corresponding classical group in Theorem~\ref{th:U+}. Therefore, Theorems~\ref{th:auto_FKM_1} and~\ref{th:U+} establish the following isomorphism of groups
\[
\Pin(\cal{P})\cdot\U^+(\cal{P})  \to  \Pin(m+1)\cdot \wt{\U}^+(\cal{P}),\quad QU\mapsto\wt{Q}\otimes\wt{U},
\]
where $\wt{Q}$ and $\wt{U}$ are defined as follows: for each $Q\in \Pin(\cal{P})$, let $f\in \Pin(\cal{E})$ so that $Q=\chi_k(f)$, and define $\wt{Q}=\chi(f)$; given $U\in\U^+(\cal{P})$, put $U=(U_{ij})$ with $U_{ij}\in\End_\R(\g{d})$, and define $\wt{U}=(\wt{u}_{ij})$, where $\wt{u}_{ij}\in\mathbb{K}$ and $U_{ij}v=v\wt{u}_{ij}$ for all $v\in\g{d}$ and each $i,j=1,\dots, k$. Moreover, it is straightforward to check that the map $\bigoplus_{i=1}^k\g{d} \to  \g{d}\otimes_\mathbb{K}\mathbb{K}^k$, $(v_1,\dots,v_k)\mapsto\sum_{i=1}^k v_i\otimes e_i$,
gives an equivalence between the representation of $\Aut(\cal{F}_\cal{P})\cong \Pin(\cal{P})\cdot\U^+(\cal{P})$ on $V=\bigoplus_{i=1}^k\g{d}$ and the tensor product representation of $\Pin(m+1)\cdot \wt{\U}^+(\cal{P})$ on $\g{d}\otimes_\mathbb{K}\mathbb{K}^k$. If $m$ is even, the previous argument applies (with small changes) to the subgroup $\Spin(\cal{P})\cdot\U^+(\cal{P})$ of $\Aut(\cal{F}_\cal{P})$. Thus, we obtain:

\begin{theorem}\label{th:auto_FKM_2}
Let $\cal{F}_{\cal{P}}$ be an FKM-foliation satisfying $m_1\leq m_2$. 

If $m$ is odd, the representation of the group $\Aut(\cal{F}_\cal{P})$ on $V$ is equivalent to the action of 
\begin{align*}
\Pin(m+1)\cdot \O(k), \qquad & \text{if } m\equiv 1,7\,(\mod 8), \quad \text{or}
\\
\Pin(m+1)\cdot \Sp(k), \qquad & \text{if } m\equiv 3,5\,(\mod 8),
\end{align*}
on $\g{d}\otimes_\mathbb{K}\mathbb{K}^k$, given by the  pin representation $\g{d}$ on the left factor, and by the standard representation of $\O(k)$ or  $\Sp(k)$ on the right factor $\mathbb{R}^k$ or $\H^k$, respectively.

If $m\equiv 2\,(\mod 4)$, the representation of the subgroup $\Spin(\cal{P})\cdot\U^+(\cal{P})$ of $\Aut(\cal{F}_\cal{P})$ on $V$ is equivalent to the action of
$\Spin(m+1)\cdot \U(k)$ on $\g{d}\otimes_\mathbb{C}\mathbb{C}^k$, given by the spin representation $\g{d}$ on the left factor, and by the standard representation of $\U(k)$ on the right factor $\mathbb{C}^k$.

In case $m\equiv 0\,(\mod 4)$, the representation of the subgroup $\Spin(\cal{P})\cdot\U^+(\cal{P})$ of $\Aut(\cal{F}_\cal{P})$ on $V$ is equivalent to the action of
\begin{align*}
\Spin(m+1)\cdot(\O(k_+)\times \O(k_-)), \qquad & \text{if } m\equiv 0\,(\mod 8), \quad \text{or}
\\
\Spin(m+1)\cdot(\Sp(k_+)\times \Sp(k_-)), \qquad & \text{if } m\equiv 4\,(\mod 8),
\end{align*}
on $(\g{d}_+\otimes_\mathbb{K}\mathbb{K}^{k_+})\oplus(\g{d}_-\otimes_\mathbb{K}\mathbb{K}^{k_-})$, given by the spin representations $\g{d}_+$, $\g{d}_-$ on the left factors, and by the standard representations of $\O(k_{\pm})$ or $\Sp(k_{\pm})$ on the right factors $\mathbb{R}^{k_\pm}$ or $\mathbb{H}^{k_\pm}$. 
\end{theorem}

If $m$ is even, the description of the group of automorphisms of $\cal{F}_\cal{P}$ depends on $\U^-(\cal{P})$. We will only need a description of this set for the case $m\equiv 0\,(\mod 4)$.

\begin{proposition}\label{prop:U-}
Let $\cal{P}$ be a symmetric Clifford system with $m\equiv 0\,(\mod 4)$. We have that $\U^-(\cal{P})=\emptyset$ if $k_+\neq k_-$, or $\U^-(\cal{P})=\tau U^+(\cal{P})$ if $k_+=k_-$,
where $\tau$ is the orthogonal transformation of $V=(\oplus_{i=1}^{k_+}\g{d}_+)\oplus(\oplus_{i=1}^{k_-}\g{d}_-)$ defined by 
\[
\tau(v_1,\dots,v_{k_+},v_{k_++1},\dots,v_k)=(v_{k_++1},\dots,v_{k},v_{1},\dots,v_{k_+}),\]
where $v_i\in\g{d}_+$ for $i=1,\dots,k_+$, and $v_i\in\g{d}_-$ for $i=k_++1,\dots,k$.
\end{proposition}
\begin{proof}
First note that if $\sigma$ is an element in $\U^-(\cal{P})$, then $\U^-(\cal{P})=\sigma\U^+(\cal{P})$. With the notation as above, let $\sigma=(\sigma_{ij})$ anticommute with the endomorphisms in $\chi_{k_+,k_-}(\cal{E})$. Equivalently, for all $f\in\cal{E}$ we have that $\sigma_{ij}\chi_+(f)=-\chi_+(f)\sigma_{ij}$ if $i,j\in\{1,\dots,k_+\}$, $\sigma_{ij}\chi_-(f)=-\chi_-(f)\sigma_{ij}$ if $i,j\in\{k_++1,\dots,k\}$, $\sigma_{ij}\chi_+(f)=-\chi_-(f)\sigma_{ij}$ if $i\in\{k_++1,\dots, k\}$ and $j\in\{1,\dots,k_+\}$, and $\sigma_{ij}\chi_-(f)=-\chi_+(f)\sigma_{ij}$ if $j\in\{k_++1,\dots, k\}$ and $i\in\{1,\dots,k_+\}$. Since $\chi_-=\chi_+\circ\alpha$ and $\chi_+$, $\chi_-$ are not equivalent, these conditions imply that $\sigma_{ij}=0$ if $i,j\in\{1,\dots,k_+\}$ or $i,j\in\{k_++1,\dots,k\}$. 

If $k_+\neq k_-$ then $\sigma$ is not invertible, so $\U^-(\cal{P})=\emptyset$. If $k_+=k_-$, the orthogonal transformation $\tau=(\tau_{ij})$ given above satisfies $\tau_{ij}=\Id$ if $i=k_++j$ or $j=k_++i$, and $\tau_{ij}=0$ otherwise. Then $\tau$ anticommutes with the elements of $\chi_{k_+,k_-}(\cal{E})$.
\end{proof}

\section{Singular Riemannian foliations on complex projective spaces}\label{sec:foliations}
In this section we present some general theory for the study of singular Riemannian foliations with closed leaves on $\C P^n$. First, we recall the definition of singular Riemannian foliation. In \S\ref{subsec:preserving} we obtain a criterion to determine all complex structures preserving a given foliation. The congruence of foliations on $\C P^n$ projected using different complex structures is analyzed in \S\ref{subsec:congruence}.

Let $\mathcal{F}$ be a decomposition of a Riemannian manifold into connected injectively immersed submanifolds, called \emph{leaves}, which may have different dimensions. We say that $\mathcal{F}$ is a \emph{singular Riemannian foliation} if it is a transnormal system (i.e.\ every geodesic orthogonal to one leaf remains orthogonal to all the leaves that it intersects), and if $T_p L=\{X_p:X\in\Xi_\mathcal{F}\}$ for every leaf $L$ in $\mathcal{F}$ and every $p\in L$, where $\Xi$ is the module of smooth vector fields on the ambient manifold that are everywhere tangent to the leaves of $\mathcal{F}$. The leaves of maximal dimension are called \emph{regular}, and the other ones are called \emph{singular}. Further information on this concept can be found in \cite{Al04}, \cite{Th10}.
For the sake of brevity, in this work we refer to singular Riemannian foliations simply as foliations. 

\subsection{Complex structures preserving foliations}\label{subsec:preserving}

Let $V=\R^{2n+2}$ and let $\cal{F}$ be a foliation on $S(V)=S^{2n+1}$. We will say that a complex structure $J$ in $V$ \emph{preserves the foliation} $\mathcal{F}$ if $\mathcal{F}$ is the lift of some foliation of the complex projective space $\C P^n$ under the Hopf map $S^{2n+1}\to\C P^n$ determined by $J$; or equivalently, if the leaves of $\mathcal{F}$ are foliated by the Hopf circles determined by $J$.
Since Hopf fibrations are Riemannian submersions, each foliation on $\C P^n$ is obtained by projecting some foliation $\cal{F}$ on $S^{2n+1}$ by some Hopf map whose $J$ preserves $\cal{F}$. Therefore, the study of foliations on complex projective spaces is reduced to the study of the complex structures that preserve foliations on odd-dimensional spheres.

It is equivalent to give a foliation of a sphere $S(V)$ and to give a foliation of the Euclidean space $V$ whose leaves are contained in concentric spheres with center at the origin: simply extend the given foliation on $S(V)$ by homotheties to $V$, or inversely, restrict the foliation on $V$ to $S(V)$. Sometimes along this work we will implicitly take this fact into account. 

Fix a foliation $\cal{F}$ of the sphere $S(V)\subset V$. Consider an effective representation $\rho\colon K\to \O(V)$ of a Lie group $K$ such that $\rho(K)$ is the maximal connected group of orthogonal transformations of $V$ that send each leaf of $\cal{F}$ onto itself. Let $\rho_*\colon\g{k}\to\g{so}(V)$ be the Lie algebra homomorphism defined by $\rho$.

\begin{proposition}\label{prop:complex_structure}
With $\cal{F}$, $K$, and $\rho$ as above, we have:
\begin{itemize}
\item[(i)] A complex structure $J$ on $V$ preserves $\cal{F}$ if and only if $J=\rho_*(X)$ for some $X\in\g{k}$.
\item[(ii)] Assume that $K$ is compact and fix a maximal abelian subalgebra $\g{t}$ of $\g{k}$. If $H\in\g{t}$ and $k\in K$, then $\rho_*(H)$ is a complex structure on $V$ if and only if $\rho_*(\Ad(k)H)$ is a complex structure on $V$. Moreover, a complex structure $J$ on $V$ preserves $\cal{F}$ if and only if $J=\rho_*(\Ad(k)H)$ for some $k\in K$ and $H\in\g{t}$.
\item[(iii)] Let $\cal{F}=\cal{F}_1\times \dots\times \cal{F}_r$ be a product foliation on $V=\bigoplus_{i=1}^r V_i$, where each $\cal{F}_i$ is the extension of a foliation on $S(V_i)$ to $V_i$. Then $K=\prod_{i=1}^r K_i$ for certain subgroups $K_i$ of $K$, where $\rho(K_i)$ is the maximal connected group of orthogonal transformations of $V$ that act trivially on the orthogonal complement of $V_i$ in $V$ and map the leaves of $\cal{F}_i$ onto themselves. If $X=\sum_{i=1}^r X_i\in\g{k}=\bigoplus_{i=1}^r \g{k}_i$, then $\rho_*(X)$ is a complex structure on $V$ if and only if $\rho_*(X_i)\rvert_{V_i}$ is a complex structure on $V_i$, for every $i$.
\end{itemize}
\end{proposition}
\begin{proof}
If $J=\rho_*(X)$ is a complex structure on $V$, its Hopf circles are integral curves of the Hopf vector field $J\xi$, where $\xi_v=v$ for $v\in S(V)$, and each Hopf circle is contained in one leaf of $\cal{F}$, since for every $v\in V$, $Jv=\rho_*(X) v$ is tangent to the leaf through $v$. 

Conversely, assume that $J$ is a complex structure on $V$ that preserves $\cal{F}$. Then $\cal{T}^1=\{\cos(t)\Id+\sin(t)J:t\in\R\}$ is a $1$-dimensional group which preserves $\cal{F}$. Let $K'$ be the subgroup of $\O(V)$ generated by $\rho(K)$ and $\cal{T}^1$, which is connected and leaves every leaf of $\cal{F}$ invariant. By the maximality of $\rho(K)$, we have that $K'\subset \rho(K)$ and then $\cal{T}^1$ is a subgroup of $\rho(K)$. If we differentiate, we get that $J\in\rho_*(\g{k})$, which shows (i).

Every transformation of $V$ of the form $\rho_*(X)$, with $X\in \g{k}$, is a complex structure if and only if $\rho_*(X)^2=-\Id$, since $\rho_*(X)\in\g{so}(V)$ is skew-symmetric. Then, with the notation of (ii), we have that $\rho_*(\Ad(k)H) = \Ad(\rho(k))\rho_*(H)=\rho(k)\rho_*(H)\rho(k)^{-1}$ and, hence, $\rho_*(\Ad(k)H)^2=\rho(k)\rho_*(H)^2\rho(k)^{-1}$. Since $K$ is a connected compact Lie group, then $\g{k}=\bigcup_{k\in K}\Ad(k)\g{t}$, from where we get (ii).

Finally, (iii) follows from the effectiveness of $\rho$ and from the facts that the leaves of $\cal{F}$ are products of leaves of the foliations $\cal{F}_i$, and the $V_i$ are invariant subspaces for $\rho$.
\end{proof}

From now on, $\cal{F}$ will be a foliation with closed leaves on $S(V)$. Then $K$ is compact. We also fix a maximal abelian subalgebra $\g{t}$ of $\g{k}$. Let $(\cdot)^\C$ denote complexification. We will use some known facts on compact Lie groups that can be consulted in \cite[Ch.~IV]{Kn02}.

Let $\Delta_\g{k}=\Delta(\g{k}^\C,\g{t}^\C)$ be the root system of $\g{k}$ with respect to $\g{t}$, that is, the set of nonzero elements $\alpha\in(\g{t}^\C)^*$ such that the corresponding eigenspace $\g{k}_\alpha=\{X\in\g{k}^\C:\ad(H)X=i\alpha(H)X,\text{ for all }H\in\g{t}\}$ is nonzero. Let  $\g{k}^\C=\g{t}^\C\oplus\bigoplus_{\alpha\in\Delta_\g{k}}\g{k}_\alpha$ be the root space decomposition of $\g{k}^\C$ with respect to $\g{t}^\C$. Recall that $\g{t}=Z(\g{k})\oplus\g{t}'$, where $Z(\g{k})$ is the center of $\g{k}$ and $\g{t}'$ is a maximal abelian subalgebra of the semisimple Lie algebra $[\g{k},\g{k}]$. The roots in $\Delta_\g{k}$ vanish on $Z(\g{k}^\C)$, and $\Delta_\g{k}$ is an abstract reduced root system in the subspace $((\g{t}')^\C)^*$ of $(\g{t}^\C)^*$.

Let $\Delta_V=\Delta(V^\C,\g{t}^\C)$ be the set of weights of the representation $\rho_*^\C\colon\g{k}^\C\to\g{gl}(V^\C)$, that is, those elements $\lambda\in(\g{t}^\C)^*$ so that the subspace $V_\lambda=\{v\in V^\C:\rho_*^\C(H)v=i\lambda(H)v, \text{ for all } H\in\g{t}\}$ is nonzero. Then we have the weight space decomposition $V^\C=\bigoplus_{\lambda\in\Delta_V} V_\lambda$. Notice that, according to our notation, all roots and weights are real on $\g{t}$.

\begin{proposition}\label{prop:plus_minus_1}
Let $H\in\g{t}$. Then $\rho_*(H)$ is a complex structure on $V$ if and only if $\lambda(H)\in\{\pm 1\}$ for every weight $\lambda$ of the representation $\rho_*^\C$.
\end{proposition}
\begin{proof}
The skew-symmetric transformation $\rho_*(H)$ is a complex structure on $V$ if and only if $\rho_*(H)^2 = -\Id$, or equivalently, if $\rho_*^\C(H)^2=-\Id$.  For an arbitrary $\lambda\in \Delta_V$, let $v_\lambda\in V_\lambda$. Then  $\rho_*^\C(H)^2v_\lambda=-\lambda(H)^2v_\lambda$. Hence $\rho_*^\C(H)^2=-\Id$ if and only if $\lambda(H)^2=1$ for all $\lambda\in\Delta_V$.
\end{proof}

In view of Propositions \ref{prop:complex_structure} and \ref{prop:plus_minus_1}, if $0$ is a weight of $\rho_*^\C$, then $\cal{F}$ cannot be projected to the complex projective space.   
Moreover, once one knows the maximal connected group of orthogonal transformations preserving the leaves of the foliation $\cal{F}$, these propositions allow to determine all possible complex structures preserving $\cal{F}$ in a computational way.

\subsection{Congruence of projected foliations}\label{subsec:congruence}

Now we focus on the study of the congruence of foliations on complex projective spaces. We start with the following basic result.

\begin{proposition}\label{prop:general_foliation}
Let $V=\R^{2n+2}$. Let $J_1$, $J_2$ be two complex structures on $V$, $\C P^n_1$, $\C P^n_2$ the corresponding complex projective spaces, and $\pi_1$, $\pi_2$ the corresponding Hopf maps. 

Two foliations $\mathcal{G}_1\subset\C P^n_1$ and $\mathcal{G}_2\subset\C P^n_2$ are congruent if and only if there exists an orthogonal transformation $A\in \O(V)$ satisfying $A J_1 A^{-1}=\pm J_2$ and mapping leaves of $\pi_1^{-1}\mathcal{G}_1$ to leaves of $\pi_2^{-1}\mathcal{G}_2$.
\end{proposition}
\begin{proof}
$\mathcal{G}_1$ and $\mathcal{G}_2$ are congruent if and only if there exists a unitary or anti-unitary transformation $A$ between $(V,J_1)$ and $(V,J_2)$ (i.e.\ $A\in \O(V)$ and $A J_1 A^{-1}=\pm J_2$) whose induced isometry $[A]:\C P^n_1\to\C P^n_2$ takes the leaves of $\mathcal{G}_1$ to the leaves of $\mathcal{G}_2$. But this condition is equivalent to the one in the statement.
\end{proof}

In particular, a necessary condition for two foliations on a complex projective space to be congruent is that their lifts to the sphere are congruent. In view of this, in order to study the congruence of foliations on complex projective spaces it suffices to decide when two complex structures preserving some fixed foliation give rise to congruent foliations in the corresponding complex projective spaces.

Let $\cal{F}$, $K$, $\rho$, and $\g{t}$ be as in \S\ref{subsec:preserving}. Consider two complex structures $J_i=\rho_*(X_i)$, $i=1,2$, on $V$ preserving the foliation $\mathcal{F}$. Let us say that $J_1$ and $J_2$ are equivalent, and write $J_1\sim J_2$, if $J_1$ and $J_2$ give rise to congruent foliations on the complex projective space. We also denote by $\sim$ the corresponding equivalence relation on the subset $\cal{J}$ of those $X$ in $\g{k}$ such that $\rho_*(X)$ is a complex structure preserving $\mathcal{F}$. The problem of the congruence of foliations on $\C P^n$ is then reduced to the determination of the $\sim$-equivalence classes of $\cal{J}$. 

Let $\Aut(\cal{F})$ be the group of automorphisms of the foliation $\cal{F}$, i.e.\ the group of those orthogonal transformations of $V$ that map leaves of $\cal{F}$ to leaves of $\cal{F}$. Clearly, $\rho(K)$ is a subgroup of $\Aut(\cal{F})$. 
Due to the effectiveness of $\rho$, each $A\in\Aut(\cal{F})$ defines an automorphism $\phi_A\in\Aut(\g{k})$ of the Lie algebra $\g{k}$, by means of the relation $A\rho_*(X)A^{-1}=\rho_*(\phi_A(X))$. Consider the group $\Aut(\g{k},\cal{F})$ of those linear isomorphisms $\varphi_A\colon\g{k}\oplus V\to\g{k}\oplus V$ defined by $\varphi_A\rvert_\g{k}=\phi_A$ and $\varphi_A\rvert_V=A$, where $A$ runs over $\Aut(\cal{F})$. Note that $(\Ad\oplus\rho)(K)=\{\varphi_{\rho(k)}:k\in K\}$ is a subgroup of $\Aut(\g{k},\cal{F})$.

In view of this notation, Proposition~\ref{prop:general_foliation} asserts that two complex structures $J_i=\rho_*(X_i)$, $i=1,2$, are equivalent (i.e.\ $X_1\sim X_2$) if and only if there exists $A\in \Aut(\cal{F})$ with $A\rho_*(X_1) A^{-1}=\pm \rho_*(X_2)$, or equivalently, if there exists $\varphi\in\Aut(\g{k},\cal{F})$ so that $\varphi X_1=\pm X_2$.

Every $\sim$-equivalence class intersects the maximal abelian subalgebra $\g{t}$ of $\g{k}$, since $\g{k}=\bigcup_{k\in K} \Ad(k)\g{t}$ and $\Ad(K)\subset \Aut(\g{k},\cal{F})\rvert_\g{k}$. We can therefore restrict $\sim$ to $\g{t}$ and analyze the set $\cal{J}\cap\g{t}$ and its partition in $\sim$-equivalence classes. 

\begin{proposition}\label{prop:invariant_torus}
Let $T_1,T_2\in\cal{J}\cap\g{t}$. Then $T_1\sim T_2$ if and only if there exists an automorphism $\varphi\in \Aut(\g{k},\cal{F})$ preserving $\g{t}$ such that $\varphi T_1=\pm T_2$.
\end{proposition}
\begin{proof}
The sufficiency is clear according to the previous remarks. For the necessity we will use a well-known  argument in the study of compact groups (cf.~\cite[Prop.~4.53]{Kn02}).

Let $\phi\in\Aut(\g{k},\cal{F})$ be such that $\phi T_1=\pm T_2$. The centralizer $Z_K(T_2)$ of $T_2$ in $K$ is a compact group, and $\g{t}$, $\phi(\g{t})$ are maximal abelian subalgebras of $Z_\g{k}(T_2)$, which is the Lie algebra of $Z_K(T_2)$. Hence there exists $k\in Z_K(T_2)$ such that $\Ad(k)\g{t}=\phi(\g{t})$. Define $\varphi=(\Ad\oplus\rho)(k^{-1})\circ\phi\in\Aut(\g{k},\cal{F})$. Then $\varphi(\g{t})=\g{t}$ and $\varphi(T_1)=\pm \Ad(k^{-1})T_2=\pm T_2$.
\end{proof}

Since the leaves of $\cal{F}$ are closed and equidistant, it follows that the group $\Aut(\cal{F})$ is compact, so $\Aut(\g{k},\cal{F})$ and $\Aut(\g{k},\cal{F})\rvert_\g{k}$ are also compact. Hence, there exists a positive definite $\Aut(\g{k},\cal{F})\rvert_\g{k}$-invariant inner product $\langle \cdot,\cdot \rangle$ on $\g{k}$. Then $\langle\g{t}',Z(\g{k})\rangle=0$. Moreover, $\langle \cdot,\cdot \rangle$ restricted to each simple ideal of $\g{k}$ is a negative multiple of the Killing form of such ideal.

For each $\lambda\in\g{t}^*$, we define $H_\lambda\in\g{t}$ by $\langle H_\lambda,H\rangle=\lambda(H)$, for all $H\in \g{t}$. Then $\langle\cdot,\cdot\rangle$ induces an inner product on $\g{t}^*$ in a natural way, by means of $\langle \lambda,\mu\rangle=\langle H_\lambda,H_\mu\rangle$, for $\lambda,\mu\in \g{t}^*$. If $\alpha\in\Delta_\g{k}$ is a root of $\g{k}$, the corresponding $H_\alpha$ will be called a coroot, whereas if $\lambda\in\Delta_V$ is a weight of $\rho_*^\C$, we will say that $H_\lambda$ is a coweight. Note that the coroots belong to $\g{t}'$, since $\g{t}=Z(\g{k})\oplus \g{t}'$ and the roots vanish on $Z(\g{k})$.

We will say that an orthogonal transformation of $\g{t}$ is an automorphism of $\Delta_\g{k}$ if it maps the set of coroots $\{H_\alpha:\alpha\in\Delta_\g{k}\}$ onto itself. The group of automorphisms of $\Delta_\g{k}$ is noted by $\Aut(\Delta_\g{k})$. The subgroup of those automorphisms of $\Delta_\g{k}$ that map the set of coweights $\{H_\lambda:\lambda\in\Delta_V\}$ onto itself will be denoted by $\Aut(\Delta_\g{k},\Delta_V)$. The action of $\Aut(\Delta_\g{k})$ on $\g{t}$ induces an action of $\Aut(\Delta_\g{k})$ on $\g{t}^*$ by means of $\varphi(\alpha)=\alpha\circ\varphi^{-1}$ for $\alpha\in\g{t}^*$ and $\varphi\in\Aut(\Delta_\g{k})$.

\begin{proposition}\label{prop:auto_roots}
The restriction to $\g{t}$ of each element of $\Aut(\g{k},\cal{F})$ preserving $\g{t}$ gives an element of $\Aut(\Delta_\g{k},\Delta_V)$. 
\end{proposition}
\begin{proof}
Consider an element $\varphi\in\Aut(\g{k},\cal{F})$ with $\varphi(\g{t})=\g{t}$, and let $\phi=\varphi^\C$, which is  a linear automorphism of $\g{k}^\C\oplus V^\C$ such that $\phi\rvert_{\g{k}^\C}$ is a Lie algebra automorphism of $\g{k}^\C$. If $\alpha\in\Delta_\g{k}$ and $X\in\g{k}_\alpha$, then $[\phi H, \phi X]=\alpha(H)\phi X$ for all $H\in\g{t}^\C$, so $\phi(\g{k}_\alpha)=\g{k}_\beta$, where $\beta=\alpha\circ\phi^{-1}\rvert_{\g{t}^\C}\in\Delta_\g{k}$. Moreover, $\beta(H)=\alpha(\phi^{-1}H)=\langle\phi^{-1}H, H_\alpha\rangle=\langle H,\phi H_\alpha\rangle$ for all $H\in\g{t}$, and thus $\phi H_\alpha=H_\beta$. 
If $\lambda\in\Delta_V$ and $X\in V_\lambda$, then $\rho_*^\C(\phi H)\phi X= \lambda(H)\phi X$ for all $H\in\g{t}^\C$, so $\rho_*^\C(H)\phi X=\mu(H)\phi X$, where $\mu=\lambda\circ\phi^{-1}\rvert_{\g{t}^\C}\in\Delta_V$. Hence $\phi V_\lambda=V_\mu$ and, similarly as for the coroots, we get that $\phi H_\lambda=H_\mu$. We have thus shown that $\varphi\rvert_\g{t}=\phi\rvert_\g{t}\in\Aut(\Delta_\g{k},\Delta_V)$.
\end{proof}

We will denote by $\Aut_\cal{F}(\Delta_\g{k},\Delta_V)$ the subgroup of those automorphisms in $\Aut(\Delta_\g{k},\Delta_V)$ that are restriction of automorphisms in $\Aut(\g{k},\cal{F})$ preserving $\g{t}$, and by $\Aut_\cal{F}^\pm(\Delta_\g{k},\Delta_V)$ the group generated by $\Aut_\cal{F}(\Delta_\g{k},\Delta_V)$ and $-\Id_\g{t}$. The elements of $\Aut_\cal{F}^\pm(\Delta_\g{k},\Delta_V)$ leave $\Delta_V$ invariant; indeed $-\Id_\g{t}\in\Aut(\Delta_\g{k},\Delta_V)$ since $\rho_*^\C$ is a complex representation of real type, hence self dual. A straightforward consequence of Propositions~\ref{prop:invariant_torus} and \ref{prop:auto_roots} is the following

\begin{corollary}
If $T_1,T_2\in\cal{J}\cap\g{t}$, then $T_1\sim T_2$ if and only if there exists $\varphi\in \Aut_\cal{F}^\pm(\Delta_\g{k},\Delta_V)$ such that $\varphi T_1=T_2$.
\end{corollary}

Let us fix a set of simple roots $\Pi_\g{k}=\{\alpha_1,\ldots,\alpha_r\}$ for the root system of $\g{k}$.
Let $W(\Delta_\g{k})=N_K(\g{t})/Z_K(\g{t})$ be the Weyl group of the root system $\Delta_\g{k}$; we will regard $W(\Delta_\g{k})$ as a group of transformations of $\g{t}$ that leave $\g{t}'$ invariant. Since $W(\Delta_\g{k})\subset\Ad(K)\subset \Aut(\g{k},\cal{F})\rvert_\g{k}$, Proposition~\ref{prop:auto_roots} implies that $W(\Delta_\g{k})\subset\Aut(\Delta_\g{k},\Delta_V)$, i.e.\ the Weyl group action on $\g{t}$ leaves the set of coroots and the set of coweights invariant. It is known that $W(\Delta_\g{k})$ is generated by the reflections in $\g{t}$ around the hyperplanes of equation $\alpha=0$, for all $\alpha\in\Delta_\g{k}$, or even for all $\alpha\in\Pi_\g{k}$. The coroot $H_\alpha$ is a normal vector to the hyperplane $\alpha=0$.

The subset $\bar{C}$ of $\g{t}$ defined by the inequalities $\alpha\geq 0$, for every $\alpha\in\Pi_\g{k}$, constitutes a fundamental domain for the action of $W(\Delta_\g{k})$ on $\g{t}$, that is, every $W(\Delta_\g{k})$-orbit intersects $\bar{C}$ in exactly one point. The set $\bar{C}$ is the Cartesian product of the closure of a Weyl chamber of $\Delta_{\g{k}}=\Delta_{[\g{k},\g{k}]}$ in $\g{t}'$ times the center of $\g{k}$. 
We will denote by $\Out_\cal{F}^\pm(\Delta_\g{k},\Delta_V)$ (resp.\ $\Out(\Delta_\g{k})$, $\Out(\Delta_\g{k},\Delta_V)$) the subgroup of $\Aut_\cal{F}^\pm(\Delta_\g{k},\Delta_V)$ (resp.\ $\Aut(\Delta_\g{k})$, $\Aut(\Delta_\g{k},\Delta_V)$) of all automorphisms leaving $\bar{C}$ invariant, or equivalently, leaving invariant the simple coroots. 

Since $W(\Delta_\g{k})\subset \Aut_\cal{F}(\Delta_\g{k},\Delta_V)$ and $\bar{C}$ is a fundamental domain for the action of $W(\Delta_\g{k})$, the problem of understanding the partition of $\cal{J}\cap\g{t}$ in its $\sim$-equivalence classes is then reduced to determining the set $\cal{J}\cap\bar{C}$ and deciding which of its elements are $\sim$-related.

\begin{proposition}\label{prop:criterion_Out_F}
Let $T_1,T_2\in\cal{J}\cap\bar{C}$. Then $T_1\sim T_2$ if and only if there is $\varphi\in\Out_\cal{F}^\pm(\Delta_\g{k},\Delta_V)$ such that $\varphi (T_1)= T_2$. 
\end{proposition}

\begin{proof}
The sufficiency of the claim is clear. Let us assume $T_1\sim T_2$. Then there is $\phi\in\Aut_\cal{F}^\pm(\Delta_\g{k},\Delta_V)$ such that $\phi(T_1)= T_2$. In particular $\phi\in\Aut(\Delta_\g{k})\cong W(\Delta_\g{k})\ltimes \Out(\Delta_\g{k})$, so we can put $\phi=\phi'\circ \varphi$, with $\phi'\in W(\Delta_\g{k})$ and $\varphi\in \Out(\Delta_\g{k})$. Then $\varphi(T_1)\in \bar{C}$ and $\phi'(\varphi(T_1))=\phi(T_1)=T_2\in\bar{C}$. But $\bar{C}$ is a fundamental domain for the action of $W(\Delta_\g{k})$, hence $\varphi(T_1)=T_2$. Finally notice that $\varphi\in \Out_\cal{F}^\pm(\Delta_\g{k},\Delta_V)$, because $\phi,\phi'\in\Aut_\cal{F}^\pm(\Delta_\g{k},\Delta_V)$.
\end{proof}

We introduce now one of the key ideas of our work, namely: the usage of certain generalizations of the so-called extended Vogan diagrams of inner symmetric spaces. This particular case will be discussed in Section~\ref{sec:homogeneous}. 

Given a complex finite dimensional representation $\eta$ of a compact Lie algebra $\g{k}$, the \emph{lowest weight diagram} of $\eta$ is constructed as follows. Consider the Dynkin diagram of $\g{k}$, where each simple root of $\Delta_\g{k}$ is represented by a white node, and draw as many black nodes as lowest weights of $\eta$, counted with multiplicity. Join each black node corresponding to a lowest weight $\lambda$ to those white nodes corresponding to roots $\alpha$ with $\langle \alpha,\lambda\rangle\neq 0$ by means of a simple line. Finally, attach to each one of these new edges the integer value $2\langle \alpha,\lambda\rangle/\langle \alpha,\alpha\rangle$ as a label; if no label is attached, we understand that the associated value is $-1$.

An automorphism (or symmetry) of a lowest weight diagram is a permutation of its nodes preserving the graph, the black nodes, and the labels of the edges between black nodes and white nodes. Having in mind the notation of this section, we will talk about the lowest weight diagram of $\rho_*^\C$ or, directly, of the foliation $\cal{F}$. The study of the symmetries of the lowest weight diagrams of certain foliations will be a useful tool in our work. The following result gives a first idea of the interest of these diagrams. 

\begin{proposition}\label{prop:auto_diagram}
Each automorphism in $\Out(\Delta_\g{k},\Delta_V)$ induces an automorphism of the lowest weight diagram of $\rho_*^\C$ in a natural way. This correspondence is injective if the set of simple roots and lowest weights generates $\g{t}^*$.
\end{proposition}
\begin{proof}
Let $\varphi\in\Out(\Delta_\g{k},\Delta_V)$ and $\lambda\in\Delta_V$. 
Then $\lambda$ is a lowest weight of $\rho_*^\C$ if and only if $\lambda-\sum_{\alpha_i\in\Pi_\g{k}}n_i\alpha_i$ (with all $n_i\in\mathbb{N}\cup\{0\}$) is not a weight unless all $n_i$ vanish. 
Since $\varphi(H_\lambda - \sum_{\alpha_i\in\Pi_\g{k}}n_i H_{\alpha_i})=\varphi(H_\lambda) - \sum_{\alpha_i\in\Pi_\g{k}}n_i\varphi(H_{\alpha_i})$ and $\varphi$ preserves the set of simple coroots and the set of coweights, we get that $\varphi$ preserves the set of lowest coweights of $\rho_*^\C$. 
As moreover $\varphi$ is an orthogonal transformation of $\g{t}$, we conclude that $\varphi$ induces a symmetry of the lowest weight diagram of $\rho_*^\C$.
The last assertion of the statement is immediate.
\end{proof}

\section{Projecting homogeneous polar foliations} \label{sec:homogeneous}
Our goal in this section is to classify isoparametric foliations on complex projective spaces obtained by projection of homogeneous polar foliations in spheres. In \S\ref{subsec:complex_structures_homogeneous} we characterize the homogeneous polar foliations that can be projected to the complex projective space and determine the complex structures that can be used with that end. In \S\ref{subsec:congruence_homogeneous} we investigate the congruence of the corresponding projected foliations. In \S\ref{subsec:classification_homogeneous} we derive the classification.

First of all, we recall the notion of inner symmetric space, we introduce some known facts about Vogan diagrams and the Borel-de Siebenthal theory, and we fix some notation that will be used throughout this section.

Let $(G,K)$ be a symmetric pair, $\g{g}=\g{k}\oplus\g{p}$ its Cartan decomposition, and $\theta$ the corresponding Cartan involution. Then $G/K$ (or $(G,K)$) is said to be \emph{inner} or \emph{equal-rank} if its involution $\theta$ is inner. This happens exactly when $\g{g}$ and $\g{k}$ have equal rank \cite[Ch.~IX, Th.~5.6]{He78}, or when the Euler characteristic of $G/K$ is nonzero \cite[Ch.~XI, Th.~VII]{GHV76}. 

For the study of Vogan diagrams and Borel-de Siebenthal theory, we refer to \cite[\S VI.8--10 and Appx. C.3--4]{Kn02} (where the pictures of Vogan diagrams can be found) and \cite[Ch.~8]{GG78}. Here we give a quick overview for the particular case of inner symmetric spaces.

Assume that $(G,K)$ is an inner compact symmetric pair. Then a maximal abelian subalgebra $\g{t}$ of $\g{k}$ is also a maximal abelian subalgebra of $\g{g}$. Let $\Delta_\g{g}$ be the root system of $\g{g}$ with respect to $\g{t}$, and let $\g{g}^\C=\g{t}^\C\oplus\bigoplus_{\alpha\in \Delta_\g{g}}\g{g}_\alpha$ be the root space decomposition.
For every $\alpha\in\Delta_\g{g}$, either $\g{g}_\alpha\subset \g{k}^\C$ or $\g{g}_\alpha\subset \g{p}^\C$ holds. In the first case we say that $\alpha$ is \emph{compact}; in the second case, $\alpha$ is \emph{noncompact}. This terminology is motivated by the consideration of the real semisimple Lie algebra $\g{k}\oplus i\g{p}$ (see \cite[p.~390]{Kn02}). 

Choose a set $\Pi_\g{g}$ of simple roots for $\Delta_\g{g}$. The \emph{Vogan diagram} of the inner orthogonal symmetric pair $(\g{g},\g{k})$ with respect to $\g{t}$ and $\Pi_\g{g}$ is the Dynkin diagram of $\Pi_\g{g}$ with its nodes painted or not, depending on whether the corresponding simple root is noncompact or compact. 

There can be several Vogan diagrams corresponding to the same orthogonal symmetric pair. This redundancy is eliminated by the Borel-de Siebenthal theorem (see \cite[Th. 6.96]{Kn02}). For our purposes, what this result asserts is the following: given an irreducible inner compact orthogonal symmetric pair $(\g{g},\g{k})$ and given $\g{t}$ as above, there is a set of simple roots for $\Delta_\g{g}$ whose corresponding Vogan diagram has exactly one painted node. That is, we can assume that there is a set $\Pi_\g{g}$ of simple roots for $\Delta_\g{g}$ with precisely one noncompact root.
Furthermore, the set $\Delta_\g{k}$ of compact roots corresponds to the root system of $\g{k}$ and is a root subsystem of $\Delta_\g{g}$, whereas the noncompact roots are exactly the weights of the adjoint $\g{k}^\C$-representation on $\g{p}^\C$.

Applying the Borel-de Siebenthal theorem, we fix a set of simple roots $\Pi_\g{g}=\{\alpha_1,\ldots,\alpha_p\}$ for $\Delta_\g{g}$, where $\alpha_\nu$ is noncompact, for certain $\nu\in\{1,\ldots,p\}$, and the other simple roots are compact. Let $\{\h_1,\ldots,\h_p\}\subset \g{t}$ be the dual basis of $\Pi_\g{g}$.

Let $\mu=\sum_{i=1}^p y_i \alpha_i$ ($y_i\in\mathbb{N}$) be the highest root of $\Delta_\g{g}$ and put $\alpha_0=-\mu$. Then $G/K$ is Hermitian if and only if $y_\nu=1$; otherwise, $y_\nu=2$ (see Table~\ref{table:diagrams}). If $G/K$ is Hermitian, we will consider $\Pi_\g{k}=\{\alpha_1,\ldots,\alpha_{p}\}\setminus\{\alpha_\nu\}$ as a set of simple roots for $\Delta_\g{k}$. In this case, if $\g{t}'$ is the hyperplane of $\g{t}$ generated by the compact coroots, its normal space with respect to the Killing form $B_\g{g}$ is $\R\h_\nu$. If $G/K$ is not Hermitian, we will take $\Pi_\g{k}=\{\alpha_0,\alpha_1,\ldots,\alpha_p\}\setminus\{\alpha_\nu\}$ as a set of simple roots for $\Delta_\g{k}$. For a justification of these choices, see \cite[Ch.~8]{GG78}.

An enumeration of the roots in $\Delta_\g{g}$ shows that there is a unique highest noncompact root $\lambda$, in the following sense: $\lambda$ is the unique noncompact root that is greater than any other noncompact root, according to the lexicographic ordering defined by the set of simple roots $\Pi_\g{g}$. Notice that, if $\lambda=\sum_{i=1}^p z_i \alpha_i$, one has that $z_\nu=1$ and $z_i>0$ for every $i$ (see Table~\ref{table:diagrams}, cf.~\cite[Appx. C.1,2,4]{Kn02}).

\subsection{Complex structures preserving homogeneous polar foliations}\label{subsec:complex_structures_homogeneous}
Given a compact symmetric pair $(G,K)$ satisfying the maximality property, we will denote by $\mathcal{F}_{G/K}$ the orbit foliation of the $s$-representation of $G/K$ restricted to the unit sphere of $\g{p}$, and we will refer to it as the foliation determined by $G/K$ (or by $(G,K)$). The theory developed in Section~\ref{sec:foliations} applies to these foliations $\cal{F}_{G/K}$, where $V=\g{p}$ and $\rho=\Ad\colon K\to\O(\g{p})$.
The following result completely characterizes those $s$-representations whose orbit foliations can be projected to the corresponding complex projective space.

\begin{theorem}\label{th:reduction}
Let $(G,K)$ be a compact symmetric pair satisfying the maximality property, $\g{g}=\g{k}\oplus\g{p}$ its Cartan decomposition, and $\g{t}$ a maximal abelian subalgebra of $\g{k}$. 
Then there exists a complex structure in $\g{p}$ preserving $\mathcal{F}_{G/K}$ if and only if $G/K$ is inner.

In this situation, let $T\in\g{t}$. Then $\ad(T)\rvert_\g{p}$ is a complex structure in $\g{p}$ preserving $\mathcal{F}_{G/K}$ if and only if $\alpha(T)\in\{\pm 1\}$ for all (positive) noncompact roots $\alpha$ of $\g{g}$.
\end{theorem}

\begin{proof}
If $G/K$ is not inner, the centralizer of $\g{t}$ in $\g{g}$ is a maximal abelian subalgebra of $\g{g}$ of the form $\g{t}\oplus\g{a}$ with $0\neq \g{a}\subset\g{p}$. Then $0$ is a weight of the adjoint $\g{k}^\C$-representation $\rho_*^\C$, with weight space $\g{a}^\C$. By Proposition~\ref{prop:plus_minus_1}, $\cal{F}_{G/K}$ cannot be projected under any Hopf map.

Assume that $G/K$ is inner. The weights of $\rho_*^\C$ are the noncompact roots of $\g{g}$. Again by Proposition~\ref{prop:plus_minus_1}, $\ad(T)\rvert_\g{p}$ is a complex structure in $\g{p}$ if and only if $\alpha(T)\in\{\pm 1\}$ for all noncompact roots of $\g{g}$ or, equivalently, for all positive noncompact roots (for a fixed ordering of $\Delta_\g{g}$).
We still have to show that such $T\in\g{t}$ exists if $G/K$ is inner. According to Proposition \ref{prop:complex_structure}(iii), it is enough to show this for irreducible symmetric pairs $(G,K)$.

So let $G/K$ be inner and irreducible. 
If $\alpha,\beta\in\Delta_\g{g}$ are such that $\alpha+\beta\neq 0$, then $[\g{g}_\alpha,\g{g}_\beta]=\g{g}_{\alpha+\beta}$. Since $[\g{k},\g{k}]\subset \g{k}$, $[\g{k},\g{p}]\subset \g{p}$, $[\g{p},\g{p}]\subset \g{k}$, then $\alpha+\beta$ is compact if and only if $\alpha$ and $\beta$ are both compact or both noncompact. 

Hence, the positive noncompact roots are exactly those roots $\alpha=\sum_{j=1}^p m_j \alpha_j$, with odd $m_\nu$ (recall the notation introduced just before this subsection). 
But since the highest noncompact root $\lambda=\sum_{i=1}^p z_i \alpha_i$ satisfies $z_\nu=1$, for positive noncompact roots we always have $m_\nu=1$. Taking $T=\h_\nu$, then $\alpha(T)=1$ for every noncompact positive root, and hence $\ad(T)\rvert_\g{p}$ is a complex structure preserving $\cal{F}_{G/K}$.
\end{proof}

Fix now an irreducible inner compact symmetric pair $(G,K)$ satisfying the maximality property, and take a maximal abelian subalgebra $\g{t}$ of $\g{g}$ contained in $\g{k}$. As defined in \S\ref{subsec:congruence}, let $\cal{J}$ be the subset of those $X\in\g{k}$ such that $\rho_*(X)=\ad(X)\rvert_\g{p}$ is a complex structure on $\g{p}$, and let $\bar{C}\subset \g{t}$ be defined by the inequalities $\alpha\geq 0$, for every $\alpha\in\Pi_\g{k}$. We can now provide a complete description of the set $\cal{J}\cap \bar{C}$.

\begin{lemma}\label{lemma:integers}
In the above conditions, let $\mu=\sum_{i=1}^p y_i \alpha_i$ be the highest root, and $\lambda=\sum_{i=1}^p z_i \alpha_i$ the highest noncompact root. We have:
\begin{itemize}
\item[(i)] If $G/K$ is not Hermitian, then $\mathcal{J}\cap\bar{C}=\{-\h_\nu\}\cup\{-\h_\nu+2\h_i: i\in\mathcal{I}\}$, where $\mathcal{I}$ is the set of indices $i\in\{1,\ldots,p\}\setminus\{\nu\}$ such that $y_i=z_i=1$.
\item[(ii)] If $G/K$ is Hermitian, then $\mathcal{J}\cap\bar{C}=\{\pm\h_\nu\}\cup\{-\h_\nu+2\h_i: i\in\mathcal{I}\}$, where $\mathcal{I}$ is the set of indices $i\in\{1,\ldots,p\}\setminus\{\nu\}$ such that $y_i=1$.
\end{itemize}
\end{lemma}
\begin{proof}
Let us first prove some auxiliary results. For that, let $T=\sum_{i=1}^p x_i \h_i\in\mathcal{J}\cap\bar{C}$.

The condition $T\in\mathcal{J}$ implies that $\alpha(T)=\pm 1$ for all positive noncompact roots $\alpha$. Since $\alpha_\nu$ is noncompact, we get that $x_\nu=\alpha_\nu(T)=\pm 1$. 

As $T\in\bar{C}$ and $\lambda$ is noncompact, we must have $x_i=\alpha_i(T)\geq 0$ for all $i\in\{1,\ldots, p\}\setminus\{\nu\}$ and $\lambda(T)=\pm 1$. Hence $x_\nu=1$ implies that $x_i=0$ for every $i\in\{1,\ldots, p\}\setminus\{\nu\}$. 

From the information in \cite[Appx. C.1,2,4]{Kn02}), one can carry out a case-by-case analysis that shows the following. There exists a sequence of positive noncompact roots $\beta_1,\ldots,\beta_p$, with $\beta_1=\alpha_\nu$ and such that, if we express each $\beta_i$ as a linear combination of the simple roots, $\beta_i=\sum_{j=1}^p m_{ij} \alpha_j$, then each coefficient $m_{ij}$ is either $1$ or $0$, and the number of $1$-coefficients increases by one from $\beta_{i}$ to $\beta_{i+1}$, for every $i\in\{1,\ldots, p-1\}$ (in particular, the coefficients of $\beta_p$ are all $1$).

The assumptions $x_\nu=-1$, $x_i\geq 0$ for every $i\neq\nu$, and $\beta_i(T)=\pm 1$ for every $i$ imply then that $x_i=0$ for all $i\neq\nu$ except for at most one index $j$, for which $x_j=2$. 

Now we prove (i). Assume that $G/K$ is not Hermitian and, thus, $y_\nu=2$. Since in the non-Hermitian case $-\mu\in\Pi_\g{k}$, we get that $-\h_\nu\in\mathcal{J}\cap\bar{C}$ and $\h_\nu\notin\bar{C}$. As $\lambda$ is the highest noncompact root, we see that $-\h_\nu+2\h_i\in\mathcal{J}$ if and only if $z_i=1$. Moreover, $-\h_\nu+2\h_i\in\bar{C}$ if and only if $-y_\nu+2y_i\leq 0$. Since $y_\nu=2$, this condition is equivalent to $y_i=1$. 

In the Hermitian case $\pm\h_\nu\in\mathcal{J}\cap\bar{C}$. Since now $\lambda=\mu$ and the simple roots for $\Delta_\g{k}$ are just $\{\alpha_1,\ldots,\alpha_p\}\setminus\{\alpha_\nu\}$, then $-\h_\nu+2\h_i\in\mathcal{J}\cap\bar{C}$ if and only if $y_i=1$, and (ii) follows.
\end{proof}

\subsection{Congruence of the projections of homogeneous polar foliations}\label{subsec:congruence_homogeneous}
For the case of homogeneous polar foliations it is possible to refine the results of \S\ref{subsec:congruence}. This is the aim of this subsection. The criteria developed will be used in \S\ref{subsec:classification_homogeneous} to obtain the classification of isoparametric foliations on $\C P^n$ obtained from homogeneous polar foliations.

According to Proposition~\ref{prop:general_foliation} and Remark~\ref{rem:foliation_determines_space}, it is impossible that different compact orthogonal symmetric pairs satisfying the maximality property give rise to congruent foliations on a complex projective space. That is why we will focus on analyzing the congruence of foliations arisen from a fixed symmetric space.

Therefore, we fix a compact symmetric pair $(G,K)$ satisfying the maximality property and with Cartan decomposition $\g{g}=\g{k}\oplus\g{p}$. In view of Theorem \ref{th:reduction}, we can assume that $G/K$ is inner. We also fix a maximal abelian subalgebra $\g{t}$ of $\g{g}$ contained in $\g{k}$, and we let $\cal{J}$, $\sim$, and $\bar{C}$ be as in \S\ref{subsec:congruence}. Our aim is to determine the $\sim$-equivalence classes of $\mathcal{J}$. 

Theorem \ref{th:automPolar} implies that $\Aut(\g{k},\cal{F})$ is precisely the group $\Aut(\g{g},\g{k})$ of automorphisms of $\g{g}$ that restrict to automorphisms of $\g{k}$. Therefore, Proposition~\ref{prop:invariant_torus} now reads as follows.

\begin{proposition}
Let $T_1,T_2\in\mathcal{J}\cap\g{t}$. Then $T_1\sim T_2$ if and only if there exists an automorphism $\varphi\in \Aut(\g{g},\g{k})$ leaving $\g{t}$ invariant and such that $\varphi T_1=\pm T_2$.
\end{proposition}

The negative $-B_\g{g}$ of the Killing form of $\g{g}$ is a positive definite $\Aut(\g{g},\g{k})$-invariant inner product on $\g{g}$, so it can play the role of the inner product $\langle\cdot,\cdot\rangle$ considered in \S\ref{subsec:congruence}.
The set $\Delta_V=\Delta_\g{p}$ of weights of the adjoint $\g{k}^\C$-representation on $\g{p}^\C$ is precisely $\Delta_\g{g}\setminus\Delta_\g{k}$. Hence, the group $\Aut(\Delta_\g{k},\Delta_V)$ defined in \S\ref{subsec:congruence} is now the group of automorphisms of the root system $\Delta_\g{g}$ that are automorphisms of the root subsystem $\Delta_\g{k}$. In this section, we denote this group by $\Aut(\Delta_\g{g},\Delta_\g{k})$. Then we have:

\begin{proposition}
The restriction to $\g{t}$ of every element of $\Aut(\g{g},\g{k})$ preserving $\g{t}$ yields an element of $\Aut(\Delta_\g{g},\Delta_\g{k})$. Conversely, every element of $\Aut(\Delta_\g{g},\Delta_\g{k})$ can be extended to an element of $\Aut(\g{g},\g{k})$ preserving $\g{t}$.
\end{proposition}
\begin{proof}
The first claim follows from Proposition~\ref{prop:auto_roots}.
For the converse, let $\varphi\in\Aut(\Delta_\g{g},\Delta_\g{k})$. The second assertion in \cite[Ch.~IX, Th.~5.1]{He78} affirms that $\varphi\in\Aut(\Delta_\g{g})$ can be extended to an automorphism $\wt{\varphi}$ of $\g{g}$. Let $\phi=\wt{\varphi}^\C$. Arguing as in Proposition~\ref{prop:auto_roots}, if $\alpha\in\Delta_\g{k}$, then $\phi(\g{g}_\alpha)=\g{g}_\beta$ and $\varphi H_\alpha =\phi H_\alpha =H_\beta$, where $\beta=\alpha\circ\phi^{-1}\rvert_{\g{t}^\C}\in\Delta_\g{g}$. Since $\varphi$ sends compact coroots to compact coroots, we have that $\beta\in\Delta_\g{k}$. Hence $\phi(\g{g}_\alpha)\subset\g{k}^\C$ for every $\alpha\in\Delta_\g{k}$. Since $\phi(\g{t}^\C)=\g{t}^\C\subset \g{k}^\C$ as well, we get $\phi(\g{k}^\C)=\g{k}^\C$ and, due to the invariance of $\g{g}$ under $\phi$, we have that $\phi(\g{k})=\g{k}$. Therefore, $\wt{\varphi}$ is the desired extension of $\varphi$.
\end{proof}

As if $\varphi\in \Aut(\Delta_\g{g},\Delta_\g{k})$, then also $-\varphi\in \Aut(\Delta_\g{g},\Delta_\g{k})$, the previous two propositions imply:

\begin{corollary}\label{cor:congruence}
Let $T_1,T_2\in\mathcal{J}\cap\g{t}$. Then $T_1\sim T_2$ if and only if there exists $\varphi\in \Aut(\Delta_\g{g},\Delta_\g{k})$ such that $\varphi T_1= T_2$.
\end{corollary}

Henceforth we will further assume that the compact inner symmetric pair $(G,K)$ satisfying the maximality property is irreducible. The classification of the complex structures preserving  foliations induced by reducible symmetric spaces follows from the classification of the irreducible case, in view of Proposition~\ref{prop:complex_structure}(iii) and Corollary~\ref{cor:congruence}.

We will denote by $\Out(\Delta_\g{g},\Delta_\g{k})$ the subgroup of $\Aut(\Delta_\g{g},\Delta_\g{k})$ of  automorphisms leaving $\bar{C}$ invariant, or equivalently, leaving invariant the simple compact coroots. Now the groups $\Out(\Delta_\g{k},\Delta_V)$ and $\Out^\pm_\cal{F}(\Delta_\g{k},\Delta_V)$ introduced in \S\ref{subsec:congruence} are exactly $\Out(\Delta_\g{g},\Delta_\g{k})$. We have:

\begin{proposition}\label{prop:criterion_Out}
Let $T_1,T_2\in\mathcal{J}\cap\bar{C}$. Then $T_1\sim T_2$ if and only if there exists $\varphi\in\Out(\Delta_\g{g},\Delta_\g{k})$ such that $\varphi (T_1)=T_2$. If moreover $G/K$ is Hermitian, we have:
\begin{itemize}
\item[(i)]If $T_1\notin (\g{t}')^\perp$ and $T_2\in(\g{t}')^\perp$, then $T_1\nsim T_2$. 
\item[(ii)]If $T_1,T_2\in(\g{t}')^\perp$, then $T_1\sim T_2$.
\end{itemize}
\end{proposition}

\begin{proof}
The first claim is a rewriting of Proposition~\ref{prop:criterion_Out_F}.
Let $G/K$ be Hermitian. The fact that every element of $\Aut(\Delta_\g{g},\Delta_\g{k})$ is an orthogonal transformation of $\g{t}$ which leaves $\g{t}'$ invariant implies (i). Since $-\Id_\g{t}\in\Aut(\Delta_\g{g},\Delta_\g{k})$, and since the intersection of $\mathcal{J}$ with each $1$-dimensional subspace of $\g{t}$ is either empty or a pair of opposite vectors, we obtain (ii).
\end{proof}

We need now to introduce an important notion for our work. First recall that the extended Dynkin diagram of $\g{g}$ is the Dynkin diagram of $\g{g}$ together with the extra node $\alpha_0$, which is joined to the other nodes according to the usual rules. Thus, we define the \emph{extended Vogan diagram} of $(\g{g},\g{k})$ as the extended Dynkin diagram of $\g{g}$, where the nodes corresponding to noncompact roots are painted while the other nodes remain unpainted. This definition depends in principle on the maximally compact abelian subalgebra $\g{t}$ and on the chosen set of simple roots $\Pi_\g{g}$. However, by the Borel-de Siebenthal theorem and the choices made at the beginning of this section, we can and will assume that every extended Vogan diagram has either exactly one or exactly two painted nodes ($\alpha_\nu$ and, maybe, $\alpha_0$). The first case happens when $G/K$ is not Hermitian and hence the adjoint $\g{k}^\C$-representation on $\g{p}^\C$ is irreducible. The second case occurs when $G/K$ is Hermitian, so the adjoint $\g{k}^\C$-representation on $\g{p}^\C$ decomposes into the sum of two irreducible representations. In both cases, the roots corresponding to the painted nodes in the extended Vogan diagram (that is, the roots in $\Pi_\g{g}\setminus\Pi_\g{k}$) are exactly the lowest weights of the adjoint $\g{k}^\C$-representation on $\g{p}^\C$. Therefore, extended Vogan diagrams represent a very particular case of lowest weight diagrams. The extended Vogan diagrams of irreducible inner symmetric spaces can be obtained from Table~\ref{table:diagrams}. Analogously as in \S\ref{subsec:congruence}, we define an automorphism of an extended Vogan diagram as a permutation of its nodes preserving the graph and the painted nodes. For details, references, and recent applications of extended Vogan diagrams, see \cite{CH10}.

We can now improve Proposition~\ref{prop:auto_diagram} for $s$-representations of inner symmetric spaces.

\begin{proposition}\label{prop:autoEVD}
Every automorphism in $\Out(\Delta_\g{g},\Delta_\g{k})$ determines an automorphism of the extended Vogan diagram of $(\g{g},\g{k})$ in a unique natural way, and conversely.
\end{proposition}
\begin{proof}
The first claim follows directly from Proposition~\ref{prop:auto_diagram}.
Let us show the converse. Every automorphism of the extended Dynkin diagram of $\g{g}$ defines an automorphism $\varphi\in\Aut(\Delta_\g{g})$ (see \cite[Ch.~VII, Prop.~1.4(a)]{Lo69}). Since every automorphism of the extended Vogan diagram of $(\g{g},\g{k})$ preserves the unpainted nodes, the induced automorphism $\varphi\in\Aut(\Delta_\g{g})$ leaves invariant the simple compact coroots, and hence $\varphi\in\Out(\Delta_\g{g},\Delta_\g{k})$.
\end{proof}

We finish this subsection by proving a quite useful technical lemma.

\begin{lemma}\label{lemma:sigma}
Let $\mu=\sum_{i=1}^p y_i \alpha_i$ be the highest root. Given $\varphi\in\Out(\Delta_\g{g},\Delta_\g{k})$, let $\sigma$ be the permutation of the set of indices $\{0,1,\ldots,p\}$ that defines the automorphism of the extended Vogan diagram associated to $\varphi$. We have:
\begin{itemize}
\item[(i)] If $\sigma(0)=0$, then $\varphi(\h_i)=\h_{\sigma(i)}$ for all $i\in\{1,\ldots,p\}$.
\item[(ii)] If $\sigma(0)=\nu$, then $\varphi(\h_i)=\h_{\sigma(i)}-y_i\h_\nu$ for $i\in\{1,\ldots,p\}\setminus\{\nu\}$, and $\varphi(\h_\nu)=-\h_\nu$.
\item[(iii)] If $\sigma$ interchanges $0$ and $k$, with $k\in\{1,\ldots,p\}\setminus\{\nu\}$, then $\varphi(\h_i)=\h_{\sigma(i)}-y_i\h_k$ for $i\in\{1,\ldots,n\}\setminus\{k\}$, and $\varphi(\h_k)=-y_k\h_k$.
\end{itemize}
\end{lemma}
\begin{proof}
Fix $i\in\{1,\ldots,p\}$. For every $j\in\{0,1,\ldots,p\}$ we have:
\begin{align*}
\alpha_j(\varphi(\h_i))=B_\g{g}(H_{\alpha_j},\varphi(\h_i))=B_\g{g}(\varphi^{-1}(H_{\alpha_j}),\h_i)=B_\g{g}(H_{\alpha_{\sigma^{-1}(j)}}, \h_i)= 
\alpha_{\sigma^{-1}(j)}(\h_i),
\end{align*}
which is equal to the Kronecker delta $\delta_{\sigma(i),j}$ if $\sigma^{-1}(j)\neq 0$, and is equal to $-y_i$ if $\sigma^{-1}(j)=0$.

Assume that $\sigma(0)=0$. Then $\sigma$ leaves $\{1,\ldots,p\}$ invariant, and hence $\alpha_j(\varphi(\h_i))=\delta_{\sigma(i),j}$ for all $j\in\{1,\ldots, p\}$. Then (i) follows.

Assume now that $\sigma(0)=\nu$. Then $\alpha_0$ is noncompact, $\sigma$ interchanges $0$ and $\nu$ and preserves $\{1,\ldots,p\}\setminus\{\nu\}$. Consider first that $i\neq \nu$. Then $\alpha_\nu(\varphi(\h_i))=\alpha_0(\h_i)=-y_i$, $\alpha_{\sigma(i)}(\varphi(\h_i))=1$, and $\alpha_j(\varphi(\h_i))=0$ if $j\in\{1,\ldots,p\}\setminus\{\nu,\sigma(i)\}$. Therefore, if $i\neq \nu$, then $\varphi(\h_i)=\h_{\sigma(i)}-y_i\h_\nu$. Since $\alpha_j(\varphi(\h_\nu))=0$ for all $j\in\{1,\ldots,p\}\setminus\{\nu\}$, and $\alpha_\nu(\varphi(\h_\nu))=-y_\nu=-1$, we get $\varphi(\h_\nu)=-\h_\nu$.

Finally let $\sigma$ be as in (iii). Then $\alpha_0$ is compact and hence $\sigma(\nu)=\nu$. If $i\neq k$, we have that $\alpha_{\sigma(i)}(\varphi(\h_i))=1$, $\alpha_k(\varphi(\h_i))=-y_i$, and $\alpha_j(\varphi(\h_i))=0$ for all $j\in\{1,\ldots,p\}\setminus\{k,\sigma(i)\}$. It follows $\varphi(\h_i)=\h_{\sigma(i)}-y_i\h_k$ if $i\neq k$. Since $\alpha_k(\varphi(\h_k))=-y_k$, and $\alpha_j(\varphi(\h_k))=0$ for all $j\in\{1,\ldots,p\}\setminus\{k\}$, we obtain $\varphi(\h_k)=-y_k\h_k$.
\end{proof}

\subsection{Classification of the complex structures}\label{subsec:classification_homogeneous}

We are now in position to get the case-by-case classification of the complex structures that preserve homogeneous polar foliations, up to congruence of the projected foliations on the complex projective space. As in the previous subsection, we will consider compact irreducible inner symmetric pairs $(G,K)$ satisfying the maximality property. 
According to the results above in this section, the set $\mathcal{J}\cap\bar{C}$ and the cardinality $N=N(\cal{F}_{G/K})$ of the quotient $(\mathcal{J}\cap\bar{C})/\!\sim$ can be calculated from the following data: $\mu$, $\lambda$, and the symmetries of the extended Vogan diagram. All this information can be extracted from Table~\ref{table:diagrams}, where the resulting value of $N$ is also shown. 

In each case, we begin by indicating the corresponding orthogonal symmetric pair $(\g{g},\g{k})$ and the possible values of $p$ and $\nu$. Then we specify the Hermitian or non-Hermitian character and the set $\mathcal{J}\cap\bar{C}$. If needed, we give a set of generators of $\Out(\Delta_\g{g},\Delta_\g{k})$ (defined by means of symmetries of the extended Vogan diagram), and their action on (maybe only some elements of) $\mathcal{J}\cap\bar{C}$. Finally, we specify the value of $N$.

\subsection*{Type \textbf{A III}: $(\g{su}(p+1),\g{s}(\g{u}(\nu)\oplus\g{u}(p-\nu+1)))$, \quad $p\geq 3$, \quad $2\leq\nu\leq p-1$.} 
\begin{itemize}
\item Hermitian.

\item $\mathcal{J}\cap\bar{C}=\{\pm\h_\nu\}\cup\{-\h_\nu+2\h_i: i=1,\ldots,p; i\neq \nu\}$.

\item Generators of $\Out(\Delta_\g{g},\Delta_\g{k})$:
	\begin{itemize}
	\item[$\circ$] $\varphi_1$: $\alpha_i\leftrightarrow\alpha_{\nu-i}$ for all $i=0,\ldots, \nu$, and $\alpha_i\leftrightarrow 
	       \alpha_{p+\nu-i+1}$ for all $i=\nu+1,\ldots, p$.
	\item[$\circ$] $\varphi_2$ (only if $2\nu=p+1$): $\alpha_i\leftrightarrow \alpha_{p-i+1}$ for $i=1,\ldots,p$, and fixes $\alpha_0$ and $\alpha_\nu$.
	\end{itemize}
	
\item Action on $\mathcal{J}\cap\bar{C}$:
	\begin{itemize}
	\item[$\circ$] $\varphi_1$: $\h_\nu\leftrightarrow-\h_\nu$,     
	               $-\h_\nu+2\h_i\leftrightarrow-\h_\nu+2\h_{\nu-i}$ for $i=1,\ldots, \nu-1$, and 
	               $-\h_\nu+2\h_i\leftrightarrow-\h_\nu+2\h_{p+\nu-i+1}$ for $i=\nu+1,\ldots, p$.
	\item[$\circ$] $\varphi_2$ (only if $2\nu=p+1$): $-\h_\nu+2\h_i\leftrightarrow -\h_\nu+2\h_{p-i+1}$ for $i\in\{1,\ldots,p\}\setminus\{\nu\}$, and fixes $\pm\h_\nu$.
	\end{itemize}
	
\item $N=1+\left[\frac{\nu}{2}\right]+\left[\frac{p-\nu+1}{2}\right]$ if $2\nu\neq p+1$, and $N=1+\left[\frac{\nu}{2}\right]$ if $2\nu=p+1$ (where $[\cdot]$ denotes the integer part of a real number).
\end{itemize}

\subsection*{Type \textbf{B I}: $(\g{so}(2p+1),\g{so}(2\nu)\oplus\g{so}(2p-2\nu+1))$, \quad $p\geq 3$.}

\begin{itemize}
\item Hermitian if and only if $\nu=1$.

\item $\mathcal{J}\cap\bar{C}=\{\pm\h_1\}$ if $\nu=1$, and $\mathcal{J}\cap\bar{C} = \{-\h_\nu, -\h_\nu +        2\h_1\}$ otherwise.

\item Generator of $\Out(\Delta_\g{g},\Delta_\g{k})$: $\varphi$ interchanges $\alpha_0\leftrightarrow\alpha_1$, and fixes $\alpha_i$ for all $i\geq 2$.
	
\item Action on $\mathcal{J}\cap\bar{C}$: $\varphi$ interchanges $-\h_\nu\leftrightarrow-\h_\nu+2\h_1$ if $\nu\neq 1$.

\item $N=1$.
\end{itemize}
By the maximality property, the rank-one case is of type \textbf{B I}, with $\nu=p\geq 1$. Then $N=1$ holds trivially.

\subsection*{Type \textbf{C I}: $(\g{sp}(p),\g{u}(p))$, \quad $\nu=p\geq 2$.}
\begin{itemize}
\item 
\begin{tabular}{@{}b{\bulletspacest}@{}b{\bulletspacend}@{}b{\bulletspacerd}@{}}
Hermitian. & $\bullet\;\,\mathcal{J}\cap\bar{C}=\{\pm\h_\nu\}$. & $\bullet\;\,N=1$.
\end{tabular}
\end{itemize}

\subsection*{Type \textbf{C II}: $(\g{sp}(p),\g{sp}(\nu)\oplus\g{sp}(p-\nu))$, \quad $p\geq 4$, \quad $2\leq \nu \leq p-2$.}
\begin{itemize}
\item 
\begin{tabular}{@{}b{0.5\linewidth}@{}b{0.5\linewidth}@{}}
Non-Hermitian. & $\bullet\;\,\mathcal{J}\cap\bar{C}= \{-\h_\nu,-\h_\nu+2\h_p\}$. 
\end{tabular}

\item Generator of $\Out(\Delta_\g{g},\Delta_\g{k})$: $\varphi$ (only if $2\nu=p$) interchanges $\alpha_i\leftrightarrow \alpha_{p-i}$ for all $i\in\{0,\ldots,p\}$.
	
\item Action on $\mathcal{J}\cap\bar{C}$: $\varphi$ (only if $2\nu=p$) interchanges $-\h_\nu\leftrightarrow -\h_\nu+2\h_p$.

\item $N=1$ if $2\nu=p$, and $N=2$ otherwise.
\end{itemize}

\subsection*{Type \textbf{D I}: $(\g{so}(2p),\g{so}(2\nu)\oplus\g{so}(2p-2\nu))$, \quad $p\geq 4$, \quad $\nu\leq p-2$.}
\begin{itemize}
\item Hermitian if and only if $\nu=1$.

\item $\mathcal{J}\cap\bar{C}=\{\pm\h_1,-\h_1+2\h_{p-1},-\h_1+2\h_{p}\}$ if $\nu=1$, and   
      $\mathcal{J}\cap\bar{C}=\{-\h_\nu,-\h_\nu+2\h_1,-\h_\nu+2\h_{p-1},-\h_\nu+2\h_{p}\}$       otherwise.

\item Generators of $\Out(\Delta_\g{g},\Delta_\g{k})$:
	\begin{itemize}
	\item[$\circ$] $\varphi_1$: $\alpha_{p-1}\leftrightarrow\alpha_p$ and fixes $\alpha_i$ for all 
	               $i\in\{0,\ldots,  p-1\}$.
	\item[$\circ$] $\varphi_2$: $\alpha_0\leftrightarrow\alpha_1$ and fixes $\alpha_i$ for all 
	               $i\in\{2,\ldots, p\}$.
	\item[$\circ$] $\varphi_3$ (only if $2\nu=p$): $\alpha_i\leftrightarrow \alpha_{p-i}$ for all $i\in\{0,\ldots, p\}$.
	\end{itemize}
	
\item Action on $\mathcal{J}\cap\bar{C}$:
	\begin{itemize}
	\item[$\circ$] $\varphi_1$: $-\h_\nu+2\h_{p-1}\leftrightarrow-\h_\nu+2\h_{p}$ and fixes the other                     elements.
	\item[$\circ$] $\varphi_2$: $-\h_\nu\leftrightarrow-\h_\nu+2\h_{1}$ and fixes the other elements, if $\nu\neq 1$.
	\item[$\circ$] $\varphi_3$ (only if $2\nu=p$): $-\h_\nu\leftrightarrow -\h_\nu+2\h_p$.
	\end{itemize}

\item $N=1$ if $2\nu=p$, and $N=2$ otherwise.
\end{itemize}

\subsection*{Type \textbf{D III}: $(\g{so}(2p),\g{u}(p))$, \quad $\nu=p\geq 4$.}
\begin{itemize}
\item 
\begin{tabular}{@{}b{0.5\linewidth}@{}b{0.5\linewidth}@{}}
Hermitian. & $\bullet\;\,\mathcal{J}\cap\bar{C}= \{\pm\h_p,-\h_p+2\h_1,-\h_p+2\h_{p-1}\}$. 
\end{tabular}

\item Generator of $\Out(\Delta_\g{g},\Delta_\g{k})$: $\varphi$ interchanges $\alpha_i\leftrightarrow \alpha_{p-i}$ for all $i\in\{0,\ldots,  p\}$.
	
\item Action on $\mathcal{J}\cap\bar{C}$:  $\varphi$ interchanges $-\h_p+2\h_1\leftrightarrow-\h_p+2\h_{p-1}$.

\item $N=2$.
\end{itemize}

\subsection*{Type \textbf{E II}: $(\g{e}_6,\g{su}(6)\oplus\g{su}(2))$, \quad $\nu=2$.}
\begin{itemize}
\item 
\begin{tabular}{@{}b{0.5\linewidth}@{}b{0.5\linewidth}@{}}
Non-Hermitian. & $\bullet\;\,\mathcal{J}\cap\bar{C}= \{-\h_2, -\h_2+2\h_1,-\h_2+2\h_6\}$. 
\end{tabular}

\item Generator of $\Out(\Delta_\g{g},\Delta_\g{k})$: $\varphi$ switches $\alpha_1\leftrightarrow \alpha_6$, $\alpha_3\leftrightarrow \alpha_5$, and $\alpha_0$, $\alpha_2$, $\alpha_4$ stay fixed.
	
\item Action on $\mathcal{J}\cap\bar{C}$: $\varphi$ interchanges $-\h_2+2\h_1\leftrightarrow -\h_2+2\h_6$, and fixes $-\h_2$.

\item $N=2$.
\end{itemize}

\subsection*{Type \textbf{E III}: $(\g{e}_6,\g{so}(10)\oplus\g{so}(2))$, \quad $\nu=6$.}
\begin{itemize}
\item 
\begin{tabular}{@{}b{\bulletspacest}@{}b{0.45\linewidth}@{}b{0.2\linewidth}@{}}
Hermitian. & $\bullet\;\,\mathcal{J}\cap\bar{C}=\{\pm\h_6,-\h_6+2\h_1\}$. & $\bullet\;\,N=2$.
\end{tabular}
\end{itemize}

\subsection*{Type \textbf{E V}: $(\g{e}_7,\g{su}(8))$, \quad $\nu=2$.}
\begin{itemize}
\item 
\begin{tabular}{@{}b{0.5\linewidth}@{}b{0.5\linewidth}@{}}
Non-Hermitian. & $\bullet\;\,\mathcal{J}\cap\bar{C}= \{ -\h_2,-\h_2+2\h_7\}$. 
\end{tabular}

\item Generator of $\Out(\Delta_\g{g},\Delta_\g{k})$: $\varphi$ switches $\alpha_1\!\leftrightarrow \alpha_6$, $\alpha_3\!\leftrightarrow \alpha_5$, $\alpha_0\!\leftrightarrow \alpha_7$, and fixes $\alpha_2$,~$\alpha_4$.
	
\item Action on $\mathcal{J}\cap\bar{C}$: $\varphi$ interchanges $-\h_2\leftrightarrow -\h_2+2\h_7$.

\item $N=1$.
\end{itemize}
 
\subsection*{Type \textbf{E VI}: $(\g{e}_7,\g{so}(12)\oplus \g{su}(2))$, \quad $\nu=1$.}
\begin{itemize}
\item 
\begin{tabular}{@{}b{0.5\linewidth}@{}b{0.5\linewidth}@{}}
Non-Hermitian. & $\bullet\;\,\mathcal{J}\cap\bar{C}= \{-\h_1,-\h_1+2\h_7\}$. 
\end{tabular}
\item
\begin{tabular}{@{}b{0.5\linewidth}@{}b{0.5\linewidth}@{}}
$\Out(\Delta_\g{g},\Delta_\g{k})$ is trivial. & $\bullet\;\,N=2$.
\end{tabular}
\end{itemize}

\subsection*{Type \textbf{E VII}: $(\g{e}_7,\g{e}_6\oplus\g{so}(2))$, \quad $\nu=7$.}
\begin{itemize}
\item 
\begin{tabular}{@{}b{\bulletspacest}@{}b{\bulletspacend}@{}b{\bulletspacerd}@{}}
Hermitian. & $\bullet\;\,\mathcal{J}\cap \bar{C}=\{\pm\h_7\}$. & $\bullet\;\,N=1$.
\end{tabular}
\end{itemize}

\subsection*{Type \textbf{E VIII}: $(\g{e}_8,\g{so}(16))$, \quad $\nu=1$.}
\begin{itemize}
\item 
\begin{tabular}{@{}b{\bulletspacest}@{}b{\bulletspacend}@{}b{\bulletspacerd}@{}}
Non-Hermitian. & $\bullet\;\,\mathcal{J}\cap\bar{C}=\{-\h_1\}$. & $\bullet\;\,N=1$.
\end{tabular}
\end{itemize}

\subsection*{Type \textbf{E IX}: $(\g{e}_8,\g{e}_7\oplus\g{su}(2))$, \quad $\nu=8$.}
\begin{itemize}
\item 
\begin{tabular}{@{}b{\bulletspacest}@{}b{\bulletspacend}@{}b{\bulletspacerd}@{}}
Non-Hermitian. & $\bullet\;\,\mathcal{J}\cap\bar{C}=\{-\h_8\}$. & $\bullet\;\,N=1$.
\end{tabular}
\end{itemize}

\subsection*{Type \textbf{F I}: $(\g{f}_4,\g{sp}(3)\oplus\g{su}(2))$, \quad $\nu=4$.}
\begin{itemize}
\item 
\begin{tabular}{@{}b{\bulletspacest}@{}b{\bulletspacend}@{}b{\bulletspacerd}@{}}
Non-Hermitian. & $\bullet\;\,\mathcal{J}\cap\bar{C}=\{-\h_4\}$. & $\bullet\;\,N=1$.
\end{tabular}
\end{itemize}

\subsection*{Type \textbf{G}: $(\g{g}_2,\g{su}(2)\oplus\g{su}(2))$, \quad $\nu=2$.}
\begin{itemize}
\item 
\begin{tabular}{@{}b{\bulletspacest}@{}b{\bulletspacend}@{}b{\bulletspacerd}@{}}
Non-Hermitian. & $\bullet\;\,\mathcal{J}\cap\bar{C}=\{-\h_2\}$. & $\bullet\;\,N=1$.
\end{tabular}
\end{itemize}

\section{Projecting FKM-foliations}\label{sec:FKM}
The purpose of this section is analogous to that of Section~\ref{sec:homogeneous}, but here we will deal with FKM-foliations instead of homogeneous polar foliations. Our goal will be to classify the complex structures preserving FKM-foliations under the hypothesis $m_1\leq m_2$. For the notation concerning FKM-foliations, we refer the reader to \S\ref{subsec:auto_FKM}.

In view of Theorem~\ref{th:reduction}, among all homogeneous polar foliations, only those arisen from inner symmetric spaces descend to the corresponding complex projective spaces. However, the behaviour of FKM-foliations is different: all of them can be projected. The idea behind the proof of this fact is very simple and already appeared in the original paper \cite[\S6.2]{FKM81}.

\begin{theorem}
Each FKM-foliation $\cal{F}_\cal{P}$ admits a complex structure preserving $\cal{F}_\cal{P}$. 
\end{theorem}
\begin{proof}
Let $(P_0,\dots,P_m)$ be a Clifford system defining $\cal{F}_\cal{P}\subset S^{2n+1}$. Let $F$ be its Cartan-M\"unzner polynomial. Define $J=P_0P_1$, which is a complex structure on $\R^{2n+2}$. Then
\[
F(\cos (t)x +\sin (t)Jx)=F\left((\cos(t)P_1+\sin(t)P_0)P_1x\right)=F(P_1x)=F(x)
\]
for all $x\in\R^{2n+2}$, since $F$ is invariant under $S(\cal{P})$. Hence, $J$ preserves $\cal{F}_\cal{P}$.
\end{proof}

Let $(P_1,\dots,P_m)$ be a Clifford system on $V=\R^{2n+2}$ defining $\cal{F}_\cal{P}$ and satisfying $m_1\leq m_2$. In view of Theorems~\ref{th:auto_FKM_1} and~\ref{th:auto_FKM_2}, there is an effective representation $\rho\colon K\to\O(V)$ such that $\rho(K)$ is the maximal connected group of automorphisms of $\cal{F}_\cal{P}$ preserving the leaves (in fact, $\rho(K)$ is the identity connected component of $\Aut(\cal{F}_\cal{P})$), $K=\Spin(m+1)\cdot H$ is a direct product modulo a finite subgroup, and $H$ is the identity connected component of the corresponding group in Theorem~\ref{th:U+}. The explicit description of $\rho$ as a tensor product of the spin representation of $\Spin(m+1)$ and the standard representation of the classical group $H$ is given in Theorem~\ref{th:auto_FKM_2}.  

In this situation, the results of Section~\ref{sec:foliations} are applicable. Thus, fix a maximal abelian subalgebra $\g{t}$ of the compact Lie algebra $\g{k}$. Let $\Delta_\g{k}$ be the set of roots of $\g{k}$ with respect to $\g{t}$, and $\Delta_V$ the set of weights of $\rho_*^\C$.
We have that $\g{k}=\g{so}(m+1)\oplus\g{h}$, where $\g{h}$ is equal to
\[
\begin{array}{rl@{\qquad\qquad}rl}
\g{so}(k), \qquad & \text{if } m\equiv 1,7\,(\mod 8), 
\\
\g{u}(k), \qquad & \text{if } m\equiv 2,6\,(\mod 8), & \g{so}(k_+)\oplus \g{so}(k_-), \qquad & \text{if } m\equiv 0\,(\mod 8),
\\
\g{sp}(k), \qquad & \text{if } m\equiv 3,5\,(\mod 8), & \g{sp}(k_+)\oplus \g{sp}(k_-), \qquad & \text{if } m\equiv 4\,(\mod 8).
\end{array}
\]
We present now some well-known information about the roots of $\g{k}$ and the weights of $\rho_*^\C$. It can be obtained for example from \cite[Ch. IX]{Si96} (cf. \cite[p.~683--685]{Kn02}). 

Let $p$, $q$ be the ranks of the Lie algebras $\g{so}(m+1)$ and $\g{h}\in\{\g{so}(k),\g{u}(k),\g{sp}(k)\}$, respectively, and $q_\pm$ the rank of $\g{so}(k_\pm)$ or $\g{sp}(k_\pm)$ as appropriate. Set $\g{t}=\g{t}_s\oplus\g{t}_\g{h}$, where $\g{t}_s$ and $\g{t}_\g{h}$ are maximal abelian subalgebras of $\g{so}(m+1)$ and $\g{h}$, respectively. Some of the following assertions may need a rescaling of the inner product on the irreducible factors of $\g{k}$.

Let $\{\alpha^s_1,\dots,\alpha^s_p\}$ and $\{\alpha_1,\dots,\alpha_q\}$ be systems of simple roots for $\g{so}(m+1)$ and $\g{h}\in\{\g{so}(k),\g{u}(k),\g{sp}(k)\}$, respectively. There is an orthonormal basis $\{\w^s_1,\dots,\w^s_p\}$ of $\g{t}_s^*$ so that $\alpha^s_i=\w^s_i-\w^s_{i+1}$ for $i=1,\dots, p-1$, and $\alpha^s_p=\w^s_p$ if $m$ is even, or $\alpha^s_p=\w^s_{p-1}+\w^s_p$ if $m$ is odd. The weights of the spin representation $\rho_s$ of $\g{so}(m+1)$ are $\frac{1}{2}(\pm \w^s_1\pm\dots\pm \w^s_p)$.

If $\g{h}=\g{so}(k)$, there is an orthonormal basis $\{\w_1,\dots,\w_q\}$ of $\g{t}_\g{h}^*$ such that $\alpha_i=\w_i-\w_{i+1}$ for $i=1,\dots, q-1$, and $\alpha_q=\w_q$ if $k$ is odd, or $\alpha_q=\w_{q-1}+\w_q$ if $k$ is even. The weights of the standard representation $\rho_{\g{so}(k)}$ of $\g{so}(k)$ are $\pm \w_1,\dots,\pm\w_q$ if $k$ is even, or $\pm \w_1,\dots,\pm\w_q,0$ if $k$ is odd.
Then $\rho_*^\C\cong\rho_s\otimes_\C\rho_{\g{so}(k)}$, so its weights are all possible sums of weights of $\rho_s$ with weights of $\rho_{\g{so}(k)}$.

If $\g{h}=\g{u}(k)$, then $k=q$. There exists an orthonormal basis $\{\w_1,\dots,\w_k\}$ of $\g{t}_\g{h}^*$ such that $\alpha_i=\w_i-\w_{i+1}$ for $i=1,\dots, k-1$, and such that the weights of the standard representation $\rho_{\g{u}(k)}$ of $\g{u}(k)$ are $\w_1,\dots,\w_k$. Let $(\cdot)^\R$ denote realification. Since $\rho_*\cong\eta^\R$ with $\eta=\rho_s\otimes_\C\rho_{\g{u}(k)}$, then $\rho_*^\C\cong(\eta^\R)^\C\cong\eta\oplus\bar{\eta}$, where $\bar{\eta}$ stands for the complex conjugate or contragredient representation of $\eta$. The weights of $\rho_*^\C$ are then all possible sums of weights of $\rho_s$ with weights of $\rho_{\g{u}(k)}$, and the negatives of these sums. 

If $\g{h}=\g{sp}(k)$, then $k=q$ and there is an orthonormal basis $\{\w_1,\dots,\w_k\}$ of $\g{t}_\g{h}^*$ such that $\alpha_i=\w_i-\w_{i+1}$ for $i=1,\dots, k-1$, and $\alpha_k=2\w_k$. The weights of the standard representation $\rho_{\g{sp}(k)}$ are $\pm \w_1,\dots,\pm\w_k$.
In this case, $\rho_*^\C$ is equivalent to $\rho_s\otimes_\C\rho_{\g{sp}(k)}$, so its weights are the sums of weights of $\rho_s$ with weights of $\rho_{\g{sp}(k)}$.

If $\g{h}\in\{\g{so}(k_+)\oplus \g{so}(k_-),\g{sp}(k_+)\oplus \g{sp}(k_-)\}$, one can consider a basis $\{\w_1^+,\dots,\w_{q_+}^+\}\cup\{\w_1^-,\dots,\w_{q_-}^-\}$ of $\g{t}_\g{h}^*$ so that one can express the roots of $\g{h}$ and the weights of $\rho_{\g{so}(k_\pm)}$ or $\rho_{\g{sp}(k_\pm)}$ in terms of the $\w_i^\pm$ in a completely analogous way as above. Moreover, $\rho_*^\C$ is isomorphic to $(\rho_s\otimes_\C \rho_{\g{so}(k_+)}) \oplus (\rho_s\otimes_\C \rho_{\g{so}(k_-)})$ or to $(\rho_s\otimes_\C \rho_{\g{sp}(k_+)}) \oplus (\rho_s\otimes_\C \rho_{\g{sp}(k_-)})$ accordingly.

Let us consider the bases of $\g{t}$ dual to the bases of $\g{t}^*$ that we have defined. Let them be $\{\e^s_1,\dots,\e^s_p\}\cup\{\e_1,\dots,\e_q\}$ or $\{\e^s_1,\dots,\e^s_p\}\cup\{\e^+_1,\dots,\e^+_{q_+}\}\cup\{\e^-_1,\dots,\e^-_{q_-}\}$ depending on $\g{h}$. By definition, $\w^s_i(\e^s_j)=\delta_{ij}$, $\w_i(\e_j)=\delta_{ij}$, and $\w^\pm_i(\e^\pm_j)=\delta_{ij}$.

It follows from the above discussion that the lowest weights of $\rho_*^\C$ are:
\begin{itemize}
\item If $m\equiv 0\,(\mod 8)$: $\lambda^\pm=-\frac{1}{2}(\w^s_1+\dots +\w^s_p)-\w^\pm_1$ if $k_\pm\geq 2$, and also $\mu^\pm=-\frac{1}{2}(\w^s_1+\dots +\w^s_p)+\w^\pm_1$ if $k_\pm= 2$, or $\lambda^0=-\frac{1}{2}(\w^s_1+\dots +\w^s_p)$ if $k_\pm=1$. 
\item If $m\equiv 1,7\, (\mod 8)$: $\lambda^\pm=-\frac{1}{2}(\w^s_1+\dots +\w^s_{p-1}\pm\w^s_p)-\w_1$ if $k\geq 3$, while $\lambda^\pm=-\frac{1}{2}(\w^s_1+\dots +\w^s_{p-1}\pm\w^s_p)-\w_1$ and $\mu^\pm=-\frac{1}{2}(\w^s_1+\dots +\w^s_{p-1}\pm\w^s_p)+\w_1$ if $k=2$, or $\lambda^\pm=-\frac{1}{2}(\w^s_1+\dots +\w^s_{p-1}\pm\w^s_p)$ if $k=1$.
\item If $m\equiv 2,6\, (\mod 8)$: $\lambda^+=-\frac{1}{2}(\w^s_1+\dots +\w^s_p)-\w_1$ and $\lambda^-=-\frac{1}{2}(\w^s_1+\dots +\w^s_p)+\w_k$.
\item If $m\equiv 3,5\, (\mod 8)$: $\lambda^\pm=-\frac{1}{2}(\w^s_1+\dots +\w^s_{p-1}\pm\w^s_p)-\w_1$.
\item If $m\equiv 4\, (\mod 8)$: $\lambda^\pm=-\frac{1}{2}(\w^s_1+\dots +\w^s_p)-\w^\pm_1$ if $k_\pm\neq 0$.
\end{itemize}

At this point, it is probably convenient to have in mind how the lowest weight diagrams associated to FKM-foliations look like. 
These can be obtained easily from the information above. We have included a generic picture of them in Table~\ref{table:FKM_diagrams}. 

\begin{remark}\label{rem:exceptional_diagrams}
As follows from the above description, the number of lowest weights associated to FKM-foliations satisfying $m_1\leq m_2$ may vary from one to four depending on $m$, $k$, and $k_\pm$. Although generically the number is one or two and the lowest weight diagram takes the corresponding form shown in Table~\ref{table:FKM_diagrams}, for small values (i.e.\ up to $2$) of $k$ or $k_\pm$ there can be three or four lowest weights, and the diagram may adopt a somewhat different form (this happens also if $m=1$). Just to show a pair of examples: the diagrams of the cases $m\equiv 0\,(\mod 8)$, $k_+=2$, $k_-=1$, and $m=1$, $k=4$ are, respectively
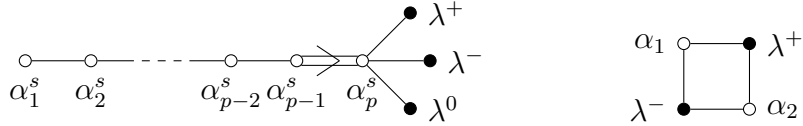
\begin{center}
\begin{tabular}{c@{\qquad \qquad}c@{.}}
\begin{tikzpicture}[node distance=\nodedistance+0.1,g/.style={circle,inner sep=\circleinnersep,draw},lw/.style={circle,inner sep=\circleinnersep,fill=black, draw}]
\node[g] (a1s)[label=below:$\alpha^s_1$] {};
\node[g] (a2s) [right=of a1s, label=below:$\alpha^s_2$] {}
edge [] (a1s);
\node[] (a3s) [right=of a2s] {};
\node[g] (ap2) [right=of a3s, label=below:$\alpha^s_{p-2}$] {};
\node[g] (ap1) [right=of ap2, label=below:$\alpha^s_{p-1}$] {}
edge [] (ap2);
\node[g] (ap) [right=of ap1, label=below:$\alpha^s_p$] {};
\draw ($(ap1)!.65!(ap)$) -- ($(ap1)!.65!(ap)+(-0.3,0.2)$);
\draw ($(ap1)!.65!(ap)$) -- ($(ap1)!.65!(ap)+(-0.3,-0.2)$);
\draw [] (ap1.north east) to (ap.north west);
\draw [] (ap1.south east) to (ap.south west);
\draw ($(a2s)!.5!(a3s)$) -- (a2s);
\draw ($(ap2)!.5!(a3s)$) -- (ap2);
\draw [dashed] ($(a2s)!.5!(a3s)$) -- ($(ap2)!.5!(a3s)$);
\node[lw] (lw+) at ($(ap)+(45:0.9)$) [label=right:$\lambda^+$] {}
edge [] (ap);
\node[lw] (lw-) at ($(ap)+(0:0.9)$) [label=right:$\lambda^-$] {}
edge [] (ap);
\node[lw] (lw-) at ($(ap)+(-45:0.9)$) [label=right:$\lambda^0$] {}
edge [] (ap);
\end{tikzpicture}
&
\begin{tikzpicture}[node distance=\nodedistance+0.1,g/.style={circle,inner sep=\circleinnersep,draw},lw/.style={circle,inner sep=\circleinnersep,fill=black, draw}]
\node[g] (a1)[label=left:$\alpha_1$] {};
\node[lw] (lw+) [right=of a1, label=right:$\lambda^+$] {}
edge [] (a1);
\node[g] (a2) [below=of lw+, label=right:$\alpha_2$] {}
edge [] (lw+);
\node[lw] (lw-) [left=of a2, label=left:$\lambda^-$] {}
edge [] (a2)
edge [] (a1);
\end{tikzpicture}
\end{tabular}
\end{center}
These pecularities should be taken into account in what follows.
\end{remark}

The last ingredient to address the classification is the following result.

\begin{theorem}\label{th:Out_FKM}
Let $\cal{F}$ be an FKM-foliation satisfying $m_1\leq m_2$. Then $\Out_{\cal{F}}^\pm(\Delta_\g{k},\Delta_V)$ is isomorphic to the group of automorphisms of the lowest weight diagram of $\cal{F}$. The correspondence is the natural one given in Proposition~\ref{prop:auto_diagram}.
\end{theorem}
\begin{proof}
In view of Proposition~\ref{prop:auto_diagram}, it suffices to prove that every symmetry of the lowest weight diagram induces an element in  $\Out_{\cal{F}}^\pm(\Delta_\g{k},\Delta_V)$. We show this by cases, depending on the shape of the diagram. For the sake of clarity, we do the proof for diagrams with generic shape. Only minor changes are needed to deal with low values of $k$, $k_\pm$.

If $m=2p-1$ is odd, then there is an automorphism $\sigma$ of the diagram that switches both lowest weights, also the roots $\alpha^s_{p-1}$ and $\alpha^s_p$, and fixes the other roots. We can assume that $\g{t}_s=\mathrm{span}\{P_0P_1,P_2P_3,\dots,P_{2p-2}P_{2p-1}\}$ and $\e^s_i=P_{2i-2}P_{2i-1}$, $1\leq i\leq p$, where $(P_0,\dots,P_{2p-1})$ is a Clifford system defining $\cal{F}$. Then $\Ad(P_{2p-1})\in\Aut(\g{so}(m+1))$ acts as a reflection on $\g{t}_s$: it fixes $e^s_i$ for $1\leq i\leq p-1$ and changes the sign of $\e^s_p$. Then the trivial extension of $\Ad(P_{2p-1})\rvert_{\g{t}_s}$ to $\g{t}$ belongs to $\Aut_\cal{F}(\Delta_\g{k},\Delta_V)$ by Theorem~\ref{th:auto_FKM_1}, and since it is precisely induced by $\sigma$, it also belongs to $\Out_\cal{F}(\Delta_\g{k},\Delta_V)$.

If $m\equiv 1,7\,(\mod 8)$ and $k=2q$ is even, there is a symmetry $\varphi$ of the diagram that fixes all roots and lowest weights, except $\alpha_{q-1}$ and $\alpha_q$, which are interchanged. It is not restrictive to assume that each generator $\e_i$ of $\g{t}_\g{h}$ is the matrix of $\g{so}(k)$ with $1$ in the position $(2i,2i-1)$, $-1$ in the position $(2i-1,2i)$, and all other entries vanish. Then the diagonal matrix $A\in\O(k)$ with entries $(1,\dots,1,-1)$ satisfies that $\Ad(A)\e_i=\e_i$, $1\leq i\leq q-1$, and $\Ad(A)\e_q=-\e_q$. By Theorem~\ref{th:auto_FKM_1}, the trivial extension of $\Ad(A)\rvert_{\g{t}_\g{h}}$ to $\g{t}$, which is precisely induced by $\varphi$, belongs to $\Out_\cal{F}(\Delta_\g{k},\Delta_V)$.
The cases $m\equiv 0\,(\mod 8)$ with $k_\pm$ even can be tackled with a similar argument.

If $m\equiv 2,6 \,(\mod 8)$, the diagram has only one symmetry $\sigma$, which fixes all roots $\alpha_j^s$, interchanges both lowest weights, and switches the roots $\alpha_i$ and $\alpha_{q-i}$, for $1\leq i\leq q-1$ (equivalently $\sigma(\e_i)=-\e_{q+1-i}$, $1\leq i\leq q$). The opposition element $op$ of the Weyl group of $\g{u}(k)$ sends each $\e_i$ to $\e_{q+1-i}$ (see \cite[p.~88]{Sa90}). Denoting also by $op$ its trivial extension from $\g{t}_\g{h}$ to $\g{t}$, we have that $op\in W(\Delta_\g{k})\subset \Aut_\cal{F}(\Delta_\g{k},\Delta_V)$, so $\sigma=- op\in\Out_\cal{F}^\pm(\Delta_\g{k},\Delta_V)$.

Let now $m\equiv 0\, (\mod 4)$ with $k_+=k_-$. After a suitable choice of $\g{t}$, the automorphism $\tau\in\Aut(\cal{F})$ defined in Proposition~\ref{prop:U-} determines an element $\varphi_\tau\in\Out_\cal{F}(\Delta_\g{k},\Delta_V)$. This $\varphi_\tau$ comes induced by the symmetry of the diagram that switches both lowest weights, interchanges $\alpha^+_i$ with $\alpha^-_i$ for $1\leq i\leq q_+=q_-$, and fixes the other roots $\alpha^s_j$.

Finally, note that the automorphisms of the lowest weight diagrams that we have considered generate the symmetry group of such diagrams, from where the result follows.
\end{proof}

\subsection{Classification of the complex structures} \label{subsec:classification_FKM}
We can now address the case-by-case classification of the complex structures preserving FKM-foliations $\cal{F}=\cal{F}_\cal{P}$ satisfying $m_1\leq m_2$. Propositions~\ref{prop:plus_minus_1}, \ref{prop:criterion_Out_F}, and Theorem~\ref{th:Out_FKM} make this work straightforward. For each case, we provide the following information: a set of simple roots for $\g{k}$ and the set of weights $\Delta_V$ (in accordance with the notation introduced above in this section); the algebraic conditions for a generic element $T=\sum_{i=1}^p x^s_i\e^s_i+\sum_{j=1}^q x_j\e_j\in\g{t}$ (or $T=\sum_{i=1}^p x^s_i\e^s_i+\sum_{j=1}^{q_+} x^+_j\e^+_j +\sum_{j=1}^{q_-} x^-_j\e^-_j$ if $m\equiv 0\, (\mod 4)$) to belong to $\cal{J}\cap\bar{C}$; the set $\cal{J}\cap\bar{C}$; a set of generators of $\Out_\cal{F}^\pm(\Delta_\g{k},\Delta_V)$ (if needed); the action of these generators on $\cal{J}\cap\bar{C}$ (if needed); and, finally, the value $N=N(\cal{F}_\cal{P})$ of different complex structures up to congruence of the corresponding projected foliations on the complex projective space.

\subsection*{Type $m\equiv 0\, (\mod 8)$ with $k_+=2q_+$, $k_-=2q_-$ even}
\begin{itemize}
\item Simple roots: $\w^s_1-\w^s_2, \dots, \w^s_{p-1}-\w^s_p,\w^s_p$; $\w^+_1-\w^+_2,\dots,\w^+_{q_+-1}-\w_{q_+},\w^+_{q_+-1}+\w_{q_+}$; $\w^-_1-\w^-_2,\dots,\w^-_{q_--1}-\w_{q_-},\w^-_{q_--1}+\w_{q_-}$.

\item Weights: $\frac{1}{2}(\pm \w^s_1\pm\dots\pm\w^s_p)\pm \w^+_j$ for  $j=1,\dots, q_+$, and $\frac{1}{2}(\pm \w^s_1\pm\dots\pm\w^s_p)\pm \w^-_j$ for  $j=1,\dots, q_-$.

\item Conditions for $\cal{J}\cap\bar{C}$: $x^s_1\geq \dots\geq x^s_p\geq 0$, $x^+_1\geq \dots\geq x^+_{q_+-1}\geq|x^+_{q_+}|$, $x^-_1\geq \dots\geq x^-_{q_--1}\geq |x^-_{q_-}|$, and $\frac{1}{2}(\pm x^s_1\pm\dots\pm x^s_p)\pm x^\pm_j\in\{\pm 1\}$ for all combination of signs and all possible $j$.

\item $\mathcal{J}\cap\bar{C}= \{ 2\e^s_1, 
\sum_{j=1}^{q_+}\e^+_j +\sum_{j=1}^{q_-}\e^-_j, 
-\e^-_{q_-}+\sum_{j=1}^{q_+}\e^+_j +\sum_{j=1}^{q_--1}\e^-_j ,
  -\e^+_{q_+}+ \sum_{j=1}^{q_+-1}\e^+_j+\sum_{j=1}^{q_-}\e^-_j,
-\e^+_{q_+}-\e^-_{q_-}+ \sum_{j=1}^{q_+-1}\e^+_j +\sum_{j=1}^{q_--1}\e^-_j \}$.

\item Generators of $\Out_\cal{F}^\pm(\Delta_\g{k},\Delta_V)$:  
\begin{itemize}
\item[$\circ$] $\varphi^\pm$ (only if $k_\pm\geq 2$): $\e^\pm_{q_\pm}\leftrightarrow-\e^\pm_{q_\pm}$, and fixes the other $\e^+_j$, $\e^-_j$, and $\e^s_i$.
\item[$\circ$] $\tau$ (only if $k_+=k_-$): $\e^+_{j}\leftrightarrow \e^-_{j}$, and fixes the $\e^s_i$.
\end{itemize}
	
\item Action on $\mathcal{J}\cap\bar{C}$: the group generated by $\varphi^+$, $\varphi^-$, and $\tau$ (some of these may not exist) fixes $2\e^s_1$ and acts transitively on the other elements.

\item $N=2$.
\end{itemize}

\subsection*{Type $m\equiv 0\, (\mod 8)$ with $k_+=2q_+$ even, $k_-=2q_-+1$ odd}
\begin{itemize}
\item Simple roots: $\w^s_1-\w^s_2, \dots, \w^s_{p-1}-\w^s_p,\w^s_p$; $\w^+_1-\w^+_2,\dots,\w^+_{q_+-1}-\w_{q_+},\w^+_{q_+-1}+\w_{q_+}$; $\w^-_1-\w^-_2,\dots,\w^-_{q_--1}-\w_{q_-},\w_{q_-}$.

\item Weights: $\frac{1}{2}(\pm \w^s_1\pm\dots\pm\w^s_p)\pm \w^+_j$ for  $j=1,\dots, q_+$, $\frac{1}{2}(\pm \w^s_1\pm\dots\pm\w^s_p)\pm \w^-_j$ for  $j=1,\dots, q_-$, and $\frac{1}{2}(\pm \w^s_1\pm\dots\pm\w^s_p)$. 

\item Conditions for $\cal{J}\cap\bar{C}$: $x^s_1\geq \dots\geq x^s_p\geq 0$, $x^+_1\geq \dots\geq x^+_{q_+-1}\geq|x^+_{q_+}|$, $x^-_1\geq \dots\geq x^-_{q_-}\geq 0$, and $\frac{1}{2}(\pm x^s_1\pm\dots\pm x^s_p)\in\{\pm 1\}$, $\frac{1}{2}(\pm x^s_1\pm\dots\pm x^s_p)\pm x^\pm_j\in\{\pm 1\}$ for all combination of signs and all possible $j$.

\item $\mathcal{J}\cap\bar{C}= \{ 2\e^s_1 \}$, and hence $N=1$.
\end{itemize}

\subsection*{Type $m\equiv 0\, (\mod 8)$ with $k_+=2q_++1$, $k_-=2q_-+1$ odd}
\begin{itemize}
\item Simple roots: $\w^s_1-\w^s_2, \dots, \w^s_{p-1}-\w^s_p,\w^s_p$; $\w^+_1-\w^+_2,\dots,\w^+_{q_+-1}-\w_{q_+},\w_{q_+}$; 
$\w^-_1-\w^-_2,\dots,\w^-_{q_--1}-\w_{q_-},\w_{q_-}$.

\item Weights: $\frac{1}{2}(\pm \w^s_1\pm\dots\pm\w^s_p)\pm \w^+_j$ for  $j=1,\dots, q_+$, $\frac{1}{2}(\pm \w^s_1\pm\dots\pm\w^s_p)\pm \w^-_j$ for  $j=1,\dots, q_-$, and $\frac{1}{2}(\pm \w^s_1\pm\dots\pm\w^s_p)$. 

\item Conditions for $\cal{J}\cap\bar{C}$: $x^s_1\geq \dots\geq x^s_p\geq 0$, $x^+_1\geq \dots\geq x^+_{q_+}\geq 0$, $x^-_1\geq \dots\geq x^-_{q_-}\geq 0$, and $\frac{1}{2}(\pm x^s_1\pm\dots\pm x^s_p)\in\{\pm 1\}$, $\frac{1}{2}(\pm x^s_1\pm\dots\pm x^s_p)\pm x^\pm_j\in\{\pm 1\}$ for all combination of signs and all possible $j$.

\item $\mathcal{J}\cap\bar{C}= \{ 2\e^s_1 \}$, and hence $N=1$.
\end{itemize}

\subsection*{Type $m\equiv 1,7\, (\mod 8)$ with $k=2q$ even}
\begin{itemize}
\item Simple roots: $\w^s_1-\w^s_2, \dots, \w^s_{p-1}-\w^s_p,\w^s_{p-1}+\w^s_p$; $\w_1-\w_2,\dots,\w_{q-1}-\w_q,\w_{q-1}+\w_q$.

\item Weights: $\frac{1}{2}(\pm \w^s_1\pm\dots\pm\w^s_p)\pm \w_j$, for all $j=1,\dots, q$.

\item Conditions for $\cal{J}\cap\bar{C}$: $ x^s_1\geq \dots\geq x^s_{p-1}\geq \left| x^s_p \right|$, $x_1\geq \dots\geq x_{q-1}\geq |x_q|$, and $\frac{1}{2}(\pm x^s_1\pm\dots\pm x^s_p)\pm x_j\in\{\pm 1\}$ for all combination of signs and all $j=1,\dots, q$.

\item $\mathcal{J}\cap\bar{C}= \{ 2\e^s_1 , \e_1+\dots+\e_q, \e_1+\dots+\e_{q-1}-\e_q\}$, and also $-2\e^s_1$ if $m=1$.

\item Generators of $\Out_\cal{F}^\pm(\Delta_\g{k},\Delta_V)$: 
	\begin{itemize}
	\item[$\circ$] $\sigma$: $\e^s_{p}\leftrightarrow -\e^s_p$, and the other $\e^s_i$ and $\e_j$ stay fixed.
	\item[$\circ$] $\varphi$: $\e_q\leftrightarrow -\e_q$, and the other $\e^s_i$ and $\e_j$ stay fixed.	
	\end{itemize}
	
\item Action on $\mathcal{J}\cap\bar{C}$: $\sigma$ fixes all elements if $m\neq 1$ (if $m=1$, then $2\e^s_1\leftrightarrow -2\e^s_1$), while $\varphi$ interchanges $\e_1+\dots+\e_q\leftrightarrow \e_1+\dots+\e_{q-1}-\e_q$ and fixes the other elements.

\item $N=2$.
\end{itemize}

\subsection*{Type $m\equiv 1,7\, (\mod 8)$ with $k=2q+1$ odd}
\begin{itemize}
\item Simple roots: $\w^s_1-\w^s_2, \dots, \w^s_{p-1}-\w^s_p,\w^s_{p-1}+\w^s_p$; $\w_1-\w_2,\dots,\w_{q-1}-\w_q,\w_q$.

\item Weights: $\frac{1}{2}(\pm \w^s_1\pm\dots\pm\w^s_p)\pm \w_j$, for all $j=1,\dots, q$, and $\frac{1}{2}(\pm \w^s_1\pm\dots\pm\w^s_p)$.

\item Conditions for $\cal{J}\cap\bar{C}$: $ x^s_1\geq \dots\geq x^s_{p-1}\geq \left| x^s_p \right|$, $x_1\geq \dots\geq x_q\geq 0$, and $\frac{1}{2}(\pm x^s_1\pm\dots\pm x^s_p)\in\{\pm 1\}$, $\frac{1}{2}(\pm x^s_1\pm\dots\pm x^s_p)\pm x_j\in\{\pm 1\}$ for all combination of signs and all $j=1,\dots, q$.

\item $\mathcal{J}\cap\bar{C}= \{ 2\e^s_1\}$ if $m\neq 1$, or $\mathcal{J}\cap\bar{C}= \{ \pm 2\e^s_1\}$ if $m=1$.

\item $N=1$.
\end{itemize}

\subsection*{Type $m\equiv 2,6\, (\mod 8)$}
\begin{itemize}
\item Simple roots: $\w^s_1-\w^s_2, \dots, \w^s_{p-1}-\w^s_p,\w^s_p$\,;\, $\w_1-\w_2,\dots,\w_{k-1}-\w_k$.

\item Weights: $\frac{1}{2}(\pm \w^s_1\pm\dots\pm\w^s_p)\pm \w_j$, for all $j=1,\dots, k$.

\item Conditions for $\cal{J}\cap\bar{C}$: $x^s_1\geq \dots\geq x^s_p\geq 0$, $x_1\geq \dots\geq x_k$, and $\frac{1}{2}(\pm x^s_1\pm\dots\pm x^s_p)\pm x_j\in\{\pm 1\}$ for all combination of signs and all $j=1,\dots, k$.

\item $\mathcal{J}\cap\bar{C}= \{ 2\e^s_1 \}\cup\{\sum_{j=1}^k \epsilon_j \e_j:\epsilon_j=\pm 1 \text{ for all }j,\, \epsilon_1\geq\dots\geq \epsilon_k\}$.

\item Generator of $\Out_\cal{F}^\pm(\Delta_\g{k},\Delta_V)$:  $\sigma$ switches $\e_j\leftrightarrow -\e_{k+1-j}$ for all $j$, and fixes the $\e^s_i$.
	
\item Action on $\mathcal{J}\cap\bar{C}$: $\sigma$ switches $\sum_{j=1}^k \epsilon_j \e_j\leftrightarrow -\sum_{j=1}^k \epsilon_{k-j+1} \e_j$, and fixes $2\e^s_1$.

\item $N=2+\left[\frac{k}{2}\right]$.
\end{itemize}

\subsection*{Type $m\equiv 3,5\, (\mod 8)$}
\begin{itemize}
\item Simple roots: $\w^s_1-\w^s_2, \dots, \w^s_{p-1}-\w^s_p,\w^s_{p-1}+\w^s_p$; $\w_1-\w_2,\dots,\w_{k-1}-\w_k,2\w_k$.

\item Weights: $\frac{1}{2}(\pm \w^s_1\pm\dots\pm\w^s_p)\pm \w_j$, for all $j=1,\dots, k$.

\item Conditions for $\cal{J}\cap\bar{C}$: $ x^s_1\geq \dots\geq x^s_{p-1}\geq \left| x^s_p \right|$, $x_1\geq \dots\geq x_k\geq 0$, and $\frac{1}{2}(\pm x^s_1\pm\dots\pm x^s_p)\pm x_j\in\{\pm 1\}$ for all combination of signs and all $j=1,\dots, k$.

\item $\mathcal{J}\cap\bar{C}= \{ 2\e^s_1 , \e_1+\dots+\e_k\}$.

\item Generator of $\Out_\cal{F}^\pm(\Delta_\g{k},\Delta_V)$: $\sigma$ switches $\e^s_{p}\leftrightarrow -\e^s_p$, and fixes the other $\e^s_i$ and $\e_j$.
	
\item Action on $\mathcal{J}\cap\bar{C}$: both elements are fixed by $\sigma$.

\item $N=2$.
\end{itemize}

\subsection*{Type $m\equiv 4\, (\mod 8)$}
\begin{itemize}
\item Simple roots: $\w^s_1-\w^s_2, \dots, \w^s_{p-1}-\w^s_p,\w^s_p$; $\w^+_1-\w^+_2,\dots,\w^+_{k_+-1}-\w^+_{k_+},2\w^+_{k_+}$; $\w^-_1-\w^-_2,\dots,\w^-_{k_--1}-\w^-_{k_-},2\w^-_{k_-}$.

\item Weights: $\frac{1}{2}(\pm \w^s_1\pm\dots\pm\w^s_p)\pm \w^+_j$ for $j=1,\dots, k_+$, and $\frac{1}{2}(\pm \w^s_1\pm\dots\pm\w^s_p)\pm \w^-_j$ for $j=1,\dots, k_-$.

\item Conditions for $\cal{J}\cap\bar{C}$: $ x^s_1\geq \dots\geq x^s_{p}\geq 0$, $x^+_1\geq \dots\geq x^+_{k_+}\geq 0$, $x^-_1\geq \dots\geq x^-_{k_-}\geq 0$, and $\frac{1}{2}(\pm x^s_1\pm\dots\pm x^s_p)\pm x^\pm_j\in\{\pm 1\}$ for all combination of signs and all possible $j$.

\item $\mathcal{J}\cap\bar{C}= \{ 2\e^s_1 , \e^+_1+\dots+\e^+_{k_+}+\e^-_1+\dots+\e^-_{k_-}\}$.

\item Generator of $\Out_\cal{F}^\pm(\Delta_\g{k},\Delta_V)$, only if $k_+=k_-$: $\tau$ switches $\e^+_{j}\leftrightarrow \e^-_j$ for all $j$, and fixes the $\e^s_i$.
	
\item Action on $\mathcal{J}\cap\bar{C}$: both elements are fixed by $\tau$.

\item $N=2$.
\end{itemize}

\section{Inhomogeneous isoparametric foliations} \label{sec:homogeneity}
Here we analyze when the projection of an isoparametric foliation to the complex projective space gives rise to a homogeneous foliation. Some curious consequences are also derived.

Let us start with an elementary consideration.

\begin{remark}\label{rem:homogeneity}
The pullback $\cal{F}=\pi^{-1}\cal{G}$ of any homogeneous foliation $\cal{G}$ on $\C P^n$ under the Hopf map $\pi$ is homogeneous. Indeed, consider the maximal connected subgroup $\wt{K}$ of $\U(n+1)$ preserving the leaves of $\cal{G}$; by homogeneity, the orbit foliation of $\wt{K}$ on $\C P^n$ is $\cal{G}$. It follows that the orbit foliation of the action of $\wt{K}$ on $S^{2n+1}\subset\C^{n+1}$ is $\cal{F}$.
\end{remark}

Therefore, every homogeneous isoparametric foliation on $\C P^n$ must be the projection of a homogeneous polar foliation. Then our aim reduces to deciding when the projection to $\C P^n$ of a homogeneous polar foliation $\mathcal{F}_{G/K}$ on $S^{2n+1}$ is homogeneous. The following subtle improvement of \cite[Th.~3.1]{PT99} gives us the solution.

\begin{theorem}\label{th:homogeneity}
Let $(G,K)$ be a compact inner symmetric pair that satisfies the maximality property, 
 $(G,K)=\Pi_{i=1}^r(G_i,K_i)$ its decomposition in irreducible factors, and $\g{g}_i=\g{k}_i\oplus\g{p}_i$ the Cartan decomposition of $G_i/K_i$. Let $J=\ad(X)\rvert_\g{p}$ be a complex structure on $\g{p}=\bigoplus_{i=1}^r\g{p}_i$ that preserves the foliation $\mathcal{F}_{G/K}$, and put $X=X_1+\ldots+X_r$, with $X_i\in \g{k}_i$. Then the following conditions are equivalent:
\begin{itemize}
\item[(i)] The projection of $\mathcal{F}_{G/K}$ to the complex projective space $\C P^n$ determined by $J$ is a homogeneous foliation.
\item[(ii)] The irreducible factors $G_i/K_i$ of $G/K$ are Hermitian or of rank one, and for each Hermitian irreducible factor $G_i/K_i$, $X_i$ belongs to the center of $\g{k}_i$.
\end{itemize}
\end{theorem}
\begin{proof}
First assume (ii). If an irreducible factor $(G_i,K_i)$ is Hermitian and $X_i$ belongs to the center $Z(\g{k}_i)$ of $\g{k}_i$, then the adjoint $K_i$-action on $\g{p}_i$ commutes with $J$, that is, $\Ad(K_i)\rvert_\g{p_i}$ consists of unitary transformations with respect to $J$.  The irreducible factors of rank one are of the type $(\SO(2p_i+1),\SO(2p_i))$, with $2p_i=\dim \g{p}_i$. For these factors, the group $\U(p_i)$ of unitary transformations with respect to $\ad(X_i)\rvert_{\g{p}_i}$ acts on $\g{p}_i$ with the same orbits as $\SO(2p_i)$. It follows that there exists a group $\widetilde{K}$ of unitary transformations with respect to $J$ that acts on $\g{p}$ with the same orbits as $\Ad(K)\rvert_\g{p}$. Therefore $\widetilde{K}$ induces an action on $\C P^n$ whose orbits coincide with those of the projection of $\mathcal{F}_{G/K}$.

Assume now (i). Then, there exists a group $K'$ acting polarly on $\C P^n$ and with the leaves of the projection of $\mathcal{F}_{G/K}$ as orbits; the existence of sections intersecting all orbits is a consequence of the polarity of the $s$-representation of $G/K$ (see Proposition~\ref{prop:HopfMap}). It was shown in \cite[Th.~3.1]{PT99} that there exists a connected group $\widetilde{K}$ acting on $\g{p}$ effectively, unitarily with respect to $J$, and polarly, and such that the projection of its orbits yields the orbits of $K'$. Hence, $\widetilde{K}$ and $\Ad(K)\rvert_\g{p}$ act on $\g{p}$ with the same orbits. By maximality, we can identify $\widetilde{K}$ with a subgroup of $\Ad(K)\rvert_\g{p}$.

Let us consider the case when $G/K$ is irreducible. Take a subgroup $H$ of $K$ such that $\Ad(H)\rvert_\g{p}=\widetilde{K}$. Then \cite[Lemma~3.2]{PT99} implies that $G/K$ is Hermitian or $(G,K)=(\SO(2p+1),\SO(2p))$. If $G/K$ is Hermitian, then either $G/K$ is of rank greater than one or $G=\SO(3)$, $K=\SO(2)$; in this last case, $X\in\g{k}=Z(\g{k})$. Assume that $G/K$ is Hermitian of rank greater than one. Since $\widetilde{K}=\Ad(H)\rvert_\g{p}$ acts on $\g{p}$ by unitary transformations with respect to $J=\ad(X)\rvert_\g{p}$, and since the $s$-representation of $(G,K)$ is effective, we have that $X\in Z(\g{h})$. Then, the center of $\g{h}$ cannot be trivial. The main theorem of \cite{EH99mz} implies that $H=K$ unless:
\begin{itemize}
\item $G=\SO(9)$, $K=\SO(2)\times\SO(7)$, $H=\SO(2)\times\mathrm{G}_2$, or
\item $G=\SO(10)$, $K=\SO(2)\times\SO(8)$, $H=\SO(2)\times\mathrm{Spin}(7)$.
\end{itemize}
In both cases, $\dim Z(\g{h})=1$ and $Z(\g{h})=Z(\g{k})$. In any case, $X\in Z(\g{h})=Z(\g{k})$.

Assume now that $G/K$ is reducible. Then $\Ad(K)\rvert_\g{p}$ (and hence $\widetilde{K}$) acts irreducibly on each $\g{p}_i$, $i=1,\ldots,r$; by \cite[Th.~4]{Da85} each one of these actions is polar. For each $i$, let $\widetilde{K}_i$ be a quotient of $\widetilde{K}$ that acts irreducibly, polarly, and effectively on $\g{p}_i$, and with the same orbits as the action of $\widetilde{K}$ on $\g{p}_i$. The actions of $\widetilde{K}_i$ and of $\Ad(K_i)\rvert_{\g{p}_i}$ on $\g{p}_i$ have the same orbits; by maximality we can find a subgroup $H_i$ of $K_i$ such that $\Ad(H_i)\rvert_{\g{p}_i}=\widetilde{K}_i$.
 
Every $\Ad(H_i)\rvert_{\g{p}_i}$ acts unitarily on $\g{p}_i$ with respect to the complex structure $\ad(X_i)\rvert_{\g{p}_i}$, since $\tilde{K}$ acts unitarily on $\g{p}$ with respect to $J$. We can hence apply the argument above for the irreducible case. We obtain that for each $i\in\{1,\ldots,r\}$, $G_i/K_i$ is Hermitian with $X_i\in Z(\g{k}_i)$, or $(G_i,K_i)=(\SO(2p_i+1),\SO(2p_i))$.
\end{proof}

We obtain now some consequences of this result. The first one is straightforward. 

\begin{corollary}\label{cor:inhomogeneous}
Let $(G,K)$ be an irreducible compact inner symmetric pair of rank greater than one satisfying the maximality property, and let $N(\cal{F}_{G/K})$ be as in Subsection~\ref{subsec:classification_homogeneous} (see Table~\ref{table:diagrams} for its concrete value depending on $(G,K)$).

Then, among the $N(\cal{F}_{G/K})$ noncongruent irreducible isoparametric foliations of the complex projective space obtained by projecting $\mathcal{F}_{G/K}$, exactly $N(\cal{F}_{G/K})-1$ of them are inhomogeneous if $G/K$ is Hermitian, whereas all of them are inhomogeneous if $G/K$ is non-Hermitian.
\end{corollary}

The following result focuses on the existence of inhomogeneous isoparametric foliations depending on the dimension on the ambient complex projective space. Although Theorem~\ref{th:codimension}(i) has been recently proved in \cite[Th.~1.1]{GTY11} by Ge, Tang, and Yan, for the sake of completeness we include here a slightly different proof. The other claims in Theorem~\ref{th:codimension} are new. 

\begin{theorem}\label{th:codimension}
We have:
\begin{itemize}
\item[(i)] $\C P^n$ admits an inhomogeneous isoparametric foliation of codimension one if and only if $n$ is an odd number greater or equal than $3$. 
\item[(ii)] Let $q\in\mathbb{N}$, $q\geq 2$. Then $\C P^n$ admits an irreducible inhomogeneous isoparametric foliation of codimension $q$ if and only if $(q+1)^2\leq 2(n+1)$ and $q+1$ divides $2(n+1)$.
\end{itemize}
In particular, every irreducible isoparametric foliation on $\C P^n$ is homogeneous if and only if $n+1$ is a prime number.
\end{theorem}
\begin{proof}
We start with the necessity of (i). Clearly $\C P^1\cong S^2$ only admits homogeneous isoparametric foliations. Let $n$ be even,  $\mathcal{F}\subset S^{2n+1}$ an isoparametric foliation of codimension one, and $M$ an arbitrary hypersurface of $\mathcal{F}$ with $g\in\{1,2,3,4,6\}$ principal curvatures with multiplicities $m_1,\dots,m_g$. Recall that $m_i= m_{i+2}$ (indices modulo $g$). 

If  $g\in\{1,3\}$, a standard argument involving the Coxeter group of $\mathcal{F}$ (see for example \cite[p.~359]{Th10}) implies that a generic $M$ is not invariant under the antipodal map, so $M$ is not foliated by Hopf circles. If $g=2$, $\mathcal{F}$ is the orbit foliation of the $s$-representation of a product of two spheres, so its projection to $\C P^n$ is homogeneous, according to Theorem~\ref{th:homogeneity}.

If $g=4$ then $m_1+m_2=n$ is even. Hence a result of Abresch \cite{Ab83} implies that either $\min\{m_1,m_2\}=1$ or $m_1=m_2=2$. If $g=6$, then again by \cite{Ab83} we have $m_1=m_2\in\{1,2\}$; since $3(m_1+m_2)=2n$ in this case, $m_1=m_2=1$ is impossible.

If $g=4$ and $\min\{m_1,m_2\}=1$, according to Takagi's result \cite{Ta76}, $\mathcal{F}$ is the orbit foliation of the $s$-representation of the symmetric pair \textbf{B I}, with $\nu=1$. By virtue of Corollary~\ref{cor:inhomogeneous}, the projection of $\mathcal{F}$ to $\C P^n$ is homogeneous. We are left with the cases $(g,m_1,m_2)\in\{(4,2,2),(6,2,2)\}$, for which Proposition~\ref{prop:622} shows that $M$ is not foliated by Hopf circles. 

For the proof of (ii) and the sufficiency of (i), we will need the concrete values of the rank and dimension of the different symmetric spaces \cite[p.~518]{He78}. 

Assume that $\C P^n$ admits an irreducible inhomogeneous isoparametric foliation $\mathcal{G}$ of codimension $q\geq 2$. Then $\mathcal{G}$ is the projection of the foliation $\mathcal{F}_{G/K}$ of $S^{2n+1}$ defined by certain irreducible symmetric space $G/K$. According to Table~\ref{table:diagrams} and Corollary~\ref{cor:inhomogeneous}, the only possible cases for $G/K$ are: \textbf{A III}, \textbf{BD I}, \textbf{C II}, \textbf{D III}, \textbf{E II}, \textbf{E III}, \textbf{E V}, \textbf{E VI}, \textbf{E VIII}, \textbf{E IX}, \textbf{F I}, and \textbf{G}. One can easily check that, for all these cases, we have that $(\rank G/K)^2\leq \dim G/K$ and that $\rank G/K$ divides $\dim G/K$. Since $\rank G/K=\codim \mathcal{F}_{G/K}+1=q+1$ and $\dim G/K=2(n+1)$, we get that $(q+1)^2\leq 2(n+1)$ and $q+1$ divides $2(n+1)$.

Note that, conversely, if these two conditions hold, then $G/K=\SO(q+r+1)/\SO(q~+~1)\times\SO(r)$, with $r=2(n+1)/(q+1)$, defines a foliation $\cal{F}_{G/K}$ of codimension $q$ on $S^{2n+1}$. Since moreover $q+1$ or $r$ is even and $q+1>2$, then $N(\cal{F}_{G/K})\geq 1$ and $G/K$ is non-Hermitian. Hence, one can project $\cal{F}_{G/K}$ to an irreducible inhomogeneous isoparametric foliation with codimension $q$ on $\C P^n$. An analogous argument proves the sufficiency of (i), if one considers $G/K=\SU(2+r)/\mathrm{S}(\U(2)\times \U(r))$ with $r=(n+1)/2$.

The last claim of the theorem follows easily from (i) and (ii).
\end{proof}

\begin{remark}
The assumption of the irreducibility in Theorem~\ref{th:codimension} (except in part (i)) is essential. 
For example, \textbf{D III} with $p=6$ and \textbf{G} define a reducible inhomogeneous isoparametric foliation of codimension $4$ on $\C P^{18}$.
\end{remark}

\bibliographystyle{amsplain}

\begin{table}[ht]
\renewcommand{\arraystretch}{1.2}
\begin{threeparttable}
\caption{Extended Vogan diagrams of irreducible compact inner symmetric spaces}
\label{table:diagrams}
\begin{tabular}{@{} l @{\;\;} >{\Small} V @{\;} >{\Small} l @{\;\;}  >{\Small} l @{\;\;} >{\Small} l @{\;\;} >{\Small} l @{}}
 & \normalsize Extended Dynkin diagram & {\normalsize $\mu$}  &  {\normalsize $G/K$} & {\normalsize $\lambda$} & {\normalsize $N(\cal{F}_{G/K})$} \\ \hline
$A_p$ & \vspace{\upperlinespace}
\begin{tikzpicture}[node distance=\nodedistance,g/.style={circle,inner sep=\circleinnersep,draw},a0/.style={rectangle,inner sep=\rectangleinnersep,draw}]
\node[g] (a1) [label=below:$\alpha_1$] {};
\node[g] (a2) [right=of a1,label=below:$\alpha_2$] {}
edge [] (a1);
\node[] (a3) [right=of a2] {};
\node[g] (ap) [right=of a3, label=below:$\alpha_\nu$] {};
\node[] (an2) [right=of ap] {};
\node[g] (an1) [right=of an2, label=below:$\alpha_{p-1}$] {};
\node[g] (an) [right=of an1, label=below:$\alpha_p$] {}
edge [] (an1);
\node[g] (a0) [above=of ap, label=below:$\alpha_0$] {}
edge [] (a1)
edge [] (an);
\draw ($(a3)!.5!(ap)$) -- (ap);
\draw ($(an2)!.5!(ap)$) -- (ap);
\draw ($(a3)!.5!(a2)$) -- (a2);
\draw ($(an2)!.5!(an1)$) -- (an1);
\draw [dashed] ($(a3)!.5!(a2)$) -- ($(a3)!.5!(ap)$);
\draw [dashed] ($(an2)!.5!(ap)$) -- ($(an1)!.5!(an2)$);
\end{tikzpicture}
 & $(1\ldots1)$ & \textbf{A III} & $(1\ldots1)$ & \begin{tabular}{@{}l@{}}$1\!+\!\left[\frac{\nu}{2}\right]\!+\!\left[\frac{p-\nu+1}{2}\right]$ \\ (if $2\nu\neq p+1$)\\ $1+\left[\frac{\nu}{2}\right]$ \\ (if $2\nu=p+1$)\end{tabular}
\\ \hline
$B_p$ &  \vspace{\upperlinespace}
\begin{tikzpicture}[node distance=\nodedistance,g/.style={circle,inner sep=\circleinnersep,draw},a0/.style={rectangle,inner sep=\rectangleinnersep,draw}]
\node[g] (a2) [label=below:$\alpha_2$] {};
\node[] (a3) [right=of a2] {};
\node[g] (ap) [right=of a3, label=below:$\alpha_\nu$] {};
\node[] (an3) [right=of ap] {};
\node[g] (an2) [right=of an3, label=below:$\alpha_{p-2}$] {};
\node[g] (an1) [right=of an2, label=below:$\alpha_{p-1}$] {}
edge [] (an2);
\node[g] (an) [right=of an1, label=below:$\alpha_p$] {};
\draw ($(an1)!.65!(an)$) -- ($(an1)!.65!(an)+(-0.3,0.2)$);
\draw ($(an1)!.65!(an)$) -- ($(an1)!.65!(an)+(-0.3,-0.2)$);
\draw [] (an1.north east) to (an.north west);
\draw [] (an1.south east) to (an.south west);
\node[g] (a0) at (135:1) [label=below:$\alpha_0$] {}
edge [] (a2);
\node[g] (a1) at (-135:1) [label=below:$\alpha_1$] {}
edge [] (a2);
\draw ($(a3)!.5!(ap)$) -- (ap);
\draw ($(an3)!.5!(ap)$) -- (ap);
\draw ($(a3)!.5!(a2)$) -- (a2);
\draw ($(an3)!.5!(an2)$) -- (an2);
\draw [dashed] ($(a3)!.5!(a2)$) -- ($(a3)!.5!(ap)$);
\draw [dashed] ($(an3)!.5!(ap)$) -- ($(an2)!.5!(an3)$);
\end{tikzpicture}
& $(122\ldots 2)$ & \textbf{B I} & $(11\ldots \overset{\nu}{1}2\ldots 2)$  & $1$
\\ \hline
\multirow{2}*[-1.5ex]{$C_p$} & \multirow{2}*[-2.5ex]{
\begin{tikzpicture}[node distance=\nodedistance,g/.style={circle,inner sep=\circleinnersep,draw},a0/.style={rectangle,inner sep=\rectangleinnersep,draw}]
\node[g] (a1) [label=below:$\alpha_1$] {};
\node[g] (a2) [right=of a1,label=below:$\alpha_2$] {}
edge [] (a1);
\node[] (a3) [right=of a2] {};
\node[g] (ap) [right=of a3, label=below:$\alpha_\nu$] {};
\node[] (an2) [right=of ap] {};
\node[g] (an1) [right=of an2, label=below:$\alpha_{p-1}$] {};
\node[g] (an) [right=of an1, label=below:$\alpha_p$] {};
\draw ($(an1)!.35!(an)$) -- ($(an1)!.35!(an)+(0.3,0.2)$);
\draw ($(an1)!.35!(an)$) -- ($(an1)!.35!(an)+(0.3,-0.2)$);
\draw [] (an1.north east) to (an.north west);
\draw [] (an1.south east) to (an.south west);
\node[g] (a0) [left=of a1, label=below:$\alpha_0$] {};
\draw ($(a0)!.65!(a1)$) -- ($(a0)!.65!(a1)+(-0.3,0.2)$);
\draw ($(a0)!.65!(a1)$) -- ($(a0)!.65!(a1)+(-0.3,-0.2)$);
\draw [] (a0.north east) to (a1.north west);
\draw [] (a0.south east) to (a1.south west);
\draw ($(a3)!.5!(ap)$) -- (ap);
\draw ($(an2)!.5!(ap)$) -- (ap);
\draw ($(a3)!.5!(a2)$) -- (a2);
\draw ($(an2)!.5!(an1)$) -- (an1);
\draw [dashed] ($(a3)!.5!(a2)$) -- ($(a3)!.5!(ap)$);
\draw [dashed] ($(an2)!.5!(ap)$) -- ($(an1)!.5!(an2)$);
\end{tikzpicture}
}
& 
\multirow{2}*[-1.5ex]{$(2\ldots 221)$} & \begin{tabular}{@{}l@{}}\textbf{C I}  \\[-0.6ex] \tiny $(\nu=p)$  \end{tabular}  & $(2\ldots 221)$ & 1
\\ \cline{4-6} & & & \begin{tabular}{@{}l@{}} \textbf{C II} \\[-0.6ex] \tiny$(\nu<p)$ \end{tabular} & $(1\ldots 1\overset{\nu}{2}\ldots 21)$ & \begin{tabular}{@{}l@{}}$2$ (if $2\nu\neq p$)\\ $1$ (if $2\nu= p$)\end{tabular}
\\ \hline
\multirow{2}*[-1.5ex]{$D_p$} & \multirow{2}*[0.6ex]{
\begin{tikzpicture}[node distance=\nodedistance, g/.style={circle,inner sep=\circleinnersep,draw},a0/.style={rectangle,inner sep=\rectangleinnersep,draw}]
\node[g] (a2) [label=below:$\alpha_2$] {}
edge [] (a1);
\node[] (a3) [right=of a2] {};
\node[g] (ap) [right=of a3, label=below:$\alpha_\nu$] {};
\node[] (an3) [right=of ap] {};
\node[g] (an2) [right=of an3, label=below left:$\alpha_{p-2}$] {};
\node[g] (an1) at ($(an2)+(-45:1)$) [label=below:$\alpha_{p-1}$] {}
edge [] (an2);
\node[g] (an) at ($(an2)+(45:1)$) [label=below:$\alpha_p$] {}
edge [] (an2);
\node[g] (a1) at (135:1) [label=below:$\alpha_0$] {}
edge [] (a2);
\node[g] (a1) at (-135:1) [label=below:$\alpha_1$] {}
edge [] (a2);
\draw ($(a3)!.5!(ap)$) -- (ap);
\draw ($(an3)!.5!(ap)$) -- (ap);
\draw ($(a3)!.5!(a2)$) -- (a2);
\draw ($(an3)!.5!(an2)$) -- (an2);
\draw [dashed] ($(a3)!.5!(a2)$) -- ($(a3)!.5!(ap)$);
\draw [dashed] ($(an3)!.5!(ap)$) -- ($(an3)!.5!(an2)$);
\end{tikzpicture}
}
& \multirow{2}*[-2ex]{$(12\ldots 211)$} & \begin{tabular}{@{}l@{}}\, \\[-2.6ex] \textbf{D I} \\[-0.6ex] \tiny $(\nu\leq p-2)$\\[-2.6ex] \,\end{tabular} & $(1\ldots \overset{\nu}{1}2\dots 211)$ & \begin{tabular}{@{}l@{}} $2$ (if $2\nu\neq p$) \\ $1$ (if $2\nu=p$)  \end{tabular}
\\ \cline{4-6} & & &
\begin{tabular}{@{}l@{}}\, \\[-2.6ex] \textbf{D III} \\[-0.6ex] \tiny $(\nu\geq p-1)$ \\[-2.6ex]\,\end{tabular} & $(12\ldots 211)$ & $2$
\\ \hline
\multirow{2}*[-2ex]{$E_6$} & \multirow{2}*[1.5ex]{
\begin{tikzpicture}[node distance=\nodedistance, g/.style={circle,inner sep=\circleinnersep,draw},a0/.style={rectangle,inner sep=\rectangleinnersep,draw}]
\node[g] (a1) [label=below:$\alpha_1$] {};
\node[g] (a3) [left=of a1, label=below:$\alpha_3$] {}
edge [] (a1);
\node[g] (a4) [left=of a3, label=below:$\alpha_4$] {}
edge [] (a3);
\node[g] (a2) [above=of a4, label=left:$\alpha_2$, label=right:\tiny\textbf{E II}] {}
edge [] (a4);
\node[g] (a0) [above=of a2, label=left:$\alpha_0$] {}
edge [] (a2);
\node[g] (a5) [left=of a4, label=below:$\alpha_5$] {}
edge [] (a4);
\node[g] (a6) [left=of a5, label=below:$\alpha_6$, label=above:\tiny\textbf{E III}] {}
edge [] (a5);
\end{tikzpicture}
}
& \multirow{2}*[-2ex]{$(122321)$} & \begin{tabular}{@{}l@{}}\, \\[-1ex] \textbf{E II}\\[-1ex] \,\end{tabular}  & $(112321)$ & $2$
\\  \cline{4-6} & & &
  \begin{tabular}{@{}l@{}}\, \\[-1ex]\textbf{E III} \\[-1ex] \,\end{tabular} & $(122321)$  & $2$
\\ \hline
\multirow{3}*{$E_7$} & \multirow{3}*[-0.8ex]{
\begin{tikzpicture}[node distance=\nodedistance, g/.style={circle,inner sep=\circleinnersep,draw},a0/.style={rectangle,inner sep=\rectangleinnersep,draw}]
\node[g] (a1) [label=below:$\alpha_1$, label=above:\tiny\textbf{E VI}] {};
\node[g] (a3) [left=of a1, label=below:$\alpha_3$] {}
edge [] (a1);
\node[g] (a4) [left=of a3, label=below:$\alpha_4$] {}
edge [] (a3);
\node[g] (a2) [above=of a4, label=left:$\alpha_2$, label=right:\tiny\textbf{E V}] {}
edge [] (a4);
\node[g] (a5) [left=of a4, label=below:$\alpha_5$] {}
edge [] (a4);
\node[g] (a6) [left=of a5, label=below:$\alpha_6$] {}
edge [] (a5);
\node[g] (a7) [left=of a6, label=below:$\alpha_7$, label=above:\tiny\textbf{E VII}] {}
edge [] (a6);
\node[g] (a0) [right=of a1, label=below:$\alpha_0$] {}
edge [] (a1);
\end{tikzpicture}
}
& \multirow{3}*[-0.5ex]{$(2234321)$} & \textbf{E V} & $(1123321)$  & $1$
\\ \cline{4-6}& & &
\textbf{E VI}  & $(1234321)$  & $2$
\\ \cline{4-6}& & &
\textbf{E VII} & $(2234321)$ & $1$
\\ \hline
\multirow{2}*[-1ex]{$E_8$} & \multirow{2}*[0ex]{
\begin{tikzpicture}[node distance=\nodedistance, g/.style={circle,inner sep=\circleinnersep,draw},a0/.style={rectangle,inner sep=\rectangleinnersep,draw}]
\node[g] (a1) [label=below:$\alpha_1$, label=above:\negthickspace\negthickspace\negthickspace\tiny\textbf{E \!VIII}] {};
\node[g] (a3) [left=of a1, label=below:$\alpha_3$] {}
edge [] (a1);
\node[g] (a4) [left=of a3, label=below:$\alpha_4$] {}
edge [] (a3);
\node[g] (a2) [above=of a4, label=left:$\alpha_2$] {}
edge [] (a4);
\node[g] (a5) [left=of a4, label=below:$\alpha_5$] {}
edge [] (a4);
\node[g] (a6) [left=of a5, label=below:$\alpha_6$] {}
edge [] (a5);
\node[g] (a7) [left=of a6, label=below:$\alpha_7$] {}
edge [] (a6);
\node[g] (a8) [left=of a7, label=below:$\alpha_8$, label=above:\tiny\textbf{E IX}] {}
edge [] (a7);
\node[g] (a0) [left=of a8, label=below:$\alpha_0$] {}
edge [] (a8);
\end{tikzpicture}
}
& \multirow{2}*[-1.2ex]{$(23465432)$} & \begin{tabular}{@{}l@{}}\,\\[-2.1ex]\textbf{E VIII}\\[-2.1ex]\, \end{tabular} & $(13354321)$ & $1$
\\ \cline{4-6} & & &
 \begin{tabular}{@{}l@{}}\,\\[-2.1ex]\textbf{E IX}\\[-2.1ex]\, \end{tabular} & $(23465431)$ & $1$
\\ \hline
\multirow{2}*[-0.1ex]{$F_4$} & \multirow{2}*[-0.5ex]{
\begin{tikzpicture}[node distance=\nodedistance, g/.style={circle,inner sep=\circleinnersep,draw},a0/.style={rectangle,inner sep=\rectangleinnersep,draw}]
\node[g] (a1) [label=below:$\alpha_1$, label=above:\tiny\textbf{F II}] {};
\node[g] (a2) [right=of a1, label=below:$\alpha_2$] {}
edge [] (a1);
\node[g] (a3) [right=of a2, label=below:$\alpha_3$] {};
\node[g] (a4) [right=of a3, label=below:$\alpha_4$, label=above:\tiny\textbf{F I}] {}
edge [] (a3);
\node[g] (a0) [right=of a4, label=below:$\alpha_0$] {}
edge [] (a4);
\draw ($(a2)!.35!(a3)$) -- ($(a2)!.35!(a3)+(0.3,0.2)$);
\draw ($(a2)!.35!(a3)$) -- ($(a2)!.35!(a3)+(0.3,-0.2)$);
\draw [] (a2.north east) to (a3.north west);
\draw [] (a2.south east) to (a3.south west);
\end{tikzpicture}
}
& \multirow{2}*{$(2432)$} & \Small\textbf{F I}  & $(2431)$ & $1$
\\ \cline{4-6} & & &
\Small\textbf{F II}  & $(1321)$ & $-$ (rank one)
\\ \hline 
$G_2$ & 
\begin{tikzpicture}[node distance=\nodedistance, g/.style={circle,inner sep=\circleinnersep,draw},a0/.style={rectangle,inner sep=\rectangleinnersep,draw}]
\node[g] (a1) [label=below:$\alpha_1$] {};
\node[g] (a2) [right=of a1, label=below:$\alpha_2$, label=above:\tiny\textbf{G}] {}
edge []  (a1);
\node[g] (a0) [right=of a2, label=below:$\alpha_0$] {}
edge [] (a2);
\draw ($(a1)!.35!(a2)$) -- ($(a1)!.35!(a2)+(0.3,0.2)$);
\draw ($(a1)!.35!(a2)$) -- ($(a1)!.35!(a2)+(0.3,-0.2)$);
\draw [] (a1.north east) to (a2.north west);
\draw [] (a1.south east) to (a2.south west);
\end{tikzpicture}
& $(32)$ & \textbf{G} & $(31)$ & $1$
\\
\hline
\end{tabular}
\begin{tablenotes}
\vspace{0.8ex}
\item[*]\Small For each extended Dynkin diagram, we provide the maximal root $\mu$ and the associated symmetric spaces $G/K$ using Cartan's notation. For every such $G/K$, we show the corresponding maximal noncompact root $\lambda$ and the number $N(\cal{F}_{G/K})$ of noncongruent isoparametric foliations on the complex projective space induced by $\cal{F}_{G/K}$. Roots are specified in coordinates with respect to $\Pi_\g{g}$.
\end{tablenotes}
\end{threeparttable}
\end{table}

\begin{table}[ht]
\begin{threeparttable}
\caption{Lowest weight diagrams of FKM-foliations with $m_1\leq m_2$}
\label{table:FKM_diagrams}
\renewcommand{\arraystretch}{1.2}
\begin{tabular}{@{} l @{\quad} >{\Small} W @{\;\;}   >{\Small} l @{\qquad}  l @{}}
 $m$ & \normalsize Lowest weight diagram   &   {\normalsize $\g{h}$} & $N(\cal{F}_\cal{P})$ \\ \hline
 & 
\begin{tikzpicture}[node distance=\nodedistance,g/.style={circle,inner sep=\circleinnersep,draw},lw/.style={circle,inner sep=\circleinnersep,fill=black, draw}]
\node[g] (a1s)[label=below:$\alpha^s_1$] {};
\node[g] (a2s) [right=of a1s, label=below:$\alpha^s_2$] {}
edge [] (a1s);
\node[] (a3s) [right=of a2s] {};
\node[g] (ap2) [right=of a3s, label=below:$\alpha^s_{p-2}$] {};
\node[g] (ap1) [right=of ap2, label=below:$\alpha^s_{p-1}$] {}
edge [] (ap2);
\node[g] (ap) [right=of ap1, label=below:$\alpha^s_p$] {};
\draw ($(ap1)!.65!(ap)$) -- ($(ap1)!.65!(ap)+(-0.3,0.2)$);
\draw ($(ap1)!.65!(ap)$) -- ($(ap1)!.65!(ap)+(-0.3,-0.2)$);
\draw [] (ap1.north east) to (ap.north west);
\draw [] (ap1.south east) to (ap.south west);
\draw ($(a2s)!.5!(a3s)$) -- (a2s);
\draw ($(ap2)!.5!(a3s)$) -- (ap2);
\draw [dashed] ($(a2s)!.5!(a3s)$) -- ($(ap2)!.5!(a3s)$);
\node[lw] (lw+) at ($(ap)+(40:0.9)$) [label=below:$\lambda^+$] {}
edge [] (ap);
\node[g] (a1+) [right=of lw+, label=below:$\alpha^+_1$] {}
edge [] (lw+);
\node[] (a2+) [right=of a1+] {};
\node[g] (aq2+) [right=of a2+, label=below:$\alpha^+_{q_+-2}$] {};
\node[g] (aq1+) at ($(aq2+)+(-25:0.8)$) [label=right:$\alpha^+_{q_+-1}$] {}
edge [] (aq2+);
\node[g] (aq+) at ($(aq2+)+(25:0.8)$) [label=right:$\alpha^+_{q_+}$] {}
edge [] (aq2+);
\draw ($(a1+)!.5!(a2+)$) -- (a1+);
\draw ($(aq2+)!.5!(a2+)$) -- (aq2+);
\draw [dashed] ($(a1+)!.5!(a2+)$) -- ($(aq2+)!.5!(a2+)$);
\node[lw] (lw-) at ($(ap)+(-40:0.9)$) [label=below:$\lambda^-$] {}
edge [] (ap);
\node[g] (a1-) [right=of lw-, label=below:$\alpha^-_1$] {}
edge [] (lw-);
\node[] (a2-) [right=of a1-] {};
\node[g] (aq2-) [right=of a2-, label=below:$\alpha^-_{q_--2}$] {};
\node[g] (aq1-) at ($(aq2-)+(-25:0.8)$) [label=right:$\alpha^-_{q_- -1}$] {}
edge [] (aq2-);
\node[g] (aq-) at ($(aq2-)+(25:0.8)$) [label=right:$\alpha^-_{q_-}$] {}
edge [] (aq2-);
\draw ($(a1-)!.5!(a2-)$) -- (a1-);
\draw ($(aq2-)!.5!(a2-)$) -- (aq2-);
\draw [dashed] ($(a1-)!.5!(a2-)$) -- ($(aq2-)!.5!(a2-)$);
\end{tikzpicture}
 & $\begin{array}{@{}l@{}}\g{so}(k_+)\oplus\g{so}(k_-)\\ k_\pm=2q_\pm\end{array}$ & $2$
\\  \cline{2-4} $0$ & 
\begin{tikzpicture}[node distance=\nodedistance,g/.style={circle,inner sep=\circleinnersep,draw},lw/.style={circle,inner sep=\circleinnersep,fill=black, draw}]
\node[g] (a1s)[label=below:$\alpha^s_1$] {};
\node[g] (a2s) [right=of a1s, label=below:$\alpha^s_2$] {}
edge [] (a1s);
\node[] (a3s) [right=of a2s] {};
\node[g] (ap2) [right=of a3s, label=below:$\alpha^s_{p-2}$] {};
\node[g] (ap1) [right=of ap2, label=below:$\alpha^s_{p-1}$] {}
edge [] (ap2);
\node[g] (ap) [right=of ap1, label=below:$\alpha^s_p$] {};
\draw ($(ap1)!.65!(ap)$) -- ($(ap1)!.65!(ap)+(-0.3,0.2)$);
\draw ($(ap1)!.65!(ap)$) -- ($(ap1)!.65!(ap)+(-0.3,-0.2)$);
\draw [] (ap1.north east) to (ap.north west);
\draw [] (ap1.south east) to (ap.south west);
\draw ($(a2s)!.5!(a3s)$) -- (a2s);
\draw ($(ap2)!.5!(a3s)$) -- (ap2);
\draw [dashed] ($(a2s)!.5!(a3s)$) -- ($(ap2)!.5!(a3s)$);
\node[lw] (lw+) at ($(ap)+(40:0.8)$) [label=above:$\lambda^+$] {}
edge [] (ap);
\node[g] (a1+) [right=of lw+, label=above:$\alpha^+_1$] {}
edge [] (lw+);
\node[] (a2+) [right=of a1+] {};
\node[g] (aq2+) [right=of a2+, label=above:$\alpha^+_{q_+-2}$] {};
\node[g] (aq1+) at ($(aq2+)+(-25:0.8)$) [label=right:$\alpha^+_{q_+-1}$] {}
edge [] (aq2+);
\node[g] (aq+) at ($(aq2+)+(25:0.8)$) [label=right:$\alpha^+_{q_+}$] {}
edge [] (aq2+);
\draw ($(a1+)!.5!(a2+)$) -- (a1+);
\draw ($(aq2+)!.5!(a2+)$) -- (aq2+);
\draw [dashed] ($(a1+)!.5!(a2+)$) -- ($(aq2+)!.5!(a2+)$);
\node[lw] (lw-) at ($(ap)+(-40:0.8)$) [label=above:$\lambda^-$] {}
edge [] (ap);
\node[g] (a1-) [right=of lw-, label=above:$\alpha^-_1$] {}
edge [] (lw-);
\node[] (a2-) [right=of a1-] {};
\node[g] (aq1-) [right=of a2-, label=above:$\alpha^-_{q_--1}$] {};
\node[g] (aq-) [right=of aq1-, label=right:$\alpha^-_{q_-}$] {};
\draw ($(aq1-)!.65!(aq-)$) -- ($(aq1-)!.65!(aq-)+(-0.3,0.2)$);
\draw ($(aq1-)!.65!(aq-)$) -- ($(aq1-)!.65!(aq-)+(-0.3,-0.2)$);
\draw [] (aq1-.north east) to (aq-.north west);
\draw [] (aq1-.south east) to (aq-.south west);
\draw ($(a1-)!.5!(a2-)$) -- (a1-);
\draw ($(aq1-)!.5!(a2-)$) -- (aq1-);
\draw [dashed] ($(a1-)!.5!(a2-)$) -- ($(aq1-)!.5!(a2-)$);
\end{tikzpicture}
& $\begin{array}{@{}l@{}}\g{so}(k_+)\oplus\g{so}(k_-)\\ k_+=2q_+,\\ k_-=2q_-+1\end{array}$ & $1$
\\ \cline{2-4} &  
\begin{tikzpicture}[node distance=\nodedistance,g/.style={circle,inner sep=\circleinnersep,draw},lw/.style={circle,inner sep=\circleinnersep,fill=black, draw}]
\node[g] (a1s)[label=below:$\alpha^s_1$] {};
\node[g] (a2s) [right=of a1s, label=below:$\alpha^s_2$] {}
edge [] (a1s);
\node[] (a3s) [right=of a2s] {};
\node[g] (ap2) [right=of a3s, label=below:$\alpha^s_{p-2}$] {};
\node[g] (ap1) [right=of ap2, label=below:$\alpha^s_{p-1}$] {}
edge [] (ap2);
\node[g] (ap) [right=of ap1, label=below:$\alpha^s_p$] {};
\draw ($(ap1)!.65!(ap)$) -- ($(ap1)!.65!(ap)+(-0.3,0.2)$);
\draw ($(ap1)!.65!(ap)$) -- ($(ap1)!.65!(ap)+(-0.3,-0.2)$);
\draw [] (ap1.north east) to (ap.north west);
\draw [] (ap1.south east) to (ap.south west);
\draw ($(a2s)!.5!(a3s)$) -- (a2s);
\draw ($(ap2)!.5!(a3s)$) -- (ap2);
\draw [dashed] ($(a2s)!.5!(a3s)$) -- ($(ap2)!.5!(a3s)$);
\node[lw] (lw+) at ($(ap)+(35:0.8)$) [label=above:$\lambda^+$] {}
edge [] (ap);
\node[g] (a1+) [right=of lw+, label=above:$\alpha^+_1$] {}
edge [] (lw+);
\node[] (a2+) [right=of a1+] {};
\node[g] (aq1+) [right=of a2+, label=above:$\alpha^+_{q_+-1}$] {};
\node[g] (aq+) [right=of aq1+, label=right:$\alpha^+_{q_+}$] {};
\draw ($(aq1+)!.65!(aq+)$) -- ($(aq1+)!.65!(aq+)+(-0.3,0.2)$);
\draw ($(aq1+)!.65!(aq+)$) -- ($(aq1+)!.65!(aq+)+(-0.3,-0.2)$);
\draw [] (aq1+.north east) to (aq+.north west);
\draw [] (aq1+.south east) to (aq+.south west);
\draw ($(a1+)!.5!(a2+)$) -- (a1+);
\draw ($(aq1+)!.5!(a2+)$) -- (aq1+);
\draw [dashed] ($(a1+)!.5!(a2+)$) -- ($(aq1+)!.5!(a2+)$);
\node[lw] (lw-) at ($(ap)+(-35:0.8)$) [label=above:$\lambda^-$] {}
edge [] (ap);
\node[g] (a1-) [right=of lw-, label=above:$\alpha^-_1$] {}
edge [] (lw-);
\node[] (a2-) [right=of a1-] {};
\node[g] (aq1-) [right=of a2-, label=above:$\alpha^-_{q_--1}$] {};
\node[g] (aq-) [right=of aq1-, label=right:$\alpha^-_{q_-}$] {};
\draw ($(aq1-)!.65!(aq-)$) -- ($(aq1-)!.65!(aq-)+(-0.3,0.2)$);
\draw ($(aq1-)!.65!(aq-)$) -- ($(aq1-)!.65!(aq-)+(-0.3,-0.2)$);
\draw [] (aq1-.north east) to (aq-.north west);
\draw [] (aq1-.south east) to (aq-.south west);
\draw ($(a1-)!.5!(a2-)$) -- (a1-);
\draw ($(aq1-)!.5!(a2-)$) -- (aq1-);
\draw [dashed] ($(a1-)!.5!(a2-)$) -- ($(aq1-)!.5!(a2-)$);
\end{tikzpicture}
 & $\begin{array}{@{}l@{}}\g{so}(k_+)\oplus\g{so}(k_-)\\k_\pm=2q_\pm+1\end{array}$  & $1$
\\ \hline
\multirow{2}*[-4ex]{$1,7$} &  \vspace{\upperlinespace}
\begin{tikzpicture}[node distance=\nodedistance,g/.style={circle,inner sep=\circleinnersep,draw},lw/.style={circle,inner sep=\circleinnersep,fill=black, draw}]
\node[g] (a1s)[label=below:$\alpha^s_1$] {};
\node[g] (a2s) [right=of a1s, label=below:$\alpha^s_2$] {}
edge [] (a1s);
\node[] (a3s) [right=of a2s] {};
\node[g] (ap3) [right=of a3s, label=below:$\alpha^s_{p-3}$] {};
\node[g] (ap2) [right=of ap3, label=below:$\alpha^s_{p-2}$] {}
edge [] (ap3);
\node[g] (ap1) at ($(ap2)+(-35:0.8)$) [label=below:$\alpha^s_{p-1}$] {}
edge [] (ap2);
\node[g] (ap) at ($(ap2)+(35:0.8)$) [label=below:$\alpha^s_{p}$] {}
edge [] (ap2);
\draw ($(a2s)!.5!(a3s)$) -- (a2s);
\draw ($(ap3)!.5!(a3s)$) -- (ap3);
\draw [dashed] ($(a2s)!.5!(a3s)$) -- ($(ap2)!.5!(a3s)$);
\node[lw] (lw+) [right=of ap, label=below:$\lambda^+$] {}
edge [] (ap);
\node[lw] (lw-) [right=of ap1, label=below:$\lambda^-$] {}
edge [] (ap1);
\node[g] (a1) at ($(lw+)+(-35:0.8)$) [label=below:$\alpha_1$] {}
edge [] (lw+)
edge [] (lw-);
\node[g] (a2) [right=of a1, label=below:$\alpha_2$] {}
edge [] (a1);
\node[] (a3) [right=of a2] {};
\node[g] (aq2) [right=of a3, label=below:$\alpha_{q-2}$] {};
\node[g] (aq1) at ($(aq2)+(-25:0.8)$) [label=below:$\alpha_{q-1}$] {}
edge [] (aq2);
\node[g] (aq) at ($(aq2)+(25:0.8)$) [label=below:$\alpha_{q}$] {}
edge [] (aq2);
\draw ($(a2)!.5!(a3)$) -- (a2);
\draw ($(aq2)!.5!(a3)$) -- (aq2);
\draw [dashed] ($(a2)!.5!(a3)$) -- ($(aq2)!.5!(a3)$);
\end{tikzpicture}
 & $\begin{array}{@{}l@{}}\g{so}(k)\\ k=2q\end{array}$ & $2$
\\ \cline{2-4} &  \vspace{\upperlinespace}
\begin{tikzpicture}[node distance=\nodedistance,g/.style={circle,inner sep=\circleinnersep,draw},lw/.style={circle,inner sep=\circleinnersep,fill=black, draw}]
\node[g] (a1s)[label=below:$\alpha^s_1$] {};
\node[g] (a2s) [right=of a1s, label=below:$\alpha^s_2$] {}
edge [] (a1s);
\node[] (a3s) [right=of a2s] {};
\node[g] (ap3) [right=of a3s, label=below:$\alpha^s_{p-3}$] {};
\node[g] (ap2) [right=of ap3, label=below:$\alpha^s_{p-2}$] {}
edge [] (ap3);
\node[g] (ap1) at ($(ap2)+(-35:0.8)$) [label=below:$\alpha^s_{p-1}$] {}
edge [] (ap2);
\node[g] (ap) at ($(ap2)+(35:0.8)$) [label=below:$\alpha^s_{p}$] {}
edge [] (ap2);
\draw ($(a2s)!.5!(a3s)$) -- (a2s);
\draw ($(ap3)!.5!(a3s)$) -- (ap3);
\draw [dashed] ($(a2s)!.5!(a3s)$) -- ($(ap2)!.5!(a3s)$);
\node[lw] (lw+) [right=of ap, label=below:$\lambda^+$] {}
edge [] (ap);
\node[lw] (lw-) [right=of ap1, label=below:$\lambda^-$] {}
edge [] (ap1);
\node[g] (a1) at ($(lw+)+(-35:0.8)$) [label=below:$\alpha_1$] {}
edge [] (lw+)
edge [] (lw-);
\node[g] (a2) [right=of a1, label=below:$\alpha_2$] {}
edge [] (a1);
\node[] (a3) [right=of a2] {};
\node[g] (aq1) [right=of a3, label=below:$\alpha_{q-1}$] {};
\node[g] (aq) [right=of aq1, label=below:$\alpha_{q}$] {};
\draw ($(aq1)!.65!(aq)$) -- ($(aq1)!.65!(aq)+(-0.3,0.2)$);
\draw ($(aq1)!.65!(aq)$) -- ($(aq1)!.65!(aq)+(-0.3,-0.2)$);
\draw [] (aq1.north east) to (aq.north west);
\draw [] (aq1.south east) to (aq.south west);
\draw ($(a2)!.5!(a3)$) -- (a2);
\draw ($(aq1)!.5!(a3)$) -- (aq2);
\draw [dashed] ($(a2)!.5!(a3)$) -- ($(aq1)!.5!(a3)$);
\end{tikzpicture}
 &  $\begin{array}{@{}l@{}}\g{so}(k)\\ k=2q+1\end{array}$ & $1$
\\ \hline
$2,6$ &  \vspace{\upperlinespace}
\begin{tikzpicture}[node distance=\nodedistance,g/.style={circle,inner sep=\circleinnersep,draw},lw/.style={circle,inner sep=\circleinnersep,fill=black, draw}]
\node[g] (a1s)[label=below:$\alpha^s_1$] {};
\node[g] (a2s) [right=of a1s, label=below:$\alpha^s_2$] {}
edge [] (a1s);
\node[] (a3s) [right=of a2s] {};
\node[g] (ap2) [right=of a3s, label=below:$\alpha^s_{p-2}$] {};
\node[g] (ap1) [right=of ap2, label=below:$\alpha^s_{p-1}$] {}
edge [] (ap2);
\node[g] (ap) [right=of ap1, label=below:$\alpha^s_p$] {};
\draw ($(ap1)!.65!(ap)$) -- ($(ap1)!.65!(ap)+(-0.3,0.2)$);
\draw ($(ap1)!.65!(ap)$) -- ($(ap1)!.65!(ap)+(-0.3,-0.2)$);
\draw [] (ap1.north east) to (ap.north west);
\draw [] (ap1.south east) to (ap.south west);
\draw ($(a2s)!.5!(a3s)$) -- (a2s);
\draw ($(ap2)!.5!(a3s)$) -- (ap2);
\draw [dashed] ($(a2s)!.5!(a3s)$) -- ($(ap2)!.5!(a3s)$);
\node[lw] (lw+) at ($(ap)+(40:0.9)$) [label=below:$\lambda^+$] {}
edge [] (ap);
\node[g] (a1) [right=of lw+, label=below:$\alpha_1$] {}
edge [] (lw+);
\node[g] (a2) [right=of a1, label=below:$\alpha_2$] {}
edge [] (a1);
\node[] (a3) [right=of a2] {}
edge [] (a2);
\node[lw] (lw-) at ($(ap)+(-40:0.9)$) [label=above:$\lambda^-$] {}
edge [] (ap);
\node[g] (aq1) [right=of lw-, label=above:$\alpha_{q-1}$] {}
edge [] (lw-);
\node[g] (aq2) [right=of aq1, label=above:$\alpha_{q-2}$] {}
edge [] (aq1);
\node[] (aq3) [right=of aq2] {}
edge [] (aq2);
\draw [dashed] ($(aq3)+(0:-0.1)$) -- ($(a3)+(0:-0.1)$);
\end{tikzpicture}
 &  $\begin{array}{@{}l@{}}\g{u}(k)\\ k=q\end{array}$ & $2+\left[\frac{k}{2}\right]$
\\ \hline
$3,5$ &  \vspace{\upperlinespace}
\begin{tikzpicture}[node distance=\nodedistance,g/.style={circle,inner sep=\circleinnersep,draw},lw/.style={circle,inner sep=\circleinnersep,fill=black, draw}]
\node[g] (a1s)[label=below:$\alpha^s_1$] {};
\node[g] (a2s) [right=of a1s, label=below:$\alpha^s_2$] {}
edge [] (a1s);
\node[] (a3s) [right=of a2s] {};
\node[g] (ap3) [right=of a3s, label=below:$\alpha^s_{p-3}$] {};
\node[g] (ap2) [right=of ap3, label=below:$\alpha^s_{p-2}$] {}
edge [] (ap3);
\node[g] (ap1) at ($(ap2)+(-35:0.8)$) [label=below:$\alpha^s_{p-1}$] {}
edge [] (ap2);
\node[g] (ap) at ($(ap2)+(35:0.8)$) [label=below:$\alpha^s_{p}$] {}
edge [] (ap2);
\draw ($(a2s)!.5!(a3s)$) -- (a2s);
\draw ($(ap3)!.5!(a3s)$) -- (ap3);
\draw [dashed] ($(a2s)!.5!(a3s)$) -- ($(ap2)!.5!(a3s)$);
\node[lw] (lw+) [right=of ap, label=below:$\lambda^+$] {}
edge [] (ap);
\node[lw] (lw-) [right=of ap1, label=below:$\lambda^-$] {}
edge [] (ap1);
\node[g] (a1) at ($(lw+)+(-35:0.8)$) [label=below:$\alpha_1$] {}
edge [] (lw+)
edge [] (lw-);
\node[g] (a2) [right=of a1, label=below:$\alpha_2$] {}
edge [] (a1);
\node[] (a3) [right=of a2] {};
\node[g] (aq1) [right=of a3, label=below:$\alpha_{q-1}$] {};
\node[g] (aq) [right=of aq1, label=below:$\alpha_{q}$] {};
\draw ($(aq1)!.35!(aq)$) -- ($(aq1)!.35!(aq)-(-0.3,0.2)$);
\draw ($(aq1)!.35!(aq)$) -- ($(aq1)!.35!(aq)-(-0.3,-0.2)$);
\draw [] (aq1.north east) to (aq.north west);
\draw [] (aq1.south east) to (aq.south west);
\draw ($(a2)!.5!(a3)$) -- (a2);
\draw ($(aq1)!.5!(a3)$) -- (aq1);
\draw [dashed] ($(a2)!.5!(a3)$) -- ($(aq1)!.5!(a3)$);
\end{tikzpicture} &  $\begin{array}{@{}l@{}}\g{sp}(k)\\ k=q\end{array}$ & $2$
\\ \hline
$4$ &  
\begin{tikzpicture}[node distance=\nodedistance,g/.style={circle,inner sep=\circleinnersep,draw},lw/.style={circle,inner sep=\circleinnersep,fill=black, draw}]
\node[g] (a1s)[label=below:$\alpha^s_1$] {};
\node[g] (a2s) [right=of a1s, label=below:$\alpha^s_2$] {}
edge [] (a1s);
\node[] (a3s) [right=of a2s] {};
\node[g] (ap2) [right=of a3s, label=below:$\alpha^s_{p-2}$] {};
\node[g] (ap1) [right=of ap2, label=below:$\alpha^s_{p-1}$] {}
edge [] (ap2);
\node[g] (ap) [right=of ap1, label=below:$\alpha^s_p$] {};
\draw ($(ap1)!.65!(ap)$) -- ($(ap1)!.65!(ap)+(-0.3,0.2)$);
\draw ($(ap1)!.65!(ap)$) -- ($(ap1)!.65!(ap)+(-0.3,-0.2)$);
\draw [] (ap1.north east) to (ap.north west);
\draw [] (ap1.south east) to (ap.south west);
\draw ($(a2s)!.5!(a3s)$) -- (a2s);
\draw ($(ap2)!.5!(a3s)$) -- (ap2);
\draw [dashed] ($(a2s)!.5!(a3s)$) -- ($(ap2)!.5!(a3s)$);
\node[lw] (lw+) at ($(ap)+(35:0.8)$) [label=above:$\lambda^+$] {}
edge [] (ap);
\node[g] (a1+) [right=of lw+, label=above:$\alpha^+_1$] {}
edge [] (lw+);
\node[] (a2+) [right=of a1+] {};
\node[g] (aq1+) [right=of a2+, label=above:$\alpha^+_{q_+-1}$] {};
\node[g] (aq+) [right=of aq1+, label=right:$\alpha^+_{q_+}$] {};
\draw ($(aq1+)!.35!(aq+)$) -- ($(aq1+)!.35!(aq+)-(-0.3,0.2)$);
\draw ($(aq1+)!.35!(aq+)$) -- ($(aq1+)!.35!(aq+)-(-0.3,-0.2)$);
\draw [] (aq1+.north east) to (aq+.north west);
\draw [] (aq1+.south east) to (aq+.south west);
\draw ($(a1+)!.5!(a2+)$) -- (a1+);
\draw ($(aq1+)!.5!(a2+)$) -- (aq1+);
\draw [dashed] ($(a1+)!.5!(a2+)$) -- ($(aq1+)!.5!(a2+)$);
\node[lw] (lw-) at ($(ap)+(-35:0.8)$) [label=above:$\lambda^-$] {}
edge [] (ap);
\node[g] (a1-) [right=of lw-, label=above:$\alpha^-_1$] {}
edge [] (lw-);
\node[] (a2-) [right=of a1-] {};
\node[g] (aq1-) [right=of a2-, label=above:$\alpha^-_{q_--1}$] {};
\node[g] (aq-) [right=of aq1-, label=right:$\alpha^-_{q_-}$] {};
\draw ($(aq1-)!.35!(aq-)$) -- ($(aq1-)!.35!(aq-)-(-0.3,0.2)$);
\draw ($(aq1-)!.35!(aq-)$) -- ($(aq1-)!.35!(aq-)-(-0.3,-0.2)$);
\draw [] (aq1-.north east) to (aq-.north west);
\draw [] (aq1-.south east) to (aq-.south west);
\draw ($(a1-)!.5!(a2-)$) -- (a1-);
\draw ($(aq1-)!.5!(a2-)$) -- (aq1-);
\draw [dashed] ($(a1-)!.5!(a2-)$) -- ($(aq1-)!.5!(a2-)$);
\end{tikzpicture} &  $\begin{array}{@{}l@{}}\g{sp}(k_+)\oplus\g{sp}(k_-)\\ k_\pm=q_\pm\end{array}$ & $2$
\\
\hline
\end{tabular}
\begin{tablenotes}
\vspace{0.8ex}
\item[*]\Small The following data are provided for each value of $m$ ($\mod 8$): the corresponding lowest weight diagrams (see Remark~\ref{rem:exceptional_diagrams} for exceptional cases with low $k$, $k_\pm$), the Lie algebra $\g{h}$ such that $\g{so}(m+1)\oplus\g{h}$ is the Lie algebra of $\Aut(\cal{F}_\cal{P})$, and the value of $N(\cal{F}_\cal{P})$.  
\end{tablenotes}
\end{threeparttable}
\end{table}

\end{document}